\documentclass[11pt]{article}

\usepackage{amssymb,amsmath,amscd,amsthm,xy,enumitem,mathabx,mathrsfs,color,caption,stmaryrd}
\xyoption{all}
\usepackage[pdftex]{graphicx}

\newtheorem{thm}{Theorem}[section]
\newtheorem{prp}[thm]{Proposition}
\newtheorem{lmm}[thm]{Lemma}

\newtheorem{crl}[thm]{Corollary}

\theoremstyle{definition}
\newtheorem{dfn}[thm]{Definition}
\newtheorem{eg}[thm]{Example}

\theoremstyle{remark}
\newtheorem{rmk}[thm]{Remark}

\numberwithin{equation}{section}

\def\lra{\longrightarrow}

\def\BE#1{\begin{equation}\label{#1}}
\def\EE{\end{equation}}
\def\lr#1{\langle#1\rangle}
\def\flr#1{\left\lfloor{#1}\right\rfloor}
\def\blr#1{\big\langle#1\big\rangle}

\def\wt#1{\widetilde{#1}}
\def\ov#1{\overline{#1}}
\def\eref#1{(\ref{#1})}
\def\tn#1{\textnormal{#1}}
\def\sf#1{\textsf{#1}}

\def\wh#1{\widehat{#1}}
\def\wch#1{\widecheck{#1}}

\def\De{\Delta}
\def\Ga{\Gamma}
\def\La{\Lambda}
\def\Om{\Omega}
\def\OmN{\Omega_N}
\def\Si{\Sigma}
\def\Th{\Theta}

\def\al{\alpha}

\def\de{\delta}

\def\ga{\gamma}
\def\io{\iota}

\def\na{\nabla}
\def\om{\omega}
\def\si{\sigma}
\def\th{\theta}
\def\ve{\varepsilon}
\def\vph{\varphi}
\def\vp{\varpi}

\def\ze{\zeta}

\def\fB{\mathfrak B}

\def\C{\mathbb C}
\def\cC{\mathcal C}  
\def\cD{\mathcal{D}}

\def\cE{\mathcal E}
\def\E{\tn{E}}
\def\ne{\textnormal{e}}

\def\cH{\mathcal H}

\def\bI{\mathbb I}
\def\cI{\mathcal I}
\def\fI{\mathfrak i}
\def\cJ{\mathcal J}
\def\fJ{\mathfrak j}
\def\cK{\mathcal K}

\def\cL{\mathcal L}

\def\cM{\mathcal M}
\def\fM{\mathfrak M}

\def\cN{\mathcal N}
\def\cO{\mathcal O}
\def\P{\mathbb P}
\def\Q{\mathbb Q}
\def\R{\mathbb{R}}

\def\cR{\mathcal{R}}
\def\fR{\mathfrak R}

\def\cT{\mathcal T}

\def\cU{\mathcal U}

\def\Z{\mathbb{Z}}

\def\a{\mathbf a}
\def\fa{\mathfrak a}

\def\fc{\mathfrak c}

\def\ff{\mathfrak f}

\def\fs{\mathfrak s}
\def\ft{\mathfrak t}
\def\t{\mathbf t}
\def\u{\mathbf u}

\def\can{\tn{can}}

\def\cok{\textnormal{cok}}

\def\tnd{\textnormal{d}}

\def\ev{\tn{ev}}

\def\GL{\tn{GL}}
\def\Hom{\tn{Hom}}
\def\id{\textnormal{id}}
\def\Im{\tn{Im}}
\def\Id{\tn{Id}}
\def\ind{\textnormal{ind}}

\def\pt{\tn{pt}}

\def\rk{\textnormal{rk}}
\def\rdet{\wh{\tn{det}}}

\def\SU{\tn{SU}}

\def\supp{\tn{supp}}
\def\SL{\tn{SL}}

\def\top{\textnormal{top}}
\def\vrt{\tn{vrt}}

\def\0{\mathbf 0}
\def\1{\mathbf 1}

\def\dbar{\bar\partial}
\def\prt{\partial}
\def\eset{\emptyset}
\def\i{\infty}
\def\bp{\bar\partial}
\def\w{\wedge}

\def\bu{\bullet}

\begin{document}

\title{Real Gromov-Witten Theory in All Genera and\\ 
Real Enumerative Geometry: Construction}
\author{Penka Georgieva\thanks{Partially supported by ERC grant STEIN-259118} ~and 
Aleksey Zinger\thanks{Partially supported by NSF grants  DMS 0846978 and 1500875}}
\date{\today}
\maketitle

\begin{abstract}
\noindent
We construct positive-genus analogues of Welschinger's invariants for 
many real symplectic manifolds,
including the odd-dimensional projective spaces and the renowned quintic threefold.
In some cases, our invariants provide lower bounds for counts of real positive-genus curves
in real algebraic varieties.
Our approach to the orientability problem is based entirely on the topology of real 
bundle pairs over symmetric surfaces;
the previous attempts involved direct computations for 
the determinant lines of Fredholm operators over bordered surfaces.
We use the notion of real orientation introduced in this paper to obtain 
isomorphisms of real bundle pairs over families of symmetric surfaces and 
then apply the determinant functor to these isomorphisms.
This allows us to endow the uncompactified moduli spaces of real maps 
from symmetric surfaces of all topological types with natural orientations 
and to verify that they extend 
across the codimension-one boundaries of these spaces,
thus implementing a far-reaching proposal from C.-C.~Liu's thesis
for a fully fledged real Gromov-Witten theory. 
The second and third parts of this work concern applications:
they describe important properties of our orientations on the moduli spaces,
establish some connections with real enumerative geometry,
provide the relevant equivariant localization data for projective spaces,
and obtain vanishing results in the spirit of Walcher's predictions.
\end{abstract}

\tableofcontents

\section{Introduction}
\label{intro_sec}

\noindent
The theory of $J$-holomorphic maps plays prominent roles in symplectic topology,
algebraic geometry, and string theory.
The foundational work of~\cite{Gr,McSa94,RT,LT,FO} has 
established the theory of (closed) Gromov-Witten invariants,
i.e.~counts of $J$-holomorphic maps from closed Riemann surfaces to symplectic manifolds.
In contrast, the theory of real Gromov-Witten invariants, 
i.e.~counts of $J$-holomorphic maps from symmetric Riemann surfaces commuting 
with the involutions on the domain and the target,
is still in early stages of development, especially in positive genera. 
The two main obstacles to defining real Gromov-Witten invariants are the potential 
non-orientability of the moduli space of real $J$-holomorphic maps and 
the existence of real codimension-one boundary strata.\\

\noindent
In this paper, we introduce the notion of \textsf{real orientation} on 
a real symplectic $2n$-manifold $(X,\om,\phi)$; 
see Definitions~\ref{realorient_dfn} and~\ref{realorient_dfn2}.
We overcome the first obstacle by
showing that a real orientation induces orientations on the uncompactified moduli spaces 
of real maps for {\it all} genera of and for {\it all} types of 
involutions~$\si$ on the domain if $n$ is odd; see Theorem~\ref{orient_thm}.
We then show that these orientations do not change across the codimension-one boundary strata
after they are reversed for half of the involution types in each genus.
This allows us to overcome the second obstacle by gluing the moduli spaces
for different types of involutions along their common boundaries;
this realizes an aspiration going back to~\cite{Melissa}.
We thus obtain real Gromov-Witten invariants of arbitrary genus 
for many real symplectic manifolds; see Theorems~\ref{main_thm} and~\ref{dim3_thm}.
Many projective complete intersections, including the quintic threefold
which plays a central role in Gromov-Witten theory and string theory, are among these manifolds; 
see Proposition~\ref{CIorient_prp}.
These invariants can be used to obtain lower bounds for counts of real positive-genus curves
in real algebraic varieties; see Proposition~\ref{g1EG_prp}.
For example, we find that there are at least 4~real genus~1 degree~6 irreducible curves 
passing  through a generic collection of 6~pairs of conjugate points in~$\P^3$.

\subsection{Terminology and setup}
\label{RealGWth_subs}

\noindent
An \textsf{involution} on a smooth manifold~$X$ is a diffeomorphism
$\phi\!:X\!\lra\!X$ such that $\phi\!\circ\!\phi\!=\!\id_X$.
Let
$$X^\phi=\big\{x\!\in\!X\!:~\phi(x)\!=\!x\big\}$$
denote the fixed locus.
An \sf{anti-symplectic involution~$\phi$} on a symplectic manifold $(X,\om)$
is an involution $\phi\!:X\!\lra\!X$ such that $\phi^*\om\!=\!-\om$.
For example, the~maps
\begin{alignat*}{2}
\tau_n\!:\P^{n-1}&\lra\P^{n-1}, &\qquad [Z_1,\ldots,Z_n]&\lra[\ov{Z}_1,\ldots,\ov{Z}_n],\\
\eta_{2m}\!: \P^{2m-1}& \lra\P^{2m-1},&\qquad
[Z_1,Z_2,\ldots,Z_{2m-1},Z_{2m}]&\lra 
\big[-\ov{Z}_2,\ov{Z}_1,\ldots,-\ov{Z}_{2m},\ov{Z}_{2m-1}\big],
\end{alignat*}
are anti-symplectic involutions with respect to the standard Fubini-Study symplectic
forms~$\om_n$ on~$\P^{n-1}$ and~$\om_{2m}$ on $\P^{2m-1}$, respectively.
If 
$$k\!\ge\!0, \qquad \a\equiv(a_1,\ldots,a_k)\in(\Z^+)^k\,,$$
and $X_{n;\a}\!\subset\!\P^{n-1}$ is a complete intersection of multi-degree~$\a$
preserved by~$\tau_n$,  then $\tau_{n;\a}\!\equiv\!\tau_n|_{X_{n;\a}}$
is an anti-symplectic involution on $X_{n;\a}$ with respect to the symplectic form
$\om_{n;\a}\!=\!\om_n|_{X_{n;\a}}$. 
Similarly, if $X_{2m;\a}\!\subset\!\P^{2m-1}$ is preserved by~$\eta_{2m}$, then
$\eta_{2m;\a}\!\equiv\!\eta_{2m}|_{X_{2m;\a}}$
is an anti-symplectic involution on $X_{2m;\a}$ with respect to the symplectic form
$\om_{2m;\a}\!=\!\om_{2m}|_{X_{2m;\a}}$.
A \sf{real symplectic manifold} is a triple $(X,\om,\phi)$ consisting 
of a symplectic manifold~$(X,\om)$ and an anti-symplectic involution~$\phi$.\\

\noindent
Let $(X,\phi)$ be a manifold with an involution.
A \sf{conjugation} on a complex vector bundle $V\!\lra\!X$ 
\sf{lifting} an involution~$\phi$ is a vector bundle homomorphism 
$\vph\!:V\!\lra\!V$ covering~$\phi$ (or equivalently 
a vector bundle homomorphism  $\vph\!:V\!\lra\!\phi^*V$ covering~$\id_X$)
such that the restriction of~$\vph$ to each fiber is anti-complex linear
and $\vph\!\circ\!\vph\!=\!\id_V$.
A \sf{real bundle pair} $(V,\vph)\!\lra\!(X,\phi)$   
consists of a complex vector bundle $V\!\lra\!X$ and 
a conjugation~$\vph$ on $V$ lifting~$\phi$.
For example, 
$$(TX,\tnd\phi)\lra(X,\phi) \qquad\hbox{and}\qquad
(X\!\times\!\C^n,\phi\!\times\!\fc)\lra(X,\phi),$$
where $\fc\!:\C^n\!\lra\!\C^n$ is the standard conjugation on~$\C^n$,
are real bundle pairs.
For any real bundle pair $(V,\vph)\!\lra\!(X,\phi)$, 
we denote~by
$$\La_{\C}^{\top}(V,\vph)=(\La_{\C}^{\top}V,\La_{\C}^{\top}\vph)$$
the top exterior power of $V$ over $\C$ with the induced conjugation.
Direct sums, duals, and tensor products over~$\C$ of real bundle pairs over~$(X,\phi)$
are again real bundle pairs over~$(X,\phi)$.\\

\noindent
A \sf{symmetric surface} $(\Si,\si)$ is a closed connected oriented smooth 
surface~$\Si$ (manifold of real dimension~2) with an orientation-reversing involution~$\si$.
The fixed locus of~$\si$ is a disjoint union of circles.
If in addition $(X,\phi)$ is a manifold with an involution, 
a \sf{real map} 
$$u\!:(\Si,\si)\lra(X,\phi)$$ 
is a smooth map $u\!:\Si\!\lra\!X$ such that $u\!\circ\!\si=\phi\!\circ\!u$.
We denote the space of such maps by~$\fB_g(X)^{\phi,\si}$,
with $g$ denoting the genus of the domain~$\Si$ of~$\si$.\\

\noindent
For a symplectic manifold $(X,\om)$, we denote~by $\cJ_{\om}$
the space of $\om$-compatible almost complex structures on~$X$.
If $\phi$ is  an anti-symplectic involution on~$(X,\om)$, let 
\BE{cJomdfn_e}\cJ_{\om}^{\phi}=\big\{J\!\in\!\cJ_{\om}\!:\,\phi^*J\!=\!-J\big\}.\EE
For a genus~$g$ symmetric surface~$(\Si,\si)$, we similarly denote by $\cJ_{\Si}^{\si}$
the space of complex structures~$\fJ$ on~$\Si$ compatible with the orientation such that 
$\si^*\fJ\!=\!-\fJ$.
For $J\!\in\!\cJ_{\om}^{\phi}$, $\fJ\!\in\!\cJ_{\Si}^{\si}$, and
$u\!\in\!\fB_g(X)^{\phi,\si}$, let 
$$\dbar_{J,\fJ}u=\frac{1}{2}\big(\tnd u+J\circ\tnd u\!\circ\!\fJ\big)$$
be the \textsf{$\dbar_J$-operator} on~$\fB_g(X)^{\phi,\si}$.\\

\noindent
Let $g,l\!\in\!\Z^{\ge0}$,  $(\Si,\si)$ be a genus~$g$ symmetric surface,
$B\!\in\!H_2(X;\Z)\!-\!0$, and $J\!\in\!\cJ_{\om}^{\phi}$.
Let $\De^{2l}\!\subset\!\Si^{2l}$ be the big diagonal, i.e.~the subset of $2l$-tuples
with at least two coordinates equal.
Denote~by
\begin{equation*}\begin{split}
\fM_{g,l}(X,B;J)^{\phi,\si}= 
\big\{(u,(z_1^+,z_1^-),\ldots,(z_l^+,z_l^-),\fJ)\!\in\!
\fB_g(X)^{\phi,\si}\!\times\!(\Si^{2l}\!-\!\De^{2l})\!\times\!\cJ_{\Si}^{\si}\!:\qquad&\\
z_i^-\!=\!\si(z_i^+)~\forall\,i\!=\!1,\ldots,l,~u_*[\Si]_{\Z}\!=\!B,~\dbar_{J,\fJ}
u\!=\!0\big\}\big/\!\!\sim&
\end{split}\end{equation*}
the (uncompactified)
moduli space of equivalence classes of degree~$B$ real $J$-holomorphic maps
from $(\Si,\si)$ to~$(X,\phi)$ with $l$~conjugate pairs of  marked points. 
Two marked $J$-holomorphic $(\phi,\si)$-real maps determine the same element of this 
moduli space if they differ by an orientation-preserving diffeomorphism of~$\Si$ 
commuting with~$\si$.
We denote~by 
\BE{notgluedspace_e}\ov\fM_{g,l}(X,B;J)^{\phi,\si}\supset\fM_{g,l}(X,B;J)^{\phi,\si}\EE
Gromov's convergence compactification of $\fM_{g,l}(X,B;J)^{\phi,\si}$ obtained
by including stable real maps from nodal symmetric surfaces.
The (virtually) codimension-one boundary strata~of
$$\ov\fM_{g,l}(X,B;J)^{\phi,\si}-\fM_{g,l}(X,B;J)^{\phi,\si}
\subset \ov\fM_{g,l}(X,B;J)^{\phi,\si}$$
consist of real $J$-holomorphic maps from one-nodal symmetric surfaces to~$(X,\phi)$.
Each stratum is either a (virtual) hypersurface in $\ov\fM_{g,l}(X,B;J)^{\phi,\si}$
or a (virtual) boundary of the spaces $\ov\fM_{g,l}(X,B;J)^{\phi,\si}$
for precisely two  topological types of orientation-reversing involutions~$\si$
on~$\Si$.
Let 
\BE{gluedspace_e}\fM_{g,l}(X,B;J)^{\phi}=\bigsqcup_{\si}\fM_{g,l}(X,B;J)^{\phi,\si}
\quad\hbox{and}\quad
\ov\fM_{g,l}(X,B;J)^{\phi}=\bigcup_{\si}\ov\fM_{g,l}(X,B;J)^{\phi,\si}\EE
denote the (disjoint) union of the uncompactified real moduli spaces 
and the union of the compactified real moduli spaces, respectively, 
taken over all topological types of orientation-reversing involutions~$\si$ on~$\Si$.\\

\noindent
Similarly to Example~\ref{ex_tbdl}, we denote~by 
$$\det\dbar_{\C}\lra \ov\fM_{g,l}(X,B;J)^{\phi}$$
the determinant line bundle of the standard real Cauchy-Riemann operator with values in~$(\C,\fc)$.
This real line bundle is not orientable if $X$ is a point and $g\!\ge\!1$.
It is not needed to formulate the main immediately applicable
results of this paper,
Theorems~\ref{main_thm} and~\ref{dim3_thm} below, but
is used in the overarching statement of Theorem~\ref{orient_thm}.

\subsection{Real orientations and real GW-invariants}
\label{MainTh_subs}

\noindent
We now introduce the notion of real orientation on a real symplectic manifold
and state the main theorems of this paper.

\begin{dfn}\label{realorient_dfn}
A real symplectic manifold $(X,\om,\phi)$ is \sf{real-orientable} if
there exists a rank~1 real bundle pair $(L,\wt\phi)$ over $(X,\phi)$ such~that 
\BE{realorient_e}w_2(TX^{\phi})=w_1(L^{\wt\phi})^2
\qquad\hbox{and}\qquad
\La_{\C}^{\top}(TX,\tnd\phi)\approx(L,\wt\phi)^{\otimes 2}\,.\EE
\end{dfn}

\begin{dfn}\label{realorient_dfn2}
A \sf{real orientation} on a real-orientable symplectic manifold $(X,\om,\phi)$ consists~of 
\begin{enumerate}[label=(RO\arabic*),leftmargin=*]

\item\label{LBP_it} a rank~1 real bundle pair $(L,\wt\phi)$ over $(X,\phi)$
satisfying~\eref{realorient_e},

\item\label{isom_it} a homotopy class~$[\psi]$ of isomorphisms 
of real bundle pairs in~\eref{realorient_e}, and

\item\label{spin_it} a spin structure~$\fs$ on the real vector bundle
$TX^{\phi}\!\oplus\!2(L^*)^{\wt\phi^*}$ over~$X^{\phi}$
compatible with the orientation induced by~\ref{isom_it}. 

\end{enumerate}
\end{dfn}

\begin{thm}\label{orient_thm} 
Let $(X,\om,\phi)$ be a real-orientable $2n$-manifold, 
$g,l\!\in\!\Z^{\ge0}$, 
\hbox{$B\!\in\!H_2(X;\Z)$}, and $J\!\in\!\cJ_{\om}^{\phi}$. 
Then a real orientation on $(X,\om,\phi)$ determines an orientation on the real line bundle
\BE{orient_thm_e}\La_{\R}^{\top}\big(T\ov\fM_{g,l}(X,B;J)^{\phi}\big)
\otimes \big(\!\det\dbar_{\C}\big)^{\otimes(n+1)}\lra 
\ov\fM_{g,l}(X,B;J)^{\phi}.\EE
In particular, the real moduli space $\ov\fM_{g,l}(X,B;J)^{\phi}$ is orientable 
if $n$ is odd.
\end{thm}

\noindent
A homotopy class of isomorphisms as in~\eref{realorient_e} determines an orientation
on $TX^{\phi}$ and thus on $TX^{\phi}\!\oplus\!2(L^*)^{\wt\phi^*}$;
see the paragraph after Definition~\ref{realorient_dfn4}.
In particular, Theorem~\ref{orient_thm} does not apply to any real symplectic manifold
$(X,\om,\phi)$ with unorientable Lagrangian~$X^{\phi}$.
By the first assumption in~\eref{realorient_e}, the real vector bundle
$TX^{\phi}\!\oplus\!2(L^*)^{\wt\phi^*}$ over~$X^{\phi}$ admits a spin structure.
Since $2(L^*)^{\wt\phi^*}\!\approx\!L^*|_{X^{\phi}}$, 
a real orientation on $(X,\om,\phi)$ includes a relative spin structure on $X^{\phi}\!\subset\!X$
in the sense of \cite[Definition~8.1.2]{FOOO}.\\

\noindent
The moduli space $\ov\fM_{g,l}(X,B;J)^{\phi}$ is not smooth in general.
Its tangent bundle  in~\eref{orient_thm_e} 
should be viewed in the usual moduli-theoretic (or virtual) sense, 
i.e.~as the index of suitably defined linearization of the $\dbar_J$-operator
(which includes deformations of the complex structure~$\fJ$ on~$\Si$).
The first statement of Theorem~\ref{orient_thm} and its proof also apply to Kuranishi charts 
for $\ov\fM_{g,l}(X,B;J)^{\phi}$
and the tangent spaces of the moduli spaces of real $(J,\nu)$-maps for 
generic local $\phi$-invariant deformations~$\nu$ of~\cite{RT2}.
A Kuranishi structure for $\ov\fM_{g,l}(X,B;J)^{\phi}$ is obtained by carrying out
the constructions of \cite{LT,FO} in a $\phi$-invariant manner;
see \cite[Section~7]{Sol} and \cite[Appendix]{FOOO9}.
Since the (virtual) boundary of $\ov\fM_{g,l}(X,B;J)^{\phi}$ is empty,
Theorem~\ref{orient_thm} implies that this moduli space carries a virtual fundamental class
in some cases and thus gives rise to real GW-invariants in arbitrary genus.

\begin{thm}\label{main_thm} 
Let $(X,\om,\phi)$ be a compact real-orientable $2n$-manifold with $n\!\not\in\!2\Z$,
$g,l\!\in\!\Z^{\ge0}$, \hbox{$B\!\in\!H_2(X;\Z)$}, and $J\!\in\!\cJ_{\om}^{\phi}$.
Then a real orientation on~$(X,\om,\phi)$ endows the moduli space\linebreak 
$\ov\fM_{g,l}(X,B;J)^{\phi}$ with a virtual fundamental class and  
thus gives rise to  genus~$g$ real GW-invariants of $(X,\om,\phi)$
that are independent of the choice of~$J\!\in\!\cJ_{\om}^{\phi}$.
\end{thm}

\noindent
The resulting real GW-invariants of~$(X,\om,\phi)$ in general depend on the choice
of real orientation.
This situation is analogous to the dependence on the choice of relative spin structure
often seen in open GW-theory.\\

\noindent
A notion of \textsf{semi-positive} for a real symplectic manifold 
$(X,\om,\phi)$ is introduced in \cite[Definition~1.2]{RealRT}.
Monotone symplectic manifolds with an anti-symplectic involution, 
including all projective spaces with the standard involutions and 
real Fano hypersurfaces of dimension at least~3, are semi-positive.
By \cite[Theorem~3.3]{RealRT}, the semi-positive property of \cite[Definition~1.2]{RealRT} 
plays the same role in real GW-theory as 
the semi-positive property of \cite[Definition~6.4.1]{MS}
plays in ``classical" GW-theory.
In particular, the real analogues of the geometric perturbations of~\cite{RT2} 
introduced in \cite[Section~3.1]{RealRT}
suffice to define the invariants of Theorem~\ref{main_thm}
with constraints pulled back from the target and the  Deligne-Mumford moduli space of real curves
for a semi-positive real symplectic manifold $(X,\om,\phi)$ endowed with a real orientation.
In these cases, the virtual tangent space of $\ov\fM_{g,l}(X,B;J)^{\phi}$ appearing 
in~\eref{orient_thm_e} can be replaced by the actual tangent space of the moduli space 
of simple real $(J,\nu)$-holomorphic maps from smooth and one-nodal symmetric surfaces
of genus~$g$.
The invariance of the resulting counts of such maps
can then be established by following along a path of auxiliary data;
it can pass only through one-nodal degenerations.\\

\noindent
Theorem~\ref{main_thm} yields  counts of real curves with conjugate pairs of insertions only.
By the last statement of \cite[Theorem~6.5]{Ge2},
the orientability of the Deligne-Mumford moduli space $\R\ov\cM_{g,l;k}$
of real genus~$g$ curves
with $l$ conjugate pairs of marked points and $k$ real marked points
does not capture the orientability of the analogous moduli space 
$\ov\fM_{g,l;k}(X,B;J)^{\phi}$ of real maps whenever $k\!>\!0$.
Theorem~\ref{orient_thm} remains valid for such moduli spaces outside of certain
``bad" codimension-one strata.
However, these strata are avoided by generic one-parameter families of real maps
in certain cases;
Theorem~\ref{orient_thm} then yields counts of real curves with conjugate pairs of insertions 
and real point insertions.

\begin{thm}\label{dim3_thm} 
Let $(X,\om,\phi)$ be a compact real-orientable 6-manifold 
such that $\lr{c_1(X),B}\!\in\!4\Z$ for all $B\!\in\!H_2(X;\Z)$
with $\phi_*B\!=\!-B$.
For all 
$$B \in H_2(X;\Z),\quad
 \mu_1,\ldots,\mu_l\!\in\!H^6(X;\Q)\!\cup\!H^2(X;\Q), 
\quad\hbox{and}\quad k\in\Z^{\ge0},$$
a real orientation on~$(X,\om,\phi)$ determines a signed count 
$$\blr{\mu_1,\ldots,\mu_l;\pt^k}_{1,B}^{\phi}\in\Q$$
of real $J$-holomorphic genus~1 degree~$B$ curves which is independent of the choice 
of~$J\in\!\cJ_{\om}^{\phi}$.
\end{thm}

\noindent
The $n\!=\!0$ case of Theorem~\ref{orient_thm} is essentially Proposition~\ref{DMext_prp}
which describes the orientability of the Deligne-Mumford moduli space
$\R\ov\cM_{g,l}$ of genus~$g$ symmetric surfaces with $l$~conjugate pairs of marked points.
If $n\!\in\!2\Z$ and $g\!+\!l\!\ge\!2$, Theorem~\ref{orient_thm} implies that 
a real orientation on~$(X,\om,\phi)$ induces an orientation on 
the real line bundle
\BE{twist_rmk_e}\La_{\R}^{\top}\big(T\ov\fM_{g,l}(X,B;J)^{\phi}\big)
\otimes \ff^*\big(\La_{\R}^{\top}(T\R\ov\cM_{g,l})\big)\lra  \ov\fM_{g,l}(X,B;J)^{\phi},\EE
where $\ff$ is the forgetful morphism~\eref{ffdfn_e}.
This orientation can be used to construct GW-invariants of~$(X,\om,\phi)$
with classes twisted by the orientation system of~$\R\ov\cM_{g,l}$,
as done in~\cite{Ge2} in the $g\!=\!0$ case.

\subsection{Previous results and acknowledgments}
\label{introend_subs}

\noindent
Invariant signed counts of real genus~0 curves with point constraints in real symplectic 4-manifolds 
and in many real symplectic 6-manifolds are defined in~\cite{Wel4,Wel6}.
An approach to interpreting these counts in the style of Gromov-Witten theory, 
i.e.~as counts of parametrizations of such curves, is presented in~\cite{Cho,Sol}. 
Signed counts of real genus~0 curves with conjugate pairs of arbitrary (not necessarily point)
constraints in arbitrary dimensions are defined in~\cite{Ge2}. 
All of these invariants involve morphisms from $\P^1$ with the standard involution 
$\tau\!\equiv\!\tau_2$ only and are constructed under the assumption that 
the fixed circle cannot shrink in a limit;
thus, only the degenerations of type~(H3) in Section~\ref{bnd_subs} are relevant in this case.
This assumption is dropped in~\cite{Teh} by combining  counts of $(\P^1,\tau)$-morphisms
with counts of $(\P^1,\eta)$-morphisms for the fixed-point-free involution $\eta\!\equiv\!\eta_2$
on~$\P^1$ and thus also considering  the degenerations of type~(E).
As the degenerations of types~(H1) and~(H2) do not appear in genus~0,
\cite{Teh} thus implements the genus~0 case of an aspiration raised in~\cite{Melissa}
and   elucidated in \cite[Section~1.5]{PSW}.
The target manifolds considered in \cite{Teh} are real-orientable in the sense
of Definition~\ref{realorient_dfn} and have spin fixed locus.\\

\noindent
We would like to thank E.~Brugall\'e, R.~Cr\'etois, E.~Ionel, S.~Lisi,
C.-C.~Liu, J.~Solomon, J.~Starr, M.~Tehrani, G.~Tian, and J.~Welschinger for related discussions.
We would also like to thank a referee for very thorough comments on
a previous version of this paper which led to corrections of a number of misstatements
and to other improvements in the exposition.
The second author is very grateful to the IAS School of Mathematics for its hospitality 
during the initial stages of our project on real GW-theory.

\section{Examples, properties, and applications}
\label{examples_sec}

\noindent
We begin this section with examples of distinct collections of real-orientable 
symplectic manifolds.
We then describe a number of properties of the real GW-invariants of 
Theorems~\ref{main_thm} and~\ref{dim3_thm},
including connections with real enumerative geometry
and compatibility with key morphisms of GW-theory.
With the exception of Proposition~\ref{vanGW_prp} and 
Corollaries~\ref{evenGWs_crl1} and~\ref{evenGWs_crl2}, 
the claims below are established in~\cite{RealGWsII,RealGWsIII}.

\begin{prp}\label{CIorient_prp}
Let $m,n\!\in\!\Z^+$, $k\!\in\!\Z^{\ge0}$, and $\a\!\equiv\!(a_1,\ldots,a_k)\!\in\!(\Z^+)^k$.
\begin{enumerate}[label=(\arabic*),leftmargin=*]
\item If $X_{n;\a}\!\subset\!\P^{n-1}$ is a complete intersection of multi-degree~$\a$
preserved by~$\tau_n$,
$$\sum_{i=1}^ka_i\equiv n \mod2, \qquad\hbox{and}\qquad
\sum_{i=1}^ka_i^2\equiv \sum_{i=1}^ka_i \mod4,$$
then $(X_{n;\a},\om_{n;\a},\tau_{n;\a})$ is a real-orientable symplectic manifold.

\item If $X_{2m;\a}\!\subset\!\P^{2m-1}$ is a complete intersection of 
multi-degree~$\a$ preserved by~$\eta_{2m}$ and
$$a_1\!+\!\ldots\!+\!a_k\equiv 2m \mod4,$$
then $(X_{2m;\a},\om_{2m;\a},\eta_{2m;\a})$ is a real-orientable symplectic manifold.
\end{enumerate}
\end{prp}

\begin{prp}\label{CYorient_prp}
Let $(X,\om,\phi)$ be a real symplectic manifold with $w_2(X^{\phi})\!=\!0$.
If 
\begin{enumerate}[label=(\arabic*),leftmargin=*]

\item $H_1(X;\Q)\!=\!0$ and $c_1(X)\!=\!2(\mu\!-\!\phi^*\mu)$ for some $\mu\!\in\!H^2(X;\Z)$ or

\item $X$ is  compact Kahler, $\phi$ is anti-holomorphic, and 
$\cK_X\!=\!2([D]\!+\![\ov{\phi_*D}])$ for some divisor $D$ on~$X$,

\end{enumerate}
then $(X,\om,\phi)$ is a real-orientable symplectic manifold.
\end{prp}

\noindent
Both of these propositions are established in~\cite{RealGWsIII}.
The first one is obtained by explicitly constructing suitable rank~1 real bundle pairs~$(L,\wt\phi)$,
while the second follows easily from 
the proof of \cite[Proposition~1.5]{Teh}.\\

\noindent
We recall that GW-invariants involving insertions only from the target~$X$
are called \sf{primitive}.
Such GW-invariants are related to counts of $J$-holomorphic {\it curves}
in~$X$ passing through a corresponding collection of constraints
(i.e.~of representatives for the Poincare duals of the insertions~used).
In contrast, GW-invariants also involving $\psi$-classes,
i.e.~the Chern classes of the universal tangent line bundles at the marked points,
are called \sf{descendant}. 
The next vanishing result extends \cite[Theorem~2.5]{RealEnum}.
Since the proof of the latter applies, we refer the reader to~\cite{RealEnum}.

\begin{prp}\label{vanGW_prp}
Let $(X,\om,\phi)$ be a compact real-orientable $2n$-manifold with $n\!\not\in\!2\Z$
and $g\!\in\!\Z^{\ge0}$.
The primary genus~$g$ real GW-invariants of $(X,\om,\phi)$ with conjugate pairs of constraints
that include an insertion $\mu\!\in\!H^*(X;\Q)$
such that $\phi^*\mu\!=\!\mu$ vanish.
\end{prp}

\noindent
The genus~$g$ real GW-invariants of $\P^{2n-1}$ with conjugate pairs of constraints 
can be computed using the virtual equivariant localization theorem of~\cite{GP}.
In the $g\!=\!1$ case, all torus fixed loci are contained in the smooth locus
of the moduli space and the classical equivariant localization theorem of~\cite{AB} suffices.
The relevant fixed loci data, which we describe in~\cite{RealGWsIII} based
on the properties of the orientations of Theorem~\ref{orient_thm} obtained in~\cite{RealGWsII},
is consistent with \cite[(3.22)]{Wal}.
We also obtain the two types of cancellations of contributions 
from some fixed loci predicted in \cite[Sections~3.2,3.3]{Wal}.
We use this data to obtain the following qualitative observations in~\cite{RealGWsIII};
they extend \cite[Theorem~1.10]{Teh} from the $g\!=\!0$ case and 
\cite[Theorem~7.2]{Teh2} from the $g\!=\!1$ case (the latter assuming that genus~1 real
GW-invariants can be defined).

\begin{prp}\label{PnEG_prp}
The genus~$g$ degree~$d$ real GW-invariants of $(\P^{2n-1},\om_{2n},\tau_{2n})$ and 
$(\P^{4n-1},\om_{4n},\eta_{4n})$ with only conjugate pairs of insertions  vanish 
if $d\!-\!g\!\in\!2\Z$.
The genus~$g$ real GW-invariants of $(\P^{4n-1},\om_{4n},\tau_{4n})$ and $(\P^{4n-1},\om_{4n},\eta_{4n})$
with only conjugate pairs of insertions differ by the factor of~$(-1)^{g-1}$.
\end{prp}

\noindent
The primary genus~$g$ real GW-invariants arising from Theorem~\ref{main_thm}  are 
in general combinations of counts of real curves of genus~$g$ and 
counts of real curves of lower genera and/or of lower degree (lower symplectic energy).
In light of \cite[Theorems~1A,1B]{g1diff} and \cite[Theorem~1.5]{FanoGV}, 
it seems plausible that the former can be extracted
from these GW-invariants to directly provide lower bounds for
enumerative counts of real curves in good situations.
This would typically involve delicate obstruction analysis.
However, the situation is fairly simple if $g\!=\!1$ and $n\!=\!3$.
 
\begin{prp}\label{g1EG_prp}
Let $(X,\om,\phi)$ be a compact real-orientable 6-manifold 
and $J\!\in\!\cJ_{\om}^{\phi}$ be an almost complex structure which is
genus~1 regular in the sense of \cite[Definition~1.4]{g1comp}.
The primary genus~1 real GW-invariants of $(X,\om,\phi)$ are then equal to 
the corresponding signed counts of real $J$-holomorphic curves
and thus provide lower bounds for the number of real genus~1 irreducible curves in~$(X,J,\phi)$. 
\end{prp}

\noindent
Since the standard complex structure~$J_0$ on~$\P^3$ is genus~1 regular,
the genus~1 real GW-invariants of $(\P^3,\om_4,\tau_4)$ and $(\P^3,\om_4,\eta_4)$
are lower bounds for the enumerative counts of such curves in 
$(\P^3,J_0,\tau_4)$ and $(\P^3,J_0,\eta_4)$, respectively.
The claim of Proposition~\ref{g1EG_prp} is particularly evident in
the case of real invariants of $(\P^3,J_0,\tau_4)$ and $(\P^3,J_0,\eta_4)$.
The only lower-genus contributions for the genus~1 GW-invariants of
6-dimensional symplectic manifolds can come from the genus~0 curves.
If $J$ is genus~1 regular, such contributions arise from the stratum of 
the moduli space consisting of morphisms with contracted genus~1 domain 
and a single effective bubble.
In the case of real morphisms, the node of the domain of such a map would have to be real.
There are no such morphisms in the case of $(\P^3,J_0,\eta_4)$ because
the real locus of $(\P^3,\eta_4)$ is empty.
In the case of  $(\P^3,J_0,\tau_4)$, the genus~0 contribution to the  genus~1 real GW-invariant
is a multiple of the  genus~0 real GW-invariant with the same insertions.
The genus~0  real GW-invariants of $(\P^3,J_0,\tau_4)$ are known to vanish in the even degrees;
see \cite[Remark~2.4(2)]{Wel6} and \cite[Theorem~1.10]{Teh}.
However, the substance of Proposition~\ref{g1EG_prp} is that the genus~0 real enumerative
counts do not contribute to the genus~1 real GW-invariants in all of the cases under consideration;
this is shown in~\cite{RealGWsIII}.
The situation in higher genus is described in~\cite{NZ}.\\

\noindent
From the equivariant localization data in~\cite{RealGWsIII}, 
we find that the genus~1 degree~$d$  real GW-invariant of $\P^3$ with $d$~pairs of 
conjugate point insertions is~0 for $d\!=\!2$, $-1$ for $d\!=\!4$,
and $-4$ for $d\!=\!6$.
The $d\!=\!2$ number is as expected, since there are no connected degree~2 curves 
of any kind passing through two generic pairs of conjugate points in~$\P^3$.
The $d\!=\!4$ number is also not surprising, since there is only one genus~1 degree~4
curve passing through 8~generic points in~$\P^3$; see the first three paragraphs
of \cite[Section~1]{Kollar}.
By \cite{growi}, the genus~0 and genus~1 degree~6 GW-invariants of~$\P^3$ with 12 point insertions
are  2576 and 1496/3, respectively.
By \cite[Theorem~1.1]{g1comp2}, this implies that the number  of genus~1 degree~6 curves
passing through 12~generic points in~$\P^3$ is~2860.
Our signed count of~$-4$ for the real genus~1 degree~6 curves
through 6~pairs of conjugate points in~$\P^3$ is thus consistent with the complex count
and provides a non-trivial lower bound for the number of real genus~1 degree~6 curves
with 6~pairs of conjugate point insertions.
Complete computations of the $d\!=\!2,4$ numbers and of the $d\!=\!6$ number
appear in~\cite{RealGWsIII} and~\cite{RealGWsApp,NZapp}, respectively.\\

\noindent
In all cases, the lower-genus contributions to the primary genus~$g$ real GW-invariants
arise from real curves passing through corresponding constraints.
If $n\!=\!3$, $\lr{c_1(X),B}\!\neq\!0$, and the almost complex structure $J\!\in\!\cJ_{\om}^{\phi}$
is sufficiently regular,
all such contributions arise from curves of the same degree.
Since the real enumerative counts are of the same parity as the complex enumerative counts,
Propositions~\ref{vanGW_prp}, \ref{PnEG_prp}, and~\ref{g1EG_prp} yield the following observations 
concerning the {\it complex} enumerative invariants 
\BE{EgB_e} \E_{g,B}\big(\mu_1,\phi^*\mu_1,\ldots,\mu_l,\phi^*\mu_l\big)\in \Z \EE
with $\mu_i\!\in\!H^*(X;\Z)$ 
that count genus~$g$ degree~$B$ $J$-holomorphic curves
passing through generic representatives of the Poincare duals of~$\mu_i$. 

\begin{crl}\label{evenGWs_crl1}
Let $(X,\om,\phi)$ be a real-orientable 6-manifold, 
$g,l\!\in\!\Z^{\ge0}$, and  $B\!\in\!H_2(X;\Z)$ with $\lr{c_1(X),B}\!\neq\!0$.
If $\phi^*\mu_i\!=\!\mu_i$ for some $i\!=\!1,\ldots,l$ and
$J\!\in\!\cJ_{\om}^{\phi}$ is sufficiently regular, then the number~\eref{EgB_e} is even.
\end{crl}

\begin{crl}\label{evenGWs_crl2}
Let $g,l,d\!\in\!\Z^{\ge0}$ with $d\!\ge\!2g\!-\!1$.
If either $\mu_i\!\in\!H^4(\P^3;\Z)$ for some $i\!=\!1,\ldots,l$ or 
$g\!=\!0,1$ and $g\!-\!d\!\in\!2\Z$, then
the genus~$g$ degree~$d$ enumerative invariants of~$\P^3$ of the form~\eref{EgB_e}
are even.
\end{crl}

\noindent 
The real GW-invariants arising from Theorems~\ref{main_thm} and~\ref{dim3_thm}
are compatible with standard morphisms of GW-theory,
such as the morphisms forgetting pairs of conjugate marked points and 
the node-identifying immersions~\eref{iodfn_e2} below.
By construction, the orientations on the real line bundles~\eref{orient_thm_e}
induced by a fixed real orientation on~$(X,\om,\phi)$
are preserved by the morphisms forgetting  pairs of conjugate marked points
(the fibers of these morphisms are canonically oriented).
If $n\!\not\in\!2\Z$,  this implies that the orientations on the moduli spaces of real morphisms
induced by a fixed real orientation on~$(X,\om,\phi)$ are preserved by 
the forgetful morphisms.
If $n\!\in\!2\Z$,  the orientations on the real line bundles~\eref{twist_rmk_e}
induced by a fixed real orientation on~$(X,\om,\phi)$  
are preserved by the forgetful morphisms.
In both cases, the orientations are compatible with the standard node-identifying 
immersions~\eref{iodfn_e2} below; see Proposition~\ref{CompOrient_prp}.
This in turn implies that a uniform system of these orientations 
is determined by a choice of orientation of the Deligne-Mumford moduli space
$\ov\cM_{0,2}^{\tau}\!\approx\![0,\i]$, 
where $\tau\!\equiv\!\tau_2$ is the standard conjugation on~$\P^1$,
and a real orientation on~$(X,\om,\phi)$.
If $g\!\not\in\!2\Z$, this also implies that the real GW-invariants of
$(\P^{2n-1},\om_{2n},\tau_{2n})$ and $(\P^{4n-1},\om_{4n},\eta_{4n})$
are independent of the choice of real orientation.\\

\noindent
Let $(X,\om,\phi)$, $l$, $B$, and~$J$  be as in Theorem~\ref{orient_thm} and $g\!\in\!\Z$.
We denote~by  $\ov\fM_{g,l}^{\bu}(X,B;J)^{\phi}$
the moduli space of stable real degree~$B$ morphisms from possibly disconnected 
nodal symmetric surfaces of Euler characteristic $2(1\!-\!g)$ with $l$ pairs 
of conjugate marked points.
For each $i\!=\!1,\ldots,l$, let 
$$\ev_i\!: \ov\fM_{g,l}^{\bu}(X,B;J)^{\phi}\lra X, \qquad
\big[u,(z_1^+,z_1^-),\ldots,(z_l^+,z_l^-)\big]\lra u(z_i^+),$$
be the evaluation at the first point in the $i$-th pair of conjugate points.
Let 
$$\ov\fM_{g,l}'^{\bu}(X,B;J)^{\phi}
=\big\{[\u]\!\in\!\ov\fM_{g,l}^{\bu}(X,B;J)^{\phi}\!:\,
\ev_{l-1}([\u])\!=\!\ev_l([\u])\big\}.$$
The short exact sequence
$$0\lra T\ov\fM_{g,l}'^{\bu}(X,B;J)^{\phi}\lra 
T\ov\fM_{g,l}^{\bu}(X,B;J)^{\phi}|_{\ov\fM_{g,l}'^{\bu}(X,B;J)^{\phi}} \lra
\ev_l^*TX\lra 0$$
induces an isomorphism
\BE{SubIsom_e}
\La_{\R}^{\top}\big(T\ov\fM_{g,l}^{\bu}(X,B;J)^{\phi}|_{\ov\fM_{g,l}'^{\bu}(X,B;J)^{\phi}}\big)\\
\approx 
\La_{\R}^{\top}\big(T\ov\fM_{g,l}'^{\bu}(X,B;J)^{\phi}\big)
\otimes \ev_l^*\big(\La_{\R}^{\top}(TX)\big)\EE
of real line bundles over $\ov\fM_{g,l}'^{\bu}(X,B;J)^{\phi}$.\\

\noindent
The identification of the last two pairs of conjugate marked points induces
an immersion 
\BE{iodfn_e2}\io\!: 
\ov\fM_{g-2,l+2}'^{\bu}(X,B;J)^{\phi}\lra  \ov\fM_{g,l}^{\bu}(X,B;J)^{\phi}\,.\EE
This immersion takes the main stratum of the domain,
i.e.~the subspace consisting of real morphisms from smooth symmetric surfaces, 
to the subspace of the target  consisting
of real morphisms from symmetric surfaces with one pair of conjugate nodes.
There is a canonical isomorphism
$$\cN\io\equiv \frac{\io^*T\ov\fM_{g,l}^{\bu}(X,B;J)^{\phi}}{T\ov\fM_{g-2,l+2}'^{\bu}(X,B;J)^{\phi}}
\approx \cL_{l+1}\!\otimes_{\C}\!\cL_{l+2}$$
of the normal bundle of~$\io$ with the tensor product of the universal tangent line bundles
for the first points in the last two conjugate pairs.
It induces an isomorphism
\BE{RestrOrient_e0}\begin{split} 
\io^*\big(\La_{\R}^{\top}\big(T\ov\fM_{g,l}^{\bu}(X,B;J)^{\phi}\big)\big)
\approx \La_{\R}^{\top}\big(T\ov\fM_{g-2,l+2}'^{\bu}(X,B;J)^{\phi}\big)
\otimes \La_{\R}^2\big(\cL_{l+1}\!\otimes_{\C}\!\cL_{l+2}\big)
\end{split}\EE
of real line bundles over $\ov\fM_{g-2,l+2}'^{\bu}(X,B;J)^{\phi}$.
Along with~\eref{SubIsom_e} with $(g,l)$ replaced by~$(g\!-\!2,l\!+\!2)$, it determines an isomorphism
\BE{CompOrient_e}\begin{split}   
&\La_{\R}^{\top}\big(T\ov\fM_{g-2,l+2}^{\bu}(X,B;J)^{\phi}|_{\ov\fM_{g-2,l+2}'^{\bu}(X,B;J)^{\phi}}\big)
\otimes \La_{\R}^2\big(\cL_{l+1}\!\otimes_{\C}\!\cL_{l+2}\big) \\
&\hspace{1.5in} \approx\io^*\big(\La_{\R}^{\top}\big(T\ov\fM_{g,l}^{\bu}(X,B;J)^{\phi}\big)\big)
\otimes  \ev_{l+1}^*\big(\La_{\R}^{\top}(TX)\big)
\end{split}\EE
of real line bundles over $\ov\fM_{g-2,l+2}'^{\bu}(X,B;J)^{\phi}$.

\begin{prp}\label{CompOrient_prp}
Let $(X,\om,\phi)$, $g,l$, $B$, and $J$ be as in Theorem~\ref{orient_thm} with $n\!\not\in\!2\Z$.
The isomorphism~\eref{CompOrient_e} is orientation-reversing with respect 
to the orientations on the moduli spaces determined by  a real orientation on~$(X,\om,\phi)$
and  the canonical  orientations on  $\cL_{l+1}\!\otimes_{\C}\!\cL_{l+2}$
and~$TX$.
\end{prp}

\noindent
This proposition is established in~\cite{RealGWsII}.
Its substance is that the orientations on $\ov\fM_{g-2,l+2}'^{\bu}(X,B,J)^{\phi}$
induced from the orientations of $\ov\fM_{g-2,l+2}^{\bu}(X,B,J)^{\phi}$ 
and $\ov\fM_{g,l}^{\bu}(X,B,J)^{\phi}$  via the isomorphisms~\eref{SubIsom_e} 
and~\eref{RestrOrient_e0} are opposite.
This unfortunate reversal of orientations under the immersion~\eref{iodfn_e2} 
can be fixed by multiplying the orientation on $\ov\fM_{g,l}^{\bu}(X,B,J)^{\phi}$
described at the end of Section~\ref{bnd_subs} 
by $(-1)^{\flr{g/2}+1}$, for example.
Along with the sign flip at the end of Section~\ref{outline_sec},
this would change the canonical orientation on $\fM_{g,l}^{\bu}(X,B,J)^{\phi,\si}$
constructed in the proof of Corollary~\ref{orient0_crl} by $(-1)^{\flr{g/2}+|\si|_0}$,
where $|\si|_0$ is the number of topological components of the fixed locus of~$(\Si,\si)$.
This sign change would make the real genus~1 degree~$d$ GW-invariant of $(\P^3,\om_4,\tau_4)$
with $d$~pairs of conjugate point constraints
to be~0 for $d\!=\!2$, 1 for $d\!=\!4$, and $4$ for $d\!=\!6$.
In particular, it would make the $d\!=\!4$ number congruent to its complex analogue modulo~4;
this is the case for Welschinger's (genus~0) invariants for many target spaces.
However, this property fails for the $(g,d)\!=\!(1,5)$ numbers 
(the real enumerative invariant is~0, while its complex analogue is~42).\\

\noindent
We note that the statement of Proposition~\ref{CompOrient_prp} is invariant under 
interchanging the points within the last two conjugate pairs simultaneously
(this corresponds to reordering the nodes of a nodal map). 
This interchange reverses the orientation of the last factor on the left-hand side of~\eref{CompOrient_e},
because the complex rank of $\cL_{l+1}\!\otimes_{\C}\!\cL_{l+2}$ is~1, 
and the orientation of the last factor on the right-hand side of~\eref{CompOrient_e},
because the complex rank of $TX$ is~odd.\\

\noindent
If $n\!\in\!2\Z$ and $g\!+\!l\!\ge\!2$, the comparison~\eref{CompOrient_e} should be 
made with the tangent bundles of the moduli spaces twisted as in~\eref{twist_rmk_e}.
The proof of Proposition~\ref{CompOrient_prp} appearing in~\cite{RealGWsII} still applies,
but leads to the opposite conclusion; see~\cite[Remark~1.3]{RealGWsII}.

\section{Outline of the main proofs}
\label{outline_sec}

\noindent
The origins of real GW-theory go back to~\cite{Melissa},
where the spaces~\eref{gluedspace_e} are topologized by
adapting the description of Gromov's topology in~\cite{LT}
via versal families of deformations of abstract complex curves to the real setting.
This  demonstrates that the codimension~1 boundaries of the spaces 
in~\eref{notgluedspace_e}  form hypersurfaces inside 
the full moduli space~\eref{gluedspace_e} and thus reduces the problem of constructing 
a real GW-theory for a real symplectic manifold $(X,\om,\phi)$
to showing~that
\begin{enumerate}[label=(\Alph*),leftmargin=*]

\item\label{orient_it} the uncompactified moduli spaces $\fM_{g,l}(X,B;J)^{\phi,\si}$
are orientable for all types of orientation-reversing involutions~$\si$ on a genus~$g$ symmetric surface,
and 

\item\label{bnd_it} an orientation of $\fM_{g,l}(X,B;J)^{\phi}$ extends across the (virtually)
codimension-one strata of the compact moduli space $\ov\fM_{g,l}(X,B;J)^{\phi}$. 

\end{enumerate} 
In this paper, we achieve both objectives for real-orientable $2n$-manifolds 
with $n\!\not\in\!2\Z$.\\

\noindent
Let $g,l\!\in\!\Z^{\ge0}$ with $g\!+\!l\!\ge\!2$.
Denote by $\cM_{g,l}^{\si}$ the Deligne-Mumford moduli space of $\si$-compatible 
complex structures on a genus~$g$ symmetric surface $(\Si,\si)$ with
 $l$~conjugate pairs of 
marked points and by
$$\ov\cM_{g,l}^{\si}\supset\cM_{g,l}^{\si}$$
its compactification obtained by including stable nodal symmetric surfaces.
The codimension-one boundary strata~of $\ov\cM_{g,l}^{\si}\!-\!\cM_{g,l}^{\si}$
consist of real one-nodal symmetric surfaces.
Each stratum is either a hypersurface in $\ov\cM_{g,l}^{\si}$
or is a boundary of the spaces $\ov\cM_{g,l}^{\si}$
for precisely two  topological types of orientation-reversing involutions~$\si$ on~$\Si$.
Let 
$$\R\cM_{g,l}=\bigcup_{\si}\cM_{g,l}^{\si} \qquad\hbox{and}\qquad
\R\ov\cM_{g,l}=\bigcup_{\si}\ov\cM_{g,l}^{\si}$$
denote the (disjoint) union of the uncompactified real Deligne-Mumford moduli spaces 
and the union of the compactified real Deligne-Mumford moduli spaces, respectively, 
taken over all topological types of orientation-reversing involutions~$\si$
on~$\Si$.
The moduli space $\R\ov\cM_{g,l}$ is not orientable if $g\!\in\!\Z^+$.
One of the two main steps in the proof of Theorem~\ref{orient_thm} is Proposition~\ref{DMext_prp};
it implies that the real line bundle 
\BE{DMext_e}
\La_{\R}^{\top}\big(T\R\ov\cM_{g,l}\big)\otimes\big(\!\det\dbar_{\C}\big)\lra \R\ov\cM_{g,l}\EE
has a canonical orientation.\\

\noindent
With $g,l$ as above, let
\BE{ffdfn_e}\ff\!: \ov\fM_{g,l}(X,B;J)^{\phi} \lra \R\ov\cM_{g,l}\EE
denote the \textsf{forgetful morphism}.
For each $[\u]\!\in\!\ov\fM_{g,l}(X,B;J)^{\phi}$ with stable domain, 
it induces a~canonical isomorphism
\BE{TuMdecomp_e}
\La_{\R}^{\top}\big(T_{[\u]}\ov\fM_{g,l}(X,B;J)^{\phi,\si}\big)\approx
\big(\!\det D_{(TX;\tnd\phi);\u}\big)\otimes \La_{\R}^{\top}(T_{\ff(\u)}\ov\cM_{g,l}^{\si}),\EE
where $\det D_{(TX;\tnd\phi);\u}$ is the determinant of the linearization~$D_{(TX;\tnd\phi);\u}$ of 
the real $\dbar_J$-operator at~$\u$; see Section~\ref{DetLB_subs}.
The orientability of the last factor in~\eref{TuMdecomp_e} as~$\u$ varies
is indicated by the previous paragraph.
We study the orientability of the first factor on the right-hand side 
of~\eref{TuMdecomp_e} via
\sf{the relative determinant of~$D_{(TX;\tnd\phi);\u}$}, 
\BE{fDdfn_e0}
\rdet\,D_{(TX;\tnd\phi);\u}\equiv 
\big(\!\det D_{(TX;\tnd\phi);\u}\big)\otimes\big(\!\det\dbar_{\Si;\C}\big)^{\otimes n}\,,\EE
where $2n\!=\!\dim X$ and 
$\det\dbar_{\Si;\C}$ is 
the standard real Cauchy-Riemann (or \sf{CR-}) operator on the domain~$(\Si,\si)$ of~$\u$
with values in~$(\C,\fc)$.
An orientation on~\eref{fDdfn_e0} determines a correspondence between 
the orientations on $\det D_{(TX;\tnd\phi);\u}$ and on the determinant 
$\det n\dbar_{\Si;\C}$ of the standard real $\dbar$-operator on the trivial rank~$n$ 
real bundle
\hbox{$(\Si\!\times\!\C^n,\si\!\times\fc)$} over~$(\Si,\si)$.
On the other hand, orientations on $\rdet\,D_{(TX;\tnd\phi);u}$ are naturally related 
to the topology of real bundles pairs over~$(\Si,\si)$.
In particular, the second main step in the proof of Theorem~\ref{orient_thm} is 
Proposition~\ref{canonisom_prp};
it implies that a real orientation on~$(X,\om,\phi)$ determines an orientation on~\eref{fDdfn_e0}
which varies continuously with~$\u$.
Combined with the canonical orientation of~\eref{DMext_e} and
the canonical isomorphism of~\eref{TuMdecomp_e}, 
the latter orientation determines an orientation on the line bundle~\eref{main_thm}.

\subsection{The orientability problem}
\label{orient_subs}

\noindent
The typical approaches to the orientability problem in real GW-theory, 
i.e.~\ref{orient_it} on page~\pageref{orient_it},  
involve computing the signs of the actions of appropriate real diffeomorphisms
on determinant lines of real CR-operators over some coverings of
$\fM_g(X,B;J)^{\phi,\si}$ arising from bordered surfaces.
These approaches work as long as all relevant diffeomorphisms are homotopically 
fairly simple and in particular preserve a bordered surface in~$\Si$ that doubles 
to~$\Si$ or map it to its conjugate half.
This is the case if the fixed locus $\Si^{\si}\!\subset\!\Si$ of the involution~$\si$ is separating;
a good understanding  of the orientability of the moduli spaces $\fM_g(X,B;J)^{\phi,\si}$
in  such cases is obtained in \cite{Sol, FOOO9, Ge2, Remi, XCapsSetup, XCapsSigns}.
This is also the case for any involution~$\si$  of genus $g\!=\!0,1$.
In particular, the restriction of Theorem~\ref{orient_thm} to $\fM_g(X,B;J)^{\phi,\si}$ 
for the genus~1 involutions~$\si$ is essentially \cite[Theorem~1.2]{XCapsSetup};
a less general version of \cite[Theorem~1.2]{XCapsSetup} is \cite[Theorem~1.1]{Teh2}.
However, understanding the orientability in the bordered case is not sufficient beyond genus~1, 
due to the presence of real diffeomorphisms of~$(\Si,\si)$ not preserving any half of~$\Si$; 
see Example~\ref{half_eg}.
The subtle effect of such diffeomorphisms on the orientability is hard to determine.\\

\noindent
In contrast to~\cite{Sol, FOOO9}, in~\cite{XCapsSigns} we allowed the complex structure
on a bordered domain to vary and considered diffeomorphisms interchanging
the boundary components and their lifts to automorphisms of real bundle pairs.
We discovered~that they often act with the same signs~on
\begin{enumerate}[label=(A\arabic*),leftmargin=*]

\item\label{DM_it} a natural cover of $\cM_g^{\si}$ 
and the determinant line bundle for the trivial rank~1 real bundle pair over~it;

\item\label{SQ_it} the determinants of real CR-operators on
the square of a rank~1 real bundle pair with orientable real part  
and on the trivial rank~1  real bundle pair;

\item\label{Det_it} the determinants of
real CR-operators on an odd-rank real bundle pair and its top exterior power;

\end{enumerate}
see \cite[Propositions~2.5,4.1,4.2]{XCapsSigns}.
In this paper, we show that these {\it analytic} statements are in fact underpinned
by the {\it topological} statement of Proposition~\ref{canonisom_prp} 
concerning canonical homotopy classes of 
isomorphisms between real bundle pairs over a symmetric surface~$(\Si,\si)$.
As we work on the more elemental, topological level of real bundle pairs, 
we do not compute the signs of any automorphisms,  as is done in the bordered surfaces approach. 
We instead obtain isomorphisms of real bundle pairs over (families of) symmetric surfaces 
and apply the determinant functor to these isomorphisms 
(Corollaries~\ref{canonisom_crl2a} and~\ref{canonisomExt_crl2a}).
In contrast to the bordered surfaces approach, this works for all type of involutions 
on the domain and in flat families of (possibly) nodal curves.\\

\noindent
Proposition~\ref{DM_prp}, which appears to be of its own interest, 
endows the restriction of the line bundle~\eref{DMext_e} to each 
topological component~$\cM_{g,l}^{\si}$ of~$\R\cM_{g,l}$
with a canonical orientation and thus explains~\ref{DM_it}. 
This canonical orientation over an element~$[\cC]$ of~$\cM_{g,l}^{\si}$
is obtained by tensoring canonical orientations on four lines:
\begin{enumerate}[label=(\arabic*),leftmargin=*]

\item\label{KSisom_it0} the orientation~on 
the tensor product of the top exterior powers of the left and middle terms
in~\eref{DM_prp_e3} induced by the Kodaira-Spencer (or \sf{KS}) isomorphism,

\item\label{DIandSDisom_it0}
the orientation~on 
the tensor product of the top exterior powers of the  middle term
in~\eref{DM_prp_e3} and of the right-hand side in~\eref{DM_prp_e5} 
induced by the Dolbeault Isomorphism and Serre Duality,

\item\label{FMisom_it0} the orientation on~\eref{DM_prp_e7}
induced by the short exact sequence~\eref{FMses_e}
and the specified orientations of~\eref{H0ScC_e}, 

\item\label{SQisom_it0} the orientation~on the fiber of the line bundle~\eref{domains_e15}
over~$[\cC]$
determined by Corollaries~\ref{canonisom_crl} and~\ref{canonisom_crl2a}.

\end{enumerate}
We combine the canonical orientation on the restriction of the line bundle~\eref{DMext_e}
to the main stratum $\R\cM_{g,l}$, the orientation on the relative determinants~\eref{fDdfn_e0}
induced by the real orientation on~$(X,\om,\phi)$,
and the isomorphism~\eref{TuMdecomp_e} 
to establish the restriction of Theorem~\ref{orient_thm} to
the uncompactified moduli space~$\fM_{g,l}(X,B;J)^{\phi}$; see Corollary~\ref{orient0_crl}.

\subsection{The codimension-one boundary problem}
\label{bnd_subs}

\noindent
Once the orientability problem~\ref{orient_it} is resolved,
one can study the codimension-one boundary problem, i.e.~\ref{bnd_it} on page~\pageref{bnd_it}.
It then asks whether it is possible to choose an orientation on the~subspace
$$\fM_{g,l}(X,B;J)^{\phi,\si}\subset\ov\fM_{g,l}(X,B;J)^{\phi}$$
for each topological type of  orientation-reversing involutions~$\si$
on a genus~$g$ symmetric surface
so that the resulting orientations do not change across the (virtually) codimension-one strata 
of $\ov\fM_{g,l}(X,B;J)^{\phi}$.
These strata are  (virtual) hypersurfaces inside of the full moduli space
and 
consist of morphisms from one-nodal symmetric surfaces to~$(X,\phi)$.\\

\noindent
As described in \cite[Section~3]{Melissa}, there are four distinct types of
one-nodal symmetric surfaces  $(\Si,x_{12},\si)$:
\begin{enumerate}[leftmargin=.3in]

\item[(E)] $x_{12}$ is an isolated real node, i.e.~$x_{12}$ is an isolated point
of the fixed locus $\Si^{\si}\!\subset\!\Si$;

\item[(H)]\label{Hdegen_it}  $x_{12}$ is a non-isolated real node and

\begin{enumerate}[label=(H\arabic*),leftmargin=.15in] 

\item the topological component $\Si_{12}^{\si}$ of $\Si^{\si}$ containing $x_{12}$
is algebraically irreducible (the normalization~$\wt\Si_{12}^{\wt\si}$
of~$\Si_{12}^{\si}$ is connected);

\item the topological component $\Si_{12}^{\si}$ of $\Si^{\si}$ containing $x_{12}$
is algebraically reducible, but $\Si$ is algebraically irreducible
(the normalization~$\wt\Si_{12}^{\wt\si}$ of~$\Si_{12}$ is disconnected,
but the normalization $\wt\Si$ of~$\Si$ is connected);

\item  $\Si$ is algebraically reducible (the normalization $\wt\Si$ of~$\Si$ is disconnected).

\end{enumerate}
\end{enumerate} 
In \cite[Section~3]{Teh2}, the above types  are called (II), (IC1), (IC2), and~(ID), respectively.
In the genus~0 case, the degenerations~(E) and~(H3) are known as 
\textsf{codimension-one sphere bubbling} and
\textsf{disk bubbling}, respectively;
the degenerations~(H1) and~(H2) cannot occur in the genus~0 case.\\

\begin{figure}
\begin{center}
\includegraphics[width=0.7\textwidth,height=80px]{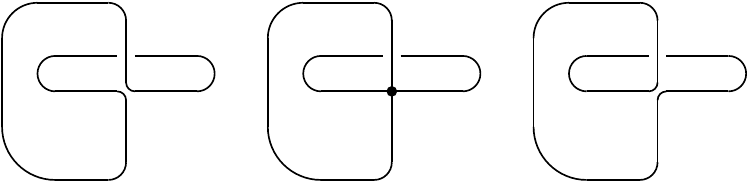}
\end{center}
\caption{The real locus transition through (H2) and (H3) degenerations}
\label{H2H3_fig}
\end{figure}

\noindent
The transitions between smooth symmetric surfaces across the four types of 
one-nodal symmetric surfaces
are illustrated in \cite[Figures~12-15]{Melissa}.
A transition through a degeneration~(H3) does not change the topological type of the involution.
Thus, each stratum of morphisms from a one-nodal symmetric surface of type~(H3)
to~$(X,\phi)$ is a hypersurface inside of $\ov\fM_{g,l}(X,B;J)^{\phi,\si}$
for some genus~$g$ involution~$\si$.
This transition does not play a material role in the approach of \cite{Wel4,Wel6}, 
which is based on counting real genus~0 curves, rather than their parametrizations.
In the approach of~\cite{Cho,Sol}, which is based on counting morphisms from disks 
as halves of morphisms from~$(\P^1,\tau)$, the degeneration~(H3) appears 
as a codimension-one boundary consisting of morphisms from two disks.
This boundary is glued to itself in~\cite{Cho,Sol} by the involution which corresponds
to flipping one of the disks; 
this involution is orientation-reversing under suitable assumptions
and so the orientation on the main stratum extends across the resulting hypersurfaces. 
A perspective that combines the hypersurface viewpoint of \cite{Wel4,Wel6} with
the parametrizations setting of~\cite{Cho,Sol} appears in~\cite{Ge2}.
It fits naturally with the approach of this paper to studying transitions
through all four degeneration types.\\

\noindent
A transition through a degeneration~(E) changes the number $|\si|_0$
of topological components (circles)
of the fixed locus $\Si^{\si}\!\subset\!\Si$ by~one. 
In the terminology of Section~\ref{SymmSurf_subs}, 
such a transition can be described
as collapsing a standard boundary component of a bordered half-surface 
(corresponding to a component of~$\Si^{\si}$) and then replacing it with a crosscap.
In particular, each stratum of morphisms from a one-nodal symmetric surface of type~(E)
to~$(X,\phi)$ is a boundary of the spaces $\ov\fM_{g,l}(X,B;J)^{\phi,\si}$
for precisely two  topological types of genus~$g$ involutions~$\si$.
In the genus~0 case, the analysis of orientations necessary for 
the gluing of the two spaces along their common boundary is carried
out in \cite[Section~3]{Teh}.\\

\noindent
A transition through a degeneration~(H1) also changes the number $|\si|_0$ 
by one, but through a more complicated process. 
Such a transition transforms two components of~$\Si^{\si}$ into one and creates 
an additional crosscap ``near" the node of the one-nodal surface $(\Si,x_{12},\si)$.
Each stratum of morphisms from a one-nodal symmetric surface of type~(H1)
to~$(X,\phi)$ is a boundary of the spaces $\ov\fM_{g,l}(X,B;J)^{\phi,\si}$
for precisely two topological types of genus~$g$ involutions~$\si$.
A degeneration~(H1) cannot occur in genus~0, but does occur in genus~1 and higher;
see the last diagram in \cite[Figure~2]{Teh2}.\\

\begin{figure}
\begin{center}
\includegraphics[width=0.7\textwidth,height=80px]{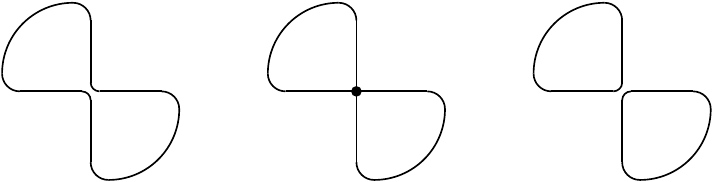}
 \end{center}
\caption{The real locus transition through an (H1) degeneration}
\label{H1_fig}
\end{figure}

\noindent
A transition through a degeneration~(H2) does not change the number of topological components
of~$\Si^{\si}$, but cuts one of them into two arcs and re-joins the arcs in the opposite way.
The transformation of the real locus is the same as in the (H3) case, but 
an (H2) transition also inserts or removes two crosscaps.
This transition may or may not change the topological type of the involution~$\si$.
If the fixed locus of  $(\Si,x_{12},\si)$ is separating,
then this transition changes the topological type of~$\si$
and each stratum of morphisms from $(\Si,x_{12},\si)$
to~$(X,\phi)$ is a boundary of the spaces $\ov\fM_{g,l}(X,B;J)^{\phi,\si}$
for precisely two topological types of genus~$g$ involutions~$\si$.
If the fixed locus of $(\Si,x_{12},\si)$ is non-separating,
then this transition does not change the topological type of~$\si$
and each stratum of morphisms from $(\Si,x_{12},\si)$
to~$(X,\phi)$ is a hypersurface inside of $\ov\fM_{g,l}(X,B;J)^{\phi,\si}$
for some genus~$g$ involution~$\si$.
A degeneration~(H2) cannot occur in genus~0 or~1, but does occur in genus~2 and~higher.\\

\noindent
The transitions~(H1) and~(H2) do not preserve any bordered surface in~$\Si$ that doubles to~$\Si$,
in contrast to the transition~(E);
the transition~(H3) does not preserve any bordered surface in~$\Si$ that doubles to~$\Si$ either,
but its nature is fairly simple.
As in the case of~\ref{orient_it} discussed in Section~\ref{orient_subs},
this makes the issue~\ref{bnd_it} difficult to study using the standard approaches 
to orienting the determinant lines of real CR-operators
even when issue~\ref{orient_it} is resolved; see \cite[Conjectures~1.3,6.3]{Teh2}.
We approach~\ref{bnd_it}  by studying isomorphisms between real bundle pairs,
but this time over one-nodal symmetric surfaces $(\Si,x_{12},\si)$.
As in the smooth case, this circumvents a direct computation of the signs
 of any automorphisms of the determinant lines of  real CR-operators.\\

\noindent
Corollary~\ref{orient0_crl} uses a  real orientation on~$(X,\om,\phi)$ to
endow  the fiber of the line bundle~\eref{orient_thm_e} 
over each element~$[\u]$ of $\fM_{g,l}(X,B;J)^{\phi}$ with an orientation.
The latter is obtained by combining
the orientation on the relative determinant~\eref{fDdfn_e0}
of the linearization of the $\dbar_J$-operator at~$\u$
induced by the real orientation on~$(X,\om,\phi)$ and
the canonical orientation on  the fiber of
the line bundle~\eref{DMext_e} over~$\ff(\u)$ via the isomorphism~\eref{TuMdecomp_e}.
The real orientation on~$(X,\om,\phi)$ specifies
a  homotopy class of isomorphisms~\eref{realorient2_e2}
with $(V,\vph)\!=\!u^*(TX,\tnd\phi)$.
The latter determines an orientation  on the relative determinant~\eref{fDdfn_e0}.
An isomorphism in the specified homotopy class over 
a one-nodal symmetric surface $(\Si,x_{12},\si)$ extends
to an isomorphism in the specified homotopy class for  each nearby smooth symmetric surface.
Therefore, so does the induced orientation 
on the relative determinant~\eref{fDdfn_e0};
see Corollary~\ref{canonisomExt2_crl2a}.
This means that the induced orientation of the line bundle formed by
the relative determinants~\eref{fDdfn_e0} does not change
across any of the codimension-one strata of $\ov\fM_{g,l}(X,B;J)^{\phi}$.\\ 

\noindent
The situation with the canonical orientation on the restriction of 
the line bundle~\eref{DMext_e} to $\R\cM_{g,l}$ provided by 
Proposition~\ref{DM_prp} on page~\pageref{DM_prp} is very different.
This is partly indicated by the statement of Proposition~\ref{DMext_prp},
but the actual situation is even more delicate.
This canonical orientation constructed in Proposition~\ref{DM_prp} is 
the tensor product of the four orientations listed at the end of Section~\ref{orient_subs}.
The line bundles on which these orientations are defined naturally extend across
the codimension-one boundary strata of~$\R\ov\cM_{g,l}$.
The behavior of the four orientations across these strata of~$\R\ov\cM_{g,l}$
is described in the proof of  Proposition~\ref{DMext_prp} at the end of Section~\ref{DMext_subs}.
The orientations~\ref{DIandSDisom_it0} and~\ref{FMisom_it0} in Section~\ref{orient_subs}
do not change across any of codimension-one strata.
The orientation~\ref{KSisom_it0}, determined by the KS isomorphism~\eref{DM_prp_e3}
for smooth symmetric surfaces, changes across all codimension-one boundary strata;
see Lemma~\ref{KSext_lmm}.
The orientation~\ref{SQisom_it0} over a smooth symmetric surface is induced 
by Corollary~\ref{canonisom_crl2a} from 
the canonical real orientation of Corollary~\ref{canonisom_crl} with $L\!=\!T^*\Si$.
The analogue of~$L$ for a one-nodal symmetric surface $(\Si,x_{12},\si)$ is played by
the restriction of the line bundle~$\wh\cT$ of Lemma~\ref{TSiext_lmm}.
The restriction of its real part to the singular component of the fixed locus 
is orientable for the degenerations of types~(E) and~(H1)
and is not orientable
for the degenerations of types~(H2) and~(H3); see Lemma~\ref{EH1vsH2H3_lmm}.
In the latter cases, the orientation~\ref{SQisom_it0} for $(\Si,x_{12},\si)$
 depends on the orientation
of the fixed locus; see Corollary~\ref{canonisom_crl}.
For the degenerations of types~(H2) and~(H3), no orientation of 
the singular component of the fixed locus  extends to nearby smooth symmetric surfaces
(because (H2) and~(H3) involve cutting a fixed circle into two arcs and re-joining them in the opposite way).
Therefore, the orientation~\ref{SQisom_it0} changes in the transitions~(H2)
and~(H3) and does not in  the transitions~(E) and~(H1); see Corollary~\ref{realorient_crl}.\\

\begin{table}
\begin{center}
\begin{small}
\renewcommand\arraystretch{1.8} 
\begin{tabular}{||c|c|c|c||}
\hline\hline
orientation/parity of&  induced by& (E)/(H1)& (H2)/(H3)\\
\hline
$\La_{\R}^{\top}(T_{\Si}\cM_{g,l}^{\si})\!\otimes\!
\La_{\R}^{\top}(\wch{H}^1(\Si;T\Si))^{\si})$ & KS isomorphism~\eref{DM_prp_e3} 
& $-$ & $-$\\
\hline
$(\det\dbar_{(T^*\Si,\tnd\si^*)^{\otimes2}})\!\otimes\!(\!\det\dbar_{\Si;\C})$
&Corollaries~\ref{canonisom_crl} and~\ref{canonisom_crl2a} & $+$ & $-$\\
\hline\hline
$|\si|_0$& N/A &  $-$ & $+$\\
\hline\hline
\end{tabular}
\end{small}
\end{center}
\caption{The extendability of the canonical orientations and of the parity of 
the number of components of~$\Si^{\si}$ across the codimension-one strata:
$+$ extends, $-$ flips.
All other canonical orientations factoring into the orientation of 
the line bundle~\eref{orient_thm_e} extend across all codimension-one strata.}
\label{SignChg_tbl}
\end{table}

\noindent
The key points of the previous paragraph are summarized in Table~\ref{SignChg_tbl}.
They imply that the canonical orientation on the restriction of 
the line bundle~\eref{DMext_e} to $\R\cM_{g,l}$ provided by 
Proposition~\ref{DM_prp}
does not change in the transitions~(H2) and~(H3) and changes
in  the transitions~(E) and~(H1).
These transitions have the same effect  on the parity of
the number~$|\si|_0$ of connected components of the fixed locus~$\Si^{\si}$ of~$(\Si,\si)$.
Thus, the canonical orientation on the restriction of~\eref{DMext_e} to $\R\cM_{g,l}$
multiplied by $(-1)^{g+|\si|_0+1}$  over each topological component~$\cM_{g,l}^{\si}$
of~$\R\cM_{g,l}$ extends over all of~$\R\ov\cM_{g,l}$.
The same considerations apply to the orientation on the restriction
of the line bundle~\eref{orient_thm_e} to $\fM_{g,l}(X,B;J)^{\phi,\si}$ 
provided by Corollary~\ref{orient0_crl}.
If $n\!\not\in\!2\Z$, this sign modification leaves the  orientations of the moduli spaces 
$\fM_{g,l}(X,B;J)^{\phi,\si}$ for separating involutions~$\si$ unchanged.

\section{Notation and review}
\label{prelim_sec}

\noindent
In this section, we set up the notation and terminology used throughout 
Sections~\ref{signcomp_sec} and~\ref{BdExt_sec}.
We recall some facts about symmetric surfaces, associated half-surfaces,
their moduli spaces, real Cauchy-Riemann operators, and 
their determinant line bundles.

\subsection{Symmetric surfaces and half-surfaces}
\label{SymmSurf_subs}

\noindent
Let $(\Si,\si)$ be a genus~$g$ symmetric surface.
We denote by $|\si|_0\!\in\!\Z^{\ge0}$ the number of connected components of~$\Si^{\si}$;
each of them is a circle.
Let $\lr\si\!=\!0$ if the quotient $\Si/\si$ is orientable, 
i.e.~$ \Si\!-\! \Si^{\si}$ is disconnected, and $\lr\si\!=\!1$ otherwise. 
There are $\flr{\frac{3g+4}{2}}$ different topological types of orientation-reversing 
involutions $\si$ on $\Si$ classified by the triples $(g,|\si|_0,\lr\si)$; 
see \cite[Corollary~1.1]{Nat}. \\

\noindent
An \sf{oriented symmetric half-surface} (or simply \sf{oriented sh-surface}) 
is a pair $(\Si^b,c)$ consisting of an oriented bordered smooth surface~$\Si^b$ 
and an involution $c\!:\prt\Si^b\!\lra\!\prt\Si^b$ preserving each component
and the orientation of~$\prt\Si^b$.
The restriction of~$c$  to a boundary component 
is either the identity or the antipodal map
\BE{antip_e}\fa\!:S^1\lra S^1, \qquad z\lra -z,\EE
for a suitable identification of $(\prt\Si^b)_i$ with $S^1\!\subset\!\C$;
the latter type of boundary structure is called \sf{crosscap} 
in the string theory literature.
We define
$$c_i=c|_{(\prt\Si^b)_i}, \qquad 
|c_i|= \begin{cases} 0,&\hbox{if}~c_i=\id;\\ 1,&\hbox{otherwise};\end{cases}
\qquad
|c|_k=\big|\{(\prt\Si^b)_i\!\subset\!\Si^b\!:\,|c_i|\!=\!k\}\big|\quad k=0,1.$$
Thus, $|c|_0$ is the number of standard boundary components of $(\Si^b,\prt\Si^b)$  
and $|c|_1$ is the number of crosscaps.
Up to isomorphism, each oriented sh-surface $(\Si^b,c)$ is determined by the genus~$g$ of~$\Si^b$,
the number~$|c|_0$ of ordinary boundary components, and the number~$|c|_1$ of crosscaps.
We denote by $(\Si_{g,m_0,m_1},c_{g,m_0,m_1})$ the genus~$g$ oriented sh-surface
with $|c_{g,m_0,m_1}|_0\!=\!m_0$ and $|c_{g,m_0,m_1}|_1\!=\!m_1$.\\

\begin{figure}
\begin{center}
\leavevmode
\includegraphics[width=0.6\textwidth,height=150px]{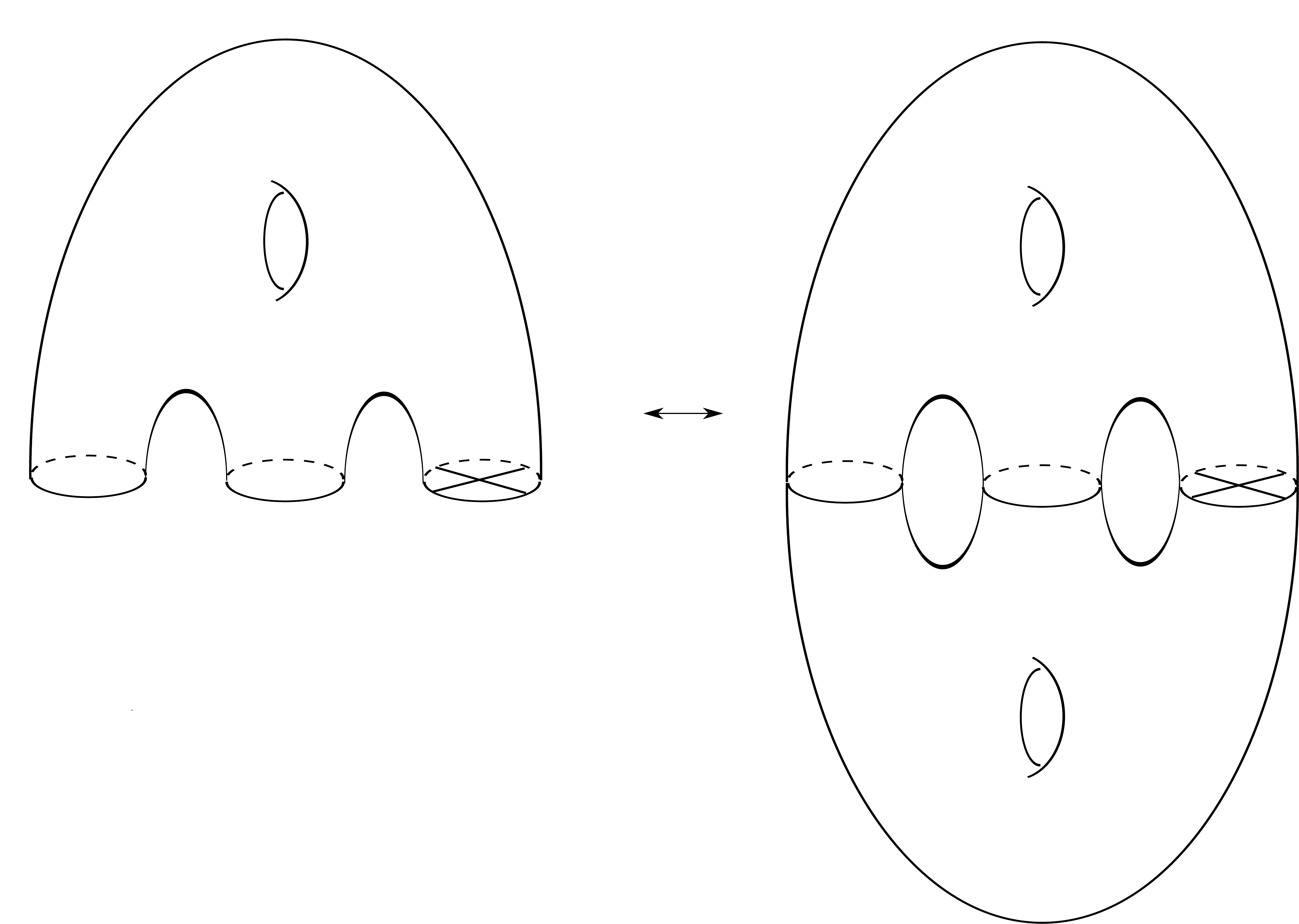}
\end{center}
\caption{Doubling an oriented sh-surface}
\label{HSdoub_fig}
\end{figure}

\noindent
An  oriented sh-surface $(\Si^b,c)$ of type $(g,m_0,m_1)$ \sf{doubles} 
to a symmetric surface~$(\Si, \si)$ of type 
$$(g(\Si), |\si|_0,\lr\si)=\begin{cases} (2g\!+\!m_0\!+\!m_1\!-\!1,m_0,0),& \text{if}~ m_1=0;\\
(2g\!+\!m_0\!+\!m_1\!-\!1,m_0,1),& \text{if}~m_1\neq 0;
\end{cases}$$
so that $\si$ restricts to~$c$ on the cutting circles (the boundary of~$\Si^b$);
see \cite[(1.6)]{XCapsSetup} and Figure~\ref{HSdoub_fig}.
Since this doubling construction covers all topological types of orientation-reversing 
involutions $\si$ on $\Si$, for every symmetric surface $(\Si,\si)$ there is 
an oriented sh-surface $(\Si^b,c)$ which doubles to~$(\Si,\si)$.
In general,  the topological type of such an sh-surface is not unique. 
There is a topologically unique oriented sh-surface~$(\Si^b,c)$
doubling to a symmetric surface~$( \Si,\si)$ if $\lr\si\!=\!0$, 
in which case $(\Si^b,c)$ has no crosscaps, or 
$|\si|_0\!\ge\!g( \Si)\!-\!1$,
in which case $(\Si^b,c)$ is either of genus at most~1 and has no crosscaps or
of genus~0 and has at most 2 crosscaps.\\

\noindent
Denote by $\cD_\si$ the group of orientation-preserving diffeomorphisms of~$\Si$  
 commuting with the involution~$\si$.
If $(X,\phi)$ is a smooth manifold with an involution, $l\!\in\!\Z^{\ge0}$,
and $B\in H_2(X;\Z)$, let  
$$\fB_{g,l}(X,B)^{\phi,\si}\subset \fB_g(X)^{\phi,\si}\times \Si^{2l}$$
denote the space of real maps  $u\!:(\Si,\si)\!\lra\!(X,\phi)$ with $u_*[ \Si]_{\Z}=B$
and $l$ pairs of conjugate non-real marked distinct points.
We define
$$\cH_{g,l}(X,B)^{\phi,\si}=
\big(\fB_{g,l}(X,B)^{\phi,\si}\!\times\!\cJ_{\Si}^{\si}\big)/\cD_\si.$$
The action of $\cD_\si$ on $\cJ_{\Si}$ given by $h\cdot\fJ=h^*\fJ$
preserves~$\cJ_{\Si}^{\si}$; thus, the above quotient is well-defined. 
If $J\!\in\!\cJ_{\om}^{\phi}$,
the moduli space of marked real $J$-holomorphic maps in the class $B\in H_2(X;\Z)$ 
is the subspace
$$\fM_{g,l}(X,B;J)^{\phi,\si}=
\big\{[u,(z_1^+,z_1^-),\ldots,(z_l^+,z_l^-),\fJ]\!\in\!\cH_{g,l}(X, B)^{\phi,\si}\!:~
\dbar_{J,\fJ}u\!=\!0\big\},$$
where $\dbar_{J,\fJ}$ is the usual Cauchy-Riemann operator with respect 
to the complex structures $J$ on $X$ and $\fJ$ on $\Si$. 
If $g\!+\!l\!\ge\!2$, 
$$\cM_{g,l}^\si\equiv \fM_{g,l}(\pt,0)^{\id,\si}\equiv \cH_{g,l}(\pt, 0)^{\id,\si}$$
is the moduli space of marked symmetric domains. 
There is a natural forgetful morphism
$$\ff:\cH_{g,l}(X,B)^{\phi,\si}\lra  \cM_{g,l}^\si\,;$$
it drops the map component~$u$ from each element of the domain.\\

\noindent 
The following example shows that the orientability of a moduli space of symmetric half-surfaces 
does not imply the orientability of the corresponding moduli space of symmetric surfaces.
It indicates the subtle effect of diffeomorphisms of a symmetric surface~$(\Si,\si)$ 
not preserving any half-surface~$\Si^b$ and 
the difficulties arising in the standard approaches to the orientability problem~\ref{orient_it}
on page~\pageref{orient_it} in positive genus.

\begin{eg}\label{half_eg}
Let $\Si^b$ be an sh-surface  of genus~2 with one boundary component and non-trivial 
involution, as in the left diagram of Figure~\ref{fig:realspace}.
Its double is a symmetric surface~$(\Si,\si)$ of genus~4 without a fixed locus, 
as in the middle diagram of Figure~\ref{fig:realspace}.
The moduli space $\cM_{\Si^b}^{c}$ of sh-surfaces~$\Si^b$ 
is orientable by \cite[Lemma~6.1]{XCapsSetup} and \cite[Lemma~2.1]{XCapsSigns}. 
The natural automorphisms of $\cM_{\Si^b}^{c}$  associated with real orientation-reversing 
diffeomorphisms of~$\Si^b$ are orientation-preserving by \cite[Lemma 6.1]{XCapsSetup} 
and \cite[Corollary~2.3]{XCapsSigns}. 
On the double~$\Si$ of~$\Si^b$, 
these diffeomorphisms correspond to flipping the surface across the crosscap. 
The real moduli space~$\cM^\si_4$ parametrizing such symmetric surfaces~$\Si$ is not orientable
for the following reason.
By \cite[Theorem~1.2]{Nat}, every representative of a point in~$\cM^\si_4$ 
has 5~invariant circles which separate the surface, 
as in the right diagram of Figure~\ref{fig:realspace}.
There is a real diffeomorphism~$h$ which fixes~3 of these circles and interchanges the other~2. 
By \cite[Corollary~2.2]{XCapsSigns}, the mapping torus of~$h$ defines a loop in $\cM^{\si}_4$ 
which pairs non-trivially with the first Stiefel-Whitney class of the moduli space. 
\end{eg}

\begin{figure}
\begin{center}
\leavevmode
\includegraphics[width=0.7\textwidth,height=150px]{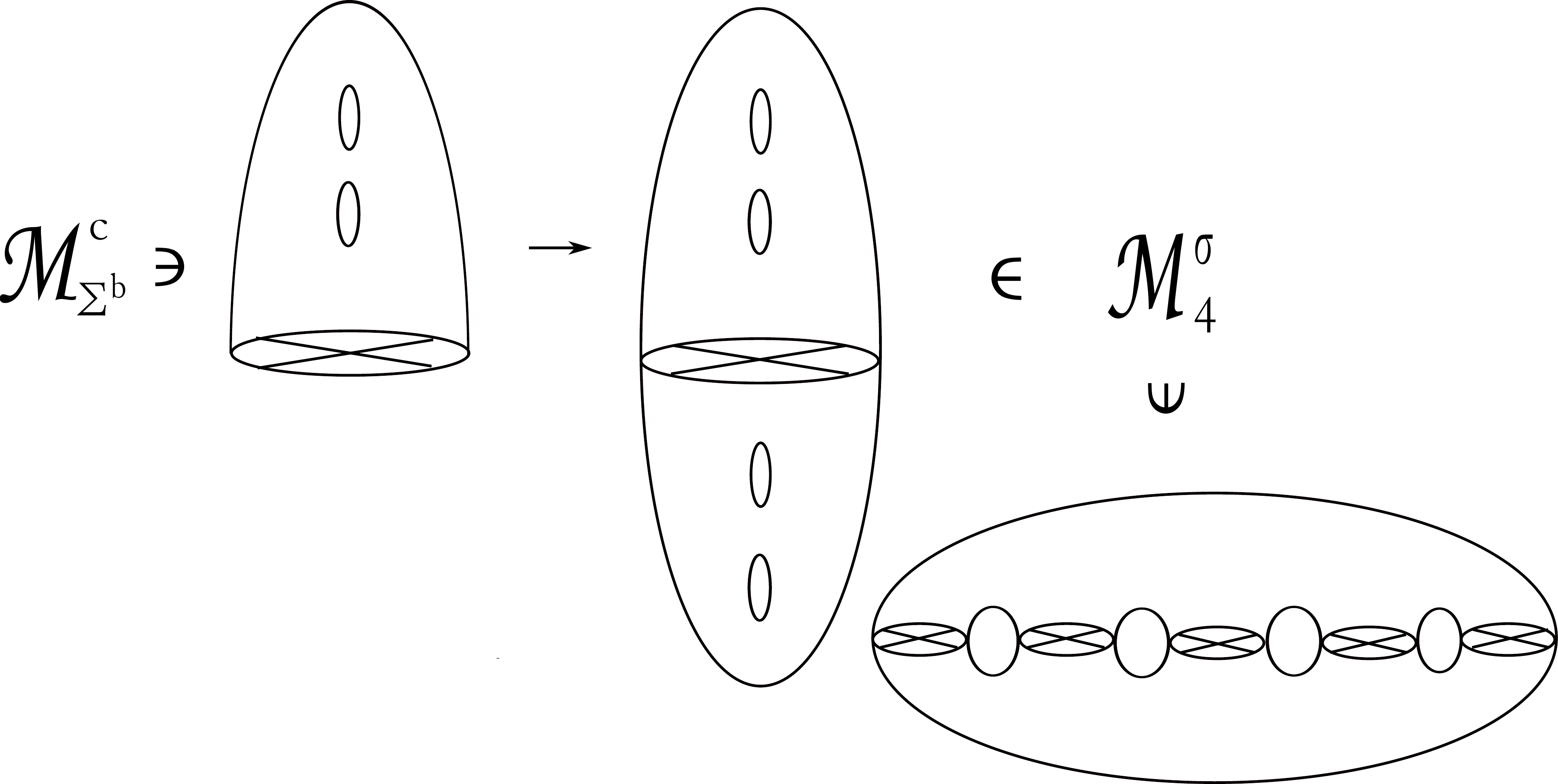}
\end{center}
\caption{Orientability of crosscaps vs.~real moduli spaces}
\label{fig:realspace}
\end{figure}


\subsection{Gromov's convergence topology}
\label{GrConv_subs}

\noindent
Let $\cC\!\equiv\!(\Si,z_1,\ldots,z_l,\fJ)$ be a compact nodal marked Riemann surface. 
A \sf{flat family of deformations} of~$\cC$ is a~tuple
\BE{CcCdfn_e}\big(\pi\!:\cU\!\lra\!\De,s_1\!:\De\!\lra\!\cU,\ldots,s_l\!:\De\!\lra\!\cU\big),\EE
where $\cU$ is a complex manifold, 
$\De\!\subset\!\C^N$ is a ball around~$0$, 
and $\pi,s_1,\ldots,s_l$ are holomorphic maps such~that
\begin{enumerate}[label=$\bullet$,leftmargin=*]

\item $\Si_{\t}\!\equiv\!\pi^{-1}(\t)$ is a (possibly nodal) 
Riemann surface for each $\t\!\in\!\De$
and $\pi$ is a submersion outside of the nodes of the fibers of~$\pi$,

\item for every $\t^*\!\equiv\!(t_1^*,\ldots,t_N^*)\!\in\!\De$ and 
every node $z^*\!\in\!\Si_{\t^*}$, 
there exist $i\!\in\!\{1,\ldots,N\}$ with $t_i^*\!=\!0$,
neighborhoods $\De_{\t^*}$ of~$\t^*$ in~$\De$ and $\cU_{z^*}$ of~$z^*$ in~$\cU$, 
and a holomorphic~map
$$\Psi\!:\cU_{z^*}\lra \big\{\big((t_1,\ldots,t_N),x,y\big)\!\in\!\De_{\t^*}\!\times\!\C^2\!:
xy\!=\!t_i\big\}$$
such that $\Psi$ is a homeomorphism onto a neighborhood of $(\t^*,0,0)$
and the composition of~$\Psi$ with the projection to~$\De_{\t^*}$ equals~$\pi|_{\cU_{z^*}}$,

\item $\pi\!\circ\!s_i\!=\!\id_{\De}$ and $s_i(\t)\!\neq\!s_j(\t)$
for all $\t\!\in\!\De$ and $i,j\!=\!1,\ldots,l$ with $i\!\neq\!j$,

\item $(\Si_0,s_1(0),\ldots,s_l(0))\!=\!\cC$.

\end{enumerate} 

\vspace{.1in}

\noindent
Let $\cC\!\equiv\!(\Si,\si,(z_1^+,z_1^-),\ldots,(z_l^+,z_l^-),\fJ)$ be 
a nodal marked symmetric Riemann surface.
A \sf{flat family of deformations} of~$\cC$ is a tuple 
$$\big(\pi\!:\cU\!\lra\!\De,\wt\fc\!:\cU\!\lra\!\cU,
s_1\!:\De\!\lra\!\cU,\ldots,s_l\!:\De\!\lra\!\cU\big)$$
such that
$(\pi,s_1,\wt\fc\!\circ\!s_1,\ldots,s_l,\wt\fc\!\circ\!s_l)$
is a flat family of deformations of $(\Si,(z_1^+,z_1^-),\ldots,(z_l^+,z_l^-),\fJ)$
and $\wt\fc$ is an anti-holomorphic involution on~$\cU$ lifting 
the standard involution~$\fc$ on~$\De$ and restricting to~$\si$ over
$\Si\!=\!\pi^{-1}(0)$.
In such a case, let $\si_{\t}\!=\!\wt\fc|_{\Si_{\t}}$ for each
parameter~$\t$ in $\De_{\R}\!\equiv\!\De\!\cap\!\R^N$.\\

\noindent
For any nodal surface~$\Si$, we denote by $\Si^*\!\subset\!\Si$ 
the subset of its smooth points.
Suppose $\pi\!:\cU\!\lra\!\De$ is a flat family of deformations of $(\Si,\fJ)$.
There then exist a neighborhood $\De'$ of~0 in~$\De$ and
a continuous collapsing map 
$$q\!:\cU\big|_{\De'} \lra \Si$$
so that the preimage of each node of $\Si$ under the restriction of~$q$ to~$\Si_{\t}$
with $\t\!\in\!\De'$ is either a node of~$\Si_{\t}$ or an embedded circle and 
the map
$$\pi\!\times\!q\!: q^{-1}(\Si^*) \lra \De'\!\times\!\Si^* $$
is a diffeomorphism.
For each $\t\!\in\!\De'$, let 
\BE{psitdfn_e}\psi_{\t}\!:\Si^*\lra q^{-1}(\Si^*) \!\cap\!\Si_{\t}\EE
be the restriction of its inverse to $\t\!\times\!\Si^*$.
If $\t_r\!\in\!\De$ is a sequence converging to $0\!\in\!\De$
and \hbox{$u_r\!:\Si_{\t_r}\!\lra\!X$} is a sequence of continuous maps
that are smooth on~$\Si_{\t_r}^*$,
we say that the sequence~$u_r$ \sf{converges to} a smooth map $u\!:\!\Si^*\!\lra\!X$
{\sf u.c.s.} (\sf{uniformly on compact subsets})
if the sequence of maps
$$u_r\!\circ\!\psi_{\t_r}\!:\Si^*\lra X$$
converges to~$u$ uniformly in the $C^{\i}$-topology on compact  subsets of~$\Si^*$.
This notion is independent of the choices of~$\De'$ and 
trivialization of~$\pi|_{q^{-1}(\Si^*)}$.\\

\noindent
For a Riemannian metric~$g$ on~$X$ and an $L^p_1$-map $u\!:\Si\!\lra\!X$,
for some $p\!>\!2$, let 
$$E_g(u)\equiv\frac12\int_{\Si}|\tnd u|_g^2 \in\R^{\ge0}$$
denote the \sf{energy} of~$u$; 
this notion is independent of the choice of $\fJ$-compatible metric on~$\Si$.

\begin{dfn}[Gromov's Convergence]\label{GromConv_dfn2}
Suppose $(X,\phi)$ is a manifold with an involution, 
$g$ is a Riemannian metric on~$X$,
and $J_r$ is an almost complex structure on~$X$ for every 
\hbox{$r\!\in\!\Z^{\ge0}\!\sqcup\!\{\i\}$}.
A sequence $(\cC_r,\si_r,u_r)$ of $\phi$-real \hbox{$J_r$-holomorphic}  maps
with $l$~conjugate pairs of marked points  
\sf{converges} to a $\phi$-real $J_{\i}$-holomorphic map $(\cC_{\i},\si_{\i},u_{\i})$
with $l$~conjugate pairs of marked points  
if $E_g(u_r)\!\lra\!E_g(u_{\i})$ as $r\!\lra\!\i$ and there exist
\begin{enumerate}[label=(\alph*),leftmargin=*]

\item a flat family $(\pi,\wt\fc,s_1,\ldots,s_l)$ of deformations of~$(\cC_{\i},\si_{\i})$
as above,

\item a sequence $\t_r\!\in\!\De_{\R}$ converging to $0\!\in\!\De$, and

\item equivalences $h_r\!:(\Si_{\t_r},\si_r)\!\lra\!(\cC_r,\si_{\t_r})$

\end{enumerate}
such that $u_r\!\circ\!h_r$ converges to~$u_{\i}|_{\Si_{\i}^*}$ u.c.s.
\end{dfn}

\noindent
Suppose $(X,\phi)$, $g$, and $J_r$ are as in Definition~\ref{GromConv_dfn2},
$X$ is compact, and the sequence $J_r$ converges to~$J_{\i}$ with respect 
to the $C^2$-topology.
Gromov's Compactness Theorem for $J$-holomorphic maps, arising from~\cite{Gr},
then implies that every sequence of stable $\phi$-real $J_r$-holomorphic maps~$u_r$
with $l$~conjugate pairs of marked points so that $\liminf E_g(u_i)$
is finite contains a subsequence that converges in the sense of Definition~\ref{GromConv_dfn2}
to some stable $\phi$-real $J_{\i}$-holomorphic map $(\cC_{\i},u_{\i})$
with $l$~conjugate pairs of marked points.
By the compactness of~$\Si_{\i}$, this notion of convergence
is independent of the choice of metric~$g$ on~$X$.

\subsection{Determinant lines of Fredholm operators}
\label{DetLB_subs}

\noindent
Let $(V,\vph)$ be a real bundle pair over a symmetric surface~$(\Si,\si)$.
A \textsf{real Cauchy-Riemann} (or \sf{CR-}) \sf{operator} on~$(V,\vph)$  is a linear map of the~form
\BE{CRdfn_e}\begin{split}
D=\bp\!+\!A\!: \Ga(\Si;V)^{\vph}
\equiv&\big\{\xi\!\in\!\Ga(\Si;V)\!:\,\xi\!\circ\!\si\!=\!\vph\!\circ\!\xi\big\}\\
&\hspace{.1in}\lra
\Ga_{\fJ}^{0,1}(\Si;V)^{\vph}\equiv
\big\{\ze\!\in\!\Ga(\Si;(T^*\Si,\fJ)^{0,1}\!\otimes_{\C}\!V)\!:\,
\ze\!\circ\!\tnd\si=\vph\!\circ\!\ze\big\},
\end{split}\EE
where $\bp$ is the holomorphic $\bp$-operator for some $\fJ\!\in\!\cJ_{\Si}^{\si}$
and a holomorphic structure in~$V$ reversed by~$\vph$
and  
$$A\in\Ga\big(\Si;\Hom_{\R}(V,(T^*\Si,\fJ)^{0,1}\!\otimes_{\C}\!V) \big)^{\vph}$$ 
is a zeroth-order deformation term. 
A real CR-operator on a real bundle pair is Fredholm in the appropriate completions.
The space of completions of all real CR-operators on~$(V,\vph)$ is contractible 
with respect to the operator~norm.\\

\noindent
If $X,Y$ are Banach spaces and $D\!:X\!\lra\!Y$ is a Fredholm operator, let
$$\det D\equiv\La_{\R}^{\top}(\ker D) \otimes \big(\La^{\top}_{\R}(\text{cok}\,D)\big)^*$$
denote the \textsf{determinant line} of~$D$. 
A continuous family of such Fredholm operators~$D_t$ over a topological space~$\cH$  
determines a line bundle over~$\cH$, called \sf{the determinant line bundle of~$\{D_t\}$}
and denoted $\det D$; see \cite[Section~A.2]{MS} and~\cite{detLB}. 
Combined with the note at the end of the previous paragraph, 
this implies that there is a canonical homotopy class of isomorphisms 
between any two CR-operators on a real bundle pair~$(V,\vph)$;
we thus denote any such operator by $D_{(V,\vph)}$.\\

\noindent
An \sf{exact triple} (short exact sequence)~$\ft$ of Fredholm operators
\BE{cTdiag_e}\begin{split}
\xymatrix{ 0 \ar[r]&  X'\ar[d]^{D'}\ar[r]& X\ar[d]^D\ar[r]&  X''\ar[d]^{D''}\ar[r]& 0\\
0 \ar[r]&  Y'\ar[r]& Y\ar[r]&  Y''\ar[r]& 0}
\end{split}\EE
determines a canonical isomorphism
\BE{sum} \Psi_{\ft}\!: (\det D')\otimes (\det D'')\stackrel{\approx}{\lra} \det D.\EE
For a continuous family of exact triples of Fredholm operators, 
the isomorphisms~\eref{sum} give rise to a canonical isomorphism
between the determinant line bundles.\\

\noindent
Let $(\Si_0,\si_0,\fJ_0)$ be a symmetric Riemann surface and 
$(\pi\!:\cU\!\lra\!\De,\wt\fc\!:\cU\!\lra\!\cU)$ be a flat family of deformations
of $(\Si_0,\si_0,\fJ_0)$ as in Section~\ref{GrConv_subs}. 
Suppose $(V,\vph)$ is a real bundle pair over~$(\cU,\wt\fc)$,
$\na$ is a $\vph$-compatible (complex-linear) connection in~$V$, and 
\BE{AcUdfn_e} A\in\Ga\big(\cU;\Hom_{\R}(V,(T^*\cU,J)^{0,1}\otimes_{\C}\!V)\big)^{\vph},\EE
where $J$ is the complex structure on~$\cU$.
The restrictions of~$\na$ and~$A$ to each fiber~$(\Si_{\t},\si_{\t})$
of~$\pi$ with $\t\!\in\!\De_{\R}$ then determine a real CR-operator
\BE{DVvphdfn_e}D_{(V,\vph);\t}\!: \Ga\big(\Si_{\t};V|_{\Si_{\t}}\big)^{\vph}\lra
\Ga_{\fJ_{\t}}^{0,1}\big(\Si_{\t};V|_{\Si_{\t}}\big)^{\vph}\EE
on $(V,\vph)|_{\Si_{\t}}$.
Let
\BE{detLB_e}\pi_{(V,\vph)}\!:\det D_{(V,\vph)}\!\equiv\!
\bigsqcup_{\t\in\De_{\R}}\!\!\!
\big(\{\t\}\!\times\!\det D_{(V,\vph);\t}\big)\lra\De_{\R}\,.\EE
The set $\det D_{(V,\vph)}$ carries natural topologies so that the projection
$\pi_{(V,\vph)}$ is a real line bundle; see Appendix~\ref{detLB_app}.
The case of~\eref{detLB_e} with $\De\!\subset\!\C$ 
(and thus $\De_{\R}$ is an open subset of~$\R$) and 
$(\Si_0,\si_0)$ having only a conjugate pair of nodes
underpins all
orienting constructions in the open GW-theory and Fukaya category literature
that follow \cite[Section~8.1]{FOOO}.\\

\noindent
Good topologies on the total space of~\eref{detLB_e} 
arise directly from some of the analytic considerations of~\cite{LT} combined with 
the algebraic conclusions of~\cite{KM}.
This implies that the resulting topologies satisfy analogues of all properties
listed in \cite[Section~2]{detLB}.
In particular,
\begin{enumerate}[label=(D\arabic*),leftmargin=*]

\item\label{dNatI_it} a homotopy class of continuous isomorphisms 
$\Psi\!:(V_1,\vph_1)\!\lra\!(V_2,\vph_2)$ of real bundle pairs over $(\cU,\wt\fc)|_{\De_{\R}}$
determines a homotopy class of isomorphisms
$$\det D_{\Psi}\!:\det D_{(V_1,\vph_1)}\lra \det D_{(V_2,\vph_2)}$$
of line bundles over $\De_{\R}$;

\item\label{dses_it} the isomorphisms \eref{sum} determine a homotopy class of isomorphisms
$$\det\!\big(\!D_{(V_1,\vph_1)\oplus(V_2,\vph_2)}\!\big)\approx
\big(\det D_{(V_1,\vph_1)}\big)\otimes \big(\det D_{(V_2,\vph_2)}\!\big)$$
of line bundles over $\De_{\R}$ for all real bundle pairs 
$(V_1,\vph_1)$ and $(V_2,\vph_2)$ over $(\cU,\wt\fc)$.

\end{enumerate}
These two properties correspond to the Naturality~I and Direct Sum properties
in \cite[Section~2]{detLB}.\\

\noindent
Families of real CR-operators often arise by pulling back data from
a target manifold by smooth maps as follows. 
Suppose $(X,J,\phi)$ is an almost complex manifold with an anti-complex  involution
and $(V,\vph)$ is a real bundle pair over~$(X,\phi)$.
Let $\na$ be a $\vph$-compatible (complex-linear) connection in $V$ and 
$$A\in\Ga\big(X;\Hom_{\R}(V,(T^*X,J)^{0,1}\otimes_{\C}\!V)\big)^{\vph}.$$ 
For any real map $u\!:(\Si,\si)\!\lra\!(X,\phi)$ 
from a symmetric surface and $\fJ\!\in\!\cJ_{\Si}^{\si}$, 
let $\na^u$ denote the induced connection in $u^*V$ and
$$ A_{\fJ;u}=A\circ \prt_{\fJ} u\in\Ga(\Si;
\Hom_{\R}(u^*V,(T^*\Si,\fJ)^{0,1}\otimes_{\C}u^*V)\big)^{u^*\vph}.$$
The homomorphisms
$$\bp_u^\na =\frac{1}{2}(\na^u+\fI\circ\na^u\circ\fJ), \,\,
D_{(V,\vph);u}\equiv \bp_u^\na\!+\!A_{\fJ;u}\!: \Ga(\Si;u^*V)^{u^*\vph}\lra
\Ga^{0,1}_{\fJ}(\Si;u^*V)^{u^*\vph}$$
are real CR-operators on $u^*(V,\vph)\!\lra\!(\Si,\si)$
that form families of real CR-operators over families of maps.\\

\noindent 
For $g,l\!\in\!\Z^{\ge0}$ and $B\!\in\!H_2(X;\Z)$, let
$$\det D_{(V,\vph)}\lra \fB_{g,l}(X,B)^{\phi,\si}\!\times\!\cJ_{\Si}^{\si}$$ 
denote the determinant line bundle of the family of the CR-operators $D_{(V,\vph);(u,\fJ)}$
constructed as above.
This line bundle descends to an orbi-bundle
$$\det D_{(V,\vph)}\lra \cH_{g,l}(X,B)^{\phi,\si};$$
it is a line bundle over the open subspace of the base consisting of marked maps 
with no non-trivial automorphisms.\\

\noindent
Let $(\pi,\wt\fc,s_1,\ldots,s_l)$ be a flat family of deformations as in Section~\ref{GrConv_subs}.
A smooth real map \hbox{$F\!:\cU\!\lra\!X$} pulls back the connection~$\na$ and
the zeroth-order deformation term~$A$ on $(V,\vph)$ above to a connection $\na^F$ and 
a zeroth-order deformation term~$A^F$ on $F^*(V,\vph)$.
The latter in turn determine a line bundle
$$\pi_{F^*(V,\vph)}\!:\det D_{F^*(V,\vph)} \lra\De_{\R}$$
as in~\eref{detLB_e}, which we call $\det D_{(V,\vph)}$ 
when there is no ambiguity.

\begin{eg}\label{ex_tbdl}
Let $g,l\!\in\!\Z^{\ge0}$ with $g\!+\!l\!\ge\!2$.
The pair $(V,\vph)\!\equiv\!(\C,\fc)$ is a real bundle pair over $(\pt,\id)$. 
The induced families of the operators $\dbar_{\C;u}\!\equiv\!D_{(\C,\fc);u}$
over flat families of stable real genus~$g$  curves with $l$ conjugate pairs of marked points 
define a line bundle 
$$\det \dbar_{\C} \lra \R\ov\cM_{g,l}\,.$$
If $(X,\phi)$ is an almost complex manifold with anti-complex involution $\phi$ and 
$$(V,\vph)=(X\!\times\!\C,\phi\!\times\!\fc)\lra (X,\phi),$$
then there is a canonical isomorphism 
$$\det D_{(\C,\fc)}\approx\ff^*\big(\!\det\dbar_{\C}\big)$$
of line bundles over $\cH_{g,l}(X,B)^{\phi,\si}$.
\end{eg}

\section{Real orientations on real bundle pairs}
\label{signcomp_sec}

\noindent
The main stepping stone in our proof of Theorem~\ref{orient_thm} for
the uncompactified moduli space
$$\fM_{g,l}(X,B;J)^{\phi}\subset \ov\fM_{g,l}(X,B;J)^{\phi}$$
is Proposition~\ref{canonisom_prp} below.
By Corollary~\ref{canonisom_crl2a} of this proposition,
a real orientation on a rank~$n$ real bundle pair $(V,\vph)$
over a symmetric surface~$(\Si,\si)$ determines an orientation on
the \sf{relative determinant}
\BE{fDdfn_e}
\rdet\,D\equiv \big(\!\det D\big)\otimes\big(\!\det\dbar_{\Si;\C}\big)^{\otimes n}\EE
for every real CR-operator~$D$ on $(V,\vph)$, where  
$\dbar_{\Si;\C}$ is 
the standard real CR-operator on~$(\Si,\si)$ with values in~$(\C,\fc)$.

\begin{dfn}\label{realorient_dfn4}
Let $(X,\phi)$ be a topological space with an involution and 
$(V,\vph)$ be a real bundle pair over~$(X,\phi)$.
A \sf{real orientation} on~$(V,\vph)$ consists~of 
\begin{enumerate}[label=(RO\arabic*),leftmargin=*]

\item\label{LBP_it2} a rank~1 real bundle pair $(L,\wt\phi)$ over $(X,\phi)$ such that 
\BE{realorient_e4}
w_2(V^{\vph})=w_1(L^{\wt\phi})^2 \qquad\hbox{and}\qquad
\La_{\C}^{\top}(V,\vph)\approx(L,\wt\phi)^{\otimes 2},\EE

\item\label{isom_it2} a homotopy class~$[\psi]$ 
of isomorphisms of real bundle pairs in~\eref{realorient_e4}, and

\item\label{spin_it2} a spin structure~$\fs$ on the real vector bundle
$V^{\vph}\!\oplus\!2(L^*)^{\wt\phi^*}$ over~$X^{\phi}$
compatible with the orientation induced by~\ref{isom_it2}.\\ 

\end{enumerate}
\end{dfn}

\noindent
An isomorphism~$\Th$ in~\eref{realorient_e4} restricts to an isomorphism 
\BE{realorient2_e3}\La_{\R}^{\top}V^{\vph}\approx (L^{\wt\phi})^{\otimes2}\EE
of real line bundles over~$X^{\phi}$.
Since the vector bundles $(L^{\wt\phi})^{\otimes2}$ and $2(L^*)^{\wt\phi^*}$ are canonically oriented, 
$\Th$~determines orientations on $V^{\vph}$ and $V^{\vph}\!\oplus\! 2(L^*)^{\wt\phi^*}$.
We will call them \textsf{the orientations determined by~\ref{isom_it2}}
if $\Th$ lies in the chosen homotopy class.
An isomorphism~$\Th$ in~\eref{realorient_e4} also induces an isomorphism 
\BE{detInd_e}\begin{split}
\La_{\C}^{\top}\big(V\!\oplus\!2L^*,\vph\!\oplus\!2\wt\phi^*\big)
&\approx \La_{\C}^{\top}(V,\vph)\otimes(L^*,\wt\phi^*)^{\otimes 2}\\
&\approx (L,\wt\phi)^{\otimes 2}\otimes(L^*,\wt\phi^*)^{\otimes 2}
\approx \big(\Si\!\times\!\C,\si\!\times\!\fc\big),
\end{split}\EE
where the last isomorphism is the canonical pairing.
We will call the homotopy class of isomorphisms~\eref{detInd_e} induced by 
the isomorphisms~$\Th$ in~\ref{isom_it2} \textsf{the homotopy class determined by~\ref{isom_it2}}.

\begin{prp}\label{canonisom_prp}
Suppose $(\Si,\si)$ is a symmetric surface and  
$(V,\vph)$ is a rank~$n$ real bundle pair over $(\Si,\si)$.
A real orientation on $(V,\vph)$ as in Definition~\ref{realorient_dfn4}
determines a homotopy class of isomorphisms
\BE{realorient2_e2} \Psi\!: \big(V\!\oplus\!2L^*,\vph\!\oplus\!2\wt\phi^*\big)
\approx\big(\Si\!\times\!\C^{n+2},\si\!\times\!\fc\big)\EE
of real bundle pairs over $(\Si,\si)$.
An isomorphism~$\Psi$ belongs to this homotopy class if and only~if
the restriction of $\Psi$ to the real locus induces the chosen spin structure~\ref{spin_it2}  and 
the isomorphism 
\BE{realorient2_e2b}
\La_{\C}^{\top}\Psi\!: \La_{\C}^{\top}\big(V\!\oplus\!2L^*,\vph\!\oplus\!2\wt\phi^*\big)
\lra \La_{\C}^{\top}\big(\Si\!\times\!\C^{n+2},\si\!\times\!\fc\big)
=\big(\Si\!\times\!\C,\si\!\times\!\fc\big)\EE
lies in the homotopy class determined by~\ref{isom_it2}.
\end{prp}

\noindent
This proposition is proved in Section~\ref{ROisom_subs}
after some topological preliminaries concerning symmetric functions on symmetric surfaces
are established in Section~\ref{topolprelim_subs}.
Proposition~\ref{canonisom_prp} is applied to the orientability problem~\ref{orient_it}
 on page~\pageref{orient_it}
in Section~\ref{applic_subs}.

\subsection{Homotopies of functions from symmetric surfaces}
\label{topolprelim_subs}

\noindent
Let $(X,\phi)$ be a topological space with an involution.
For any Lie group~$G$ with a natural conjugation, such as $\C^*$, 
$\SL_n\C$, or $\GL_n\C$,
denote by $\cC(X,\phi;G)$ the topological group of continuous maps $f\!:X\!\lra\!G$
such that \hbox{$f(\phi(x))\!=\!\ov{f(x)}$} for all $x\!\in\!X$.
The restrictions of such functions to the fixed locus $X^{\phi}\!\subset\!X$ take values in
the real locus of~$G$, i.e.~$\R^*$, $\SL_n\R$, and $\GL_n\R$,
in the three examples.

\begin{lmm}\label{lmm_repar}
Let $(\Si,\si)$ be a symmetric surface with fixed components $\Si^{\si}_1,\ldots,\Si^{\si}_m$
and $n\!\in\!\Z^+$.
For every $i=1,\ldots,m$ and continuous map $\psi\!:\Si^{\si}_i\!\lra\!\GL_n\R$,
there exists $f\!\in\!\cC(\Si,\si;\GL_n\C)$ such~that
 $f|_{\Si^{\si}_i}=\psi$ and 
$f$ is the identity outside of an arbitrarily small neighborhood of $\Si^{\si}_i$.
The same statement holds with $\GL_n\R$ and $\GL_n\C$ replaced by 
$\SL_n\R$ and $\SL_n\C$, respectively.
\end{lmm}

\begin{proof}
Let $S^1\!\times\!(-2,2)\!\lra\!\Si$ be a parametrization of a neighborhood~$U$ 
of $\Si^\si_i$ such that $S^1\!\times\!0$ corresponds to~$\Si^\si_i$ and
$$\si(\th,t)=(\th,-t) \qquad\forall~(\th,t)\!\in\!S^1\!\times\!(-2,2).$$ 
Since the inclusion $\GL_n\R\!\lra\!\GL_n\C$ induces trivial homomorphisms 
from~$\pi_1$ of either component of~$\GL_n\R$ to $\pi_1(\GL_n\C)$, 
we can homotope~$\psi$ to the identity-valued constant map through 
maps $h_t\!:S^1\!\lra\!\GL_n\C$. 
We define~$f$ on~$U$~by   
$$f(\th,t)=
\begin{cases} 
h_t(\th), &\text{if}~ t\in[0,1];\\
I_n, &\text{if}~t\in[1,2);\\
\ov{h_{-t}(\th)},&\text{if}~t\in(-2,0];
\end{cases}$$
and extend it as the identity-valued constant map over $\Si\!-\!U$.
The same argument applies with $\GL_n\R$ and $\GL_n\C$ replaced by 
$\SL_n\R$ and $\SL_n\C$, respectively.
\end{proof}

\begin{lmm}\label{homotopextend_lmm}
Suppose $(\Si,\si)$ is a symmetric surface, 
$n\!\in\!\Z^+$, and $f\!\in\!\cC(\Si,\si;\SL_n\C)$.
If 
$$f|_{\Si^{\si}}\!: \Si^{\si} \lra \SL_n\R$$
is homotopic to a constant map, 
then $f$ is homotopic to the constant map~$\Id$ through maps 
\hbox{$f_t\!\in\!\cC(\Si,\si;\SL_n\C)$}.
\end{lmm}

\begin{proof}
Let $(\Si^b,c)$ be an oriented sh-surface which doubles to~$(\Si,\si)$. 
By assumption, 
$$f|_{(\prt\Si^b)_i}\!:(\prt\Si^b)_i \lra \SL_n\C$$
is  homotopic to~$\Id$ through maps $f_t\!\in\!\cC((\prt\Si^b)_i,c;\SL_n\C)$
on each boundary component $(\prt\Si^b)_i$ of $\Si^b$ with $|c_i|\!=\!0$.
Since $f\!\in\!\cC(\Si,\si;\SL_n\C)$, this is also the case for $f|_{(\prt\Si^b)_i}$
for each boundary 
component $(\prt\Si^b)_i$ of $\Si^b$ with $|c_i|\!=\!1$;
see \cite[Lemma~2.4]{Teh}.\\

\noindent
A homotopy~$f_t$ as above extends over $\Si^b$ as follows. 
Suppose $f_0\!=\!f|_{\prt\Si^b}$ and~$f_1\!=\!\Id$.
Let \hbox{$\bI\!=\![0,1]$} and $(\prt\Si^b)\!\times\!\bI\!\lra\!U$ be a parametrization of 
a (closed) neighborhood~$U$  of  $\prt\Si^b\!\subset\!\Si^b$ with coordinates~$(w,s)$. 
Define
$$G_t:\Si^b\lra\SL_n\C \qquad\hbox{by}\quad
G_t(z)=\begin{cases}
f_{(1-s)t}(w)\cdot f^{-1}(w),&\text{if}~ z=(w,s)\in U\approx (\prt\Si^b)\!\times\!\bI;\\
I_n,& \text{if}~ z\in \Si^b\!-\!U.
\end{cases}$$
Since $G_t(w,1)\!=\!I_n$ for all~$t$, this map is continuous. 
Moreover, $G_0(z)\!=\!I_n$ for all $z\!\in\!\Si^b$ and 
$$G_t(w, 0)=f_{t}(w)\cdot f^{-1}(w)$$ 
is a homotopy between $\Id$ and $f^{-1}$. Thus, $H_t=G_t\cdot f$ is a homotopy 
over $\Si^b$ extending~$f_t$.\\

\begin{figure}
\begin{center}
\advance\topskip-10cm
\includegraphics[width=0.3\textwidth,height=150px]{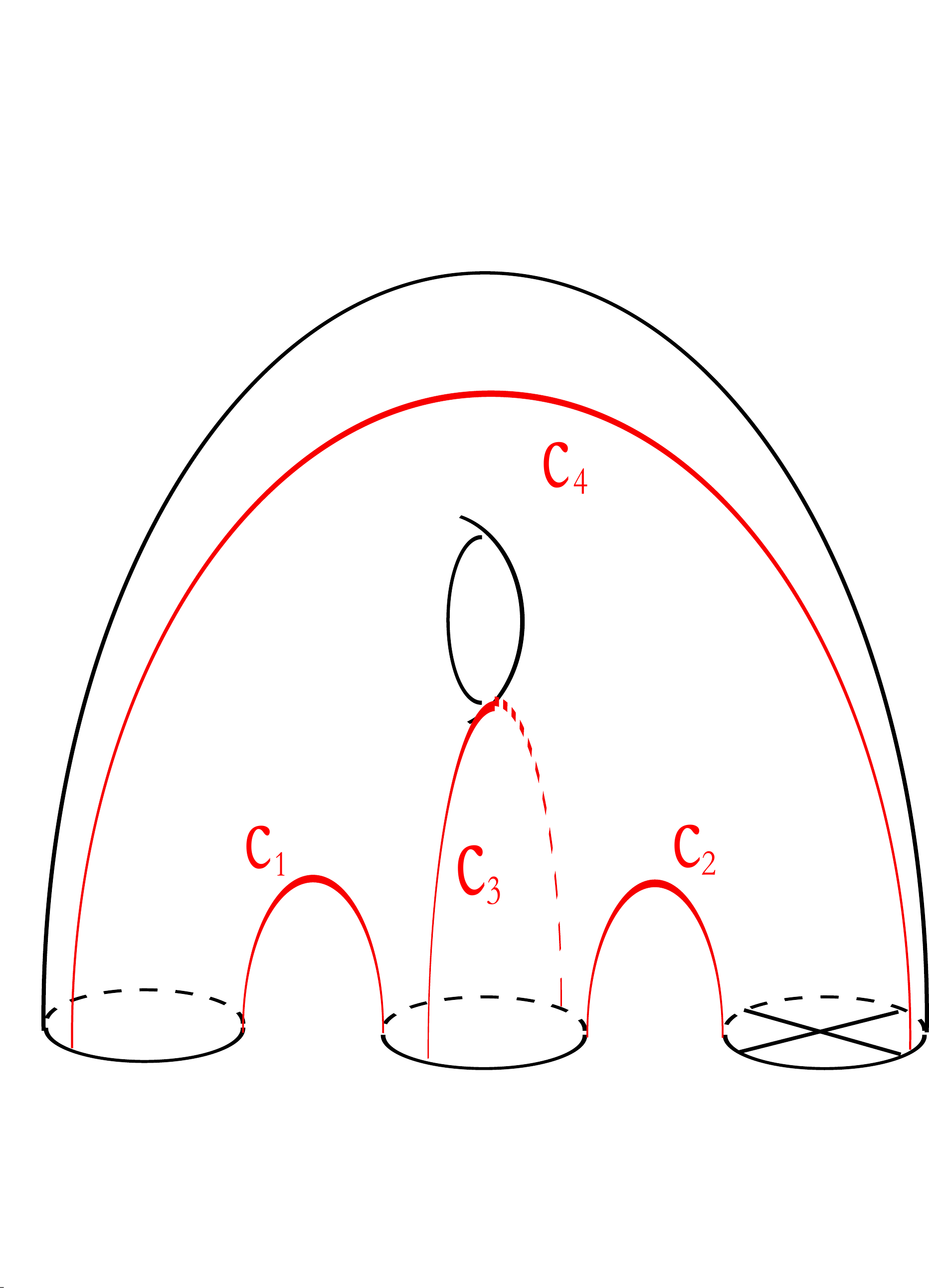}
\end{center}
\caption{The paths $C_1,\ldots,C_4$ cut $\Si^b$ to a disk.}
\label{arcs_fig}
\end{figure}

\noindent
By the previous paragraph, we may assume that $f$ is the constant map~$\Id$ on~$\prt\Si^b$. 
Choose embedded non-intersecting paths~$\{C_i\}$ in $\Si^b$ with endpoints on~$\prt\Si^b$ 
which cut $\Si^b$ into a disk~$D^2$; see Figure~\ref{arcs_fig}.
The restriction of~$f$ to each~$C_i$ defines an element~of 
$$\pi_1\big(\SL_n\C,I_n\big)\approx\pi_1\big(\SU_n,I_n\big)=0.$$
Thus, we can homotope~$f$ to~$\Id$ over~$C_i$ while keeping it fixed at the endpoints. 
Similarly to the previous paragraph, this homotopy extends over~$\Si^b$ without 
changing~$f$ over~$\prt\Si^b$ or over~$C_j$ for any $C_j\!\neq\!C_i$.
Thus, we may assume that $f$ is the constant map~$\Id$ over the boundary of~$D^2$. 
Since 
$$\pi_2\big(\SL_n\C,I_n\big)\approx\pi_2\big(\SU_n,I_n\big)=0,$$ 
the map $f\!:(D^2,S^1)\!\lra\!(\SL_n\C,I_n)$
can be homotoped to~$\Id$ as a relative map.
Doubling such a homotopy~$f_t$  by the requirement that $f_t(\si(z))\!=\!\ov{f_t(z)}$ 
for all $z\!\in\!\Si$, we obtain the desired homotopy from $f$ to~$\Id$ over all of~$\Si$.
\end{proof}

\begin{crl}\label{RBPhomot_crl}
Let $(\Si,\si)$ be a symmetric surface and
$$\Phi,\Psi\!:(V,\vph)\lra \big(\Si\!\times\!\C^n,\si\!\times\!\fc\big)$$
be isomorphisms of real bundle pairs over $(\Si,\si)$.
If the isomorphisms 
\BE{RBPhomot_e2}\begin{split}
\Phi|_{V^{\vph}},\Psi|_{V^{\vph}}\!: V^{\vph}&\lra \Si\!\times\!\R^n, \\
\La_{\C}^{\top}\Phi,\La_{\C}^n\Psi\!:\La_{\C}^{\top}(V,\vph)
&\lra \La_{\C}^{\top}\big(\Si\!\times\!\C^n,\si\!\times\!\fc\big)
=\big(\Si\!\times\!\C,\si\!\times\!\fc\big)
\end{split}\EE
are homotopic, then so are the isomorphisms~$\Phi$ and~$\Psi$. 
\end{crl}

\begin{proof}
Let $f\!\in\!\cC(\Si,\si;\C^*)$ be given by
$$\La_{\C}^{\top}\Phi=f\,\La_{\C}^n\Psi\!: \La_{\C}^{\top}(V,\vph)\lra
\big(\Si\!\times\!\C,\si\!\times\!\fc\big).$$
Since the second pair of isomorphisms in~\eref{RBPhomot_e2} are homotopic,
there exists a path $f_t\!\in\!\cC(\Si,\si;\C^*)$ such that $f_0\!=\!1$ and~$f_1\!=\!f$.
Let 
$$\Psi_{f_t}\!:(V,\vph)\lra\big(\Si\!\times\!\C^{n},\si\!\times\!\fc\big)$$ 
be the composition of $\Psi$ with the real bundle~map
\BE{Psifdfn_e}\big(\Si\!\times\!\C^n,\si\!\times\!\fc\big)\lra 
\big(\Si\!\times\!\C^n,\si\!\times\!\fc\big),  \quad
(z,v)\lra \left( \begin{array}{cccc} f_t(z)& 0& \ldots& 0\\ 0& 1& \ldots& 0\\
\vdots& \vdots& \ddots& 0\\ 0& \ldots& 0& 1\end{array}\right)v.\EE
Thus, $\Psi_f\!=\!\Psi_{f_1}$ is homotopic to $\Psi$ and 
$\La_{\C}^{\top}\Phi\!=\!\La_{\C}^{\top}\Psi_f$.\\

\noindent
Let $F\!\in\!\cC(\Si,\si;\GL_n\C)$ be given~by
$$\Psi(v)=\big\{\id\!\times\!F(\pi(v))\big\}\big(\Phi(v)\big) \qquad\forall~v\!\in\!V,$$
where $\pi\!:V\!\lra\!\Si$ is the projection map.
By the previous paragraph, we can assume that \hbox{$F\!\in\!\cC(\Si,\si;\SL_n\C)$}.
Since the first pair of isomorphisms in~\eref{RBPhomot_e2} are homotopic,
$$F|_{\Si^{\si}_i}\!: \Si^{\si}_i \lra \SL_n\R$$
is homotopically trivial for  every component
$\Si^{\si}_i\!\subset\!\Si^{\si}$ of the fixed locus.
By Lemma~\ref{homotopextend_lmm},
$F$ is thus homotopic to the constant map~$\Id$ through 
elements $F_t\!\in\!\cC(\Si,\si;\SL_n\C)$.
This establishes the claim. 
\end{proof}

\subsection{Isomorphisms induced by real orientations}
\label{ROisom_subs}

\noindent
We now apply Lemma~\ref{lmm_repar} and Corollary~\ref{RBPhomot_crl} to establish 
Proposition~\ref{canonisom_prp}.
We then deduce some corollaries from this proposition.

\begin{proof}[{\bf\emph{Proof of Proposition~\ref{canonisom_prp}}}]
Let $\Si_1^{\si},\ldots\,\Si_m^{\si}\!\subset\!\Si^{\si}$ be the connected components 
of the fixed locus.
Since $c_1(V\!\oplus\!2L^*)\!=\!0$ and the vector bundle $V^{\vph}\!\oplus\!2(L^*)^{\wt\phi^*}$ 
is orientable, an isomorphism~$\Psi$ as in~\eref{realorient2_e2} exists; 
see \cite[Propositions~4.1,4.2]{BHH}.
For each $i\!=\!1,\ldots,m$, choose $\psi_i\!:\Si_i^{\si}\!\lra\!\GL_{n+2}\R$ so that
the composition of the restriction of~$\Psi$ to 
$(V^{\vph}\!\oplus\!2(L^*)^{\wt\phi^*})|_{\Si_i^{\si}}$ with the isomorphism
$$\Si_i^{\si}\!\times\!\R^{n+2}\lra \Si_i^{\si}\!\times\!\R^{n+2}, \qquad 
(z,v)\lra \big(z,\psi_i(z)v\big),$$
induces the chosen orientation and spin structure on 
$(V^{\vph}\!\oplus\!2(L^*)^{\wt\phi^*})|_{\Si_i^{\si}}$.
Let  $f_i\!:\Si\!\lra\!\GL_{n+2}\C$ be a continuous map as in Lemma~\ref{lmm_repar} 
corresponding to~$(i,\psi_i)$. 
The composition of the original isomorphism~$\Psi$ with the real~map
$$\big(\Si\!\times\!\C^{n+2},\si\!\times\!\fc\big)\lra 
\big(\Si\!\times\!\C^{n+2},\si\!\times\!\fc\big),  \qquad
(z,v)\lra \big(z,f_1(z)\!\ldots\!f_m(z)v\big),$$ 
is again an isomorphism of real bundle pairs as in~\eref{realorient2_e2}.\\

\noindent
By the previous paragraph, there exists an isomorphism~$\Psi$ as in~\eref{realorient2_e2}
that induces the chosen orientation and spin structure on 
$V^{\vph}\!\oplus\!2(L^*)^{\wt\phi^*}$.
It determines an isomorphism 
$$\La_{\C}^{\top}\big(V\!\oplus\!2L^*,\vph\!\oplus\!2\wt\phi^*\big)
\approx\La_{\C}^{\top}\big(\Si\!\times\!\C^{n+2},\si\!\times\!\fc\big)
= \big(\Si\!\times\!\C,\si\!\times\!\fc\big)$$
and thus an isomorphism $\La_{\C}^{\top}\Psi$ as in~\eref{realorient2_e2b}. 
If $\psi$ is the isomorphism in~\eref{realorient2_e2b} determined by
  an isomorphism in~\eref{realorient_e4} from the chosen homotopy class~\ref{isom_it2}, then
\BE{realorient2_e9}  \psi=f\La_{\C}^{\top}\Psi\EE
for some $f\!\in\!\cC(\Si,\si;\C^*)$.
Let 
$$\Psi_f\!:\big(V\!\oplus\!2L^*,\vph\!\oplus\!2\wt\phi^*\big)
\approx\big(\Si\!\times\!\C^{n+2},\si\!\times\!\fc\big)$$ 
be defined as in~\eref{Psifdfn_e}.
By~\eref{realorient2_e9}, $\La_{\C}^{\top}\Psi_f\!=\!\psi$. 
Since $\Psi$ and $\psi$ induce the same orientations on $V^{\vph}\!\oplus\!2(L^*)^{\wt\phi^*}$,
$f|_{\Si^{\si}}\!>\!0$.
Thus, $\Psi_f$ induces the same orientation and spin structure on 
$V^{\vph}\!\oplus\!2(L^*)^{\wt\phi^*}$ as~$\Psi$.\\

\noindent
We conclude that there exists an isomorphism~$\Psi$ as in~\eref{realorient2_e2}
inducing the chosen orientation and spin structure on $V^{\vph}\!\oplus\!2(L^*)^{\wt\phi^*}$ 
so that the isomorphism~$\La_{\C}^{\top}\Psi$ lies  
in the homotopy class of the isomorphisms~\eref{detInd_e} determined by~\ref{isom_it2}.
By Corollary~\ref{RBPhomot_crl}, any two such isomorphisms are homotopic.
\end{proof}

\begin{crl}\label{canonisom_crl}
Suppose $(\Si,\si)$ is a symmetric surface and  
$(L,\wt\phi)\!\lra\!(\Si,\si)$ is a rank~1 real bundle pair.
If $L^{\wt\phi}\!\lra\!\Si^{\si}$ is orientable, there exists a canonical homotopy 
class of isomorphisms 
\BE{canonisom_crl_e}\big(L^{\otimes2}\!\oplus\!2L^*,\wt\phi^{\otimes2}\!\oplus\!2\wt\phi^*\big)
\approx\big(\Si\!\times\!\C^3,\si\!\times\!\fc\big)\EE
of real bundle pairs over $(\Si,\si)$.
In general, an orientation of each component $\Si_i^{\si}$ of $\Si^{\si}$ 
such~that $L^{\wt\phi}|_{\Si_i^{\si}}$ is non-orientable determines a canonical homotopy 
class of isomorphisms~\eref{canonisom_crl_e};
changing an orientation of such a component~$\Si^{\si}_i$ changes 
the induced spin structure, but not the orientation, of the real part of LHS 
in~\eref{canonisom_crl_e} over~$\Si^{\si}_i$.
\end{crl}

\begin{proof} The line bundle $(L^{\wt\phi})^{\otimes2}$ is canonically oriented and 
thus has a canonical homotopy class of trivializations.
We apply Proposition~\ref{canonisom_prp} with $(V,\vph)\!=\!(L,\wt\phi)^{\otimes2}$.
There is then a canonical choice of isomorphism in~\eref{realorient_e4}. 
It induces the canonical orientations on the real parts of $2(L^*,\wt\phi^*)$
or of LHS in~\eref{canonisom_crl_e}. 
If $L^{\wt\phi}$ is orientable, an orientation on $L^{\wt\phi}$ determines a homotopy class
of trivializations of the real part of LHS in~\eref{canonisom_crl_e}.
The resulting spin structure is independent of the choice of the orientation.\\

\noindent
If the restriction of $L^{\wt\phi}$ to a component $\Si^{\si}_i\!\approx\!\R\P^1$ 
of the fixed locus $\Si^{\si}\!\subset\!\Si$ is not orientable, 
then $(L^*)^{\wt\phi^*}|_{\Si^{\si}_i}$ 
is isomorphic to the tautological line bundle
$$\ga\equiv\big\{\big(\ell,(x,y)\big)\!\in\!\R\P^1\!\times\!\R^2\!:\,(x,y)\!\in\!\ell\!\subset\!\R^2\big\}
\lra\R\P^1\,.$$
Combining this isomorphism with the trivialization
\BE{RBisom_e6}\ga\oplus\ga\lra \R\P^1\!\times\!\R^2, \qquad
\big(\ell,(x_1,y_1),(x_2,y_2)\big)\lra \big(\ell,(x_1\!-\!y_2,x_2\!+\!y_1)\big),\EE
we obtain an isomorphism
\BE{RBisom_e7} 2(L^*)^{\wt\phi^*}  \lra \R\P^1\!\times\!\R^2.\EE
It induces the canonical orientation on the domain.
The homotopy class of the isomorphism~\eref{RBisom_e7} does not depend
on the choice of isomorphism of $(L^*)^{\wt\phi^*}|_{\Si^{\si}_i}$ with~$\ga$,
once an identification of $\Si^{\si}_i$ with $\R\P^1$ is fixed.
However, it does depend on the orientation class of this identification
even after stabilization by the trivial line bundle,
as shown in the next paragraph.\\

\noindent
A bundle isomorphism $\ga\!\lra\!\ga$ covering an orientation-reversing map $\R\P^1\!\lra\!\R\P^1$
is given~by
$$\ga\lra\ga, \qquad \big([u,v],(x,y)\big)\lra \big([u,-v],(x,-y)\big).$$
The composition of this isomorphism with the isomorphism~\eref{RBisom_e6} is
the isomorphism
\BE{RBisom_e8}\ga\oplus\ga\lra \R\P^1\!\times\!\R^2, \qquad
\big(\ell,(x_1,y_1),(x_2,y_2)\big)\lra \big(\ell,(x_1\!+\!y_2,x_2\!-\!y_1)\big).\EE
Under the standard identification of $\R^2$ with~$\C$, $\R\P^1$ can be parametrized~as
$$S^1\lra \R\P^1, \qquad \ne^{\fI\th}\lra\big[\ne^{\fI\th/2}\big].$$
Under this identification, the isomorphisms~\eref{RBisom_e6} and~\eref{RBisom_e8} are given~by 
\begin{alignat*}{2}
\big(\ne^{\fI\th},a \ne^{\fI\th/2},b \ne^{\fI\th/2}\big)
&\lra \big(\ne^{\fI\th}, \ne^{\fI\th/2}(a\!+\!\fI b)\big) &\qquad&\forall~a,b\!\in\!\R,\\
\big(\ne^{\fI\th},a \ne^{\fI\th/2},b \ne^{\fI\th/2}\big)
&\lra\big(\ne^{\fI\th}, \ne^{-\fI\th/2}(a\!+\!\fI b)\big) &\qquad&\forall~a,b\!\in\!\R,
\end{alignat*}
respectively.
They differ by the map
$$S^1\lra\GL_2\R, \qquad \ne^{\fI\th}\lra \ne^{-\fI\th}.$$
Since this map generates $\pi_1(\GL_2\R)$, the trivializations of $\ga\!\oplus\!\ga$
in~\eref{RBisom_e6} and~\eref{RBisom_e8} are not homotopy equivalent,
even after stabilization by the trivial line bundle.
\end{proof}

\begin{crl}\label{canonisom_crl2a}
Suppose $(\Si,\si)$ is a symmetric surface and  $D\!=\!D_{(V,\vph)}$ is a real CR-operator 
on a rank~$n$ real bundle pair $(V,\vph)$ over~$(\Si,\si)$.
Then a real orientation on $(V,\vph)$ as in Definition~\ref{realorient_dfn4}
induces an  orientation on the relative determinant $\rdet\,D$ of~$D$ in~\eref{fDdfn_e}.
Changing a real orientation on $(V,\vph)$ by changing the spin structure~$\fs$ in~\ref{spin_it2}
over one component~$\Si^{\si}_i$ of~$\Si^{\si}$ reverses 
the orientation on~$\rdet\,D$.
\end{crl}

\begin{proof} Let $((L,\wt\phi),[\psi],\fs)$ be a real orientation on~$(V,\vph)$.
By~\eref{sum}, there is a canonical homotopy class of isomorphisms
$$\det D_{(V\oplus 2L^*,\vph\oplus 2\wt\phi^*)}
\approx\big(\!\det D_{(V,\vph)}\big)\otimes\big(\!\det D_{(L^*,\wt\phi^*)}\big)^{\otimes2}$$
of real lines,
where the subscripts indicate the real bundle pair associated with the corresponding
real CR-operator. 
Since the last factor above is canonically oriented, so is the~line
\BE{thm_maps_e5}
\big(\!\det D_{(V,\vph)}\big)\otimes
\big(\!\det D_{(V\oplus 2L^*,\vph\oplus 2\wt\phi^*)}\big)\,.\EE
By Proposition~\ref{canonisom_prp}, the real orientation on~$(V,\vph)$
determines a homotopy class of isomorphisms
$$\big(V\!\oplus\!2L^*,\vph\!\oplus\!2\wt\phi^*\big)\approx 
\big(\Si\!\times\!\C^{n+2},\si\!\times\!\fc\big).$$
By~\eref{sum}, the latter in turn determines an orientation on the~line 
$$\wh\det D_{(V\oplus 2L^*, \vph\oplus 2\wt\phi^*)}\big)\equiv
\big(\!\det D_{(V\oplus 2L^*, \vph\oplus 2\wt\phi^*)}\big)\otimes 
\big(\!\det\dbar_{\Si;\C}\big)^{\otimes (n+2)}\,.$$
Combining this with the canonical orientation of the line~\eref{thm_maps_e5},
we obtain an orientation on~$\rdet\,D$.\\

\noindent
Let $\fs_{\can}$ denote the canonical spin structure on $\Si\!\times\!\R^{n+2}$.
By Proposition~\ref{canonisom_prp}, 
the identity automorphism of $\La_{\C}^{\top}(\Si\!\times\!\C^{n+2})$
and a spin structure on \hbox{$\Si^{\si}\!\times\!\R^{n+2}$} 
determine a homotopy class of isomorphisms
\BE{canonisom_e19}\Psi\!:\big(\Si\!\times\!\C^{n+2},\si\!\times\!\fc\big)
\lra \big(\Si\!\times\!\C^{n+2},\si\!\times\!\fc\big)\EE
of real bundle pairs over $(\Si,\si)$.
The latter in turn determines a homotopy class of isomorphisms
\BE{canonisom_e20}
\det D_{\Psi}\!:
\big(\!\det\dbar_{\Si;\C}\big)^{\otimes(n+2)}=\det\!\big(\!(n\!+\!2)\dbar_{\Si;\C}\!\big)
\lra \det\!\big(\!(n\!+\!2)\dbar_{\Si;\C}\!\big)=\big(\!\det\dbar_{\Si;\C}\big)^{\otimes(n+2)}\,.\EE
For the purposes of the last claim of this corollary, it is sufficient to check that 
the last isomorphisms are orientation-reversing for the spin structure~$\fs_i$ on 
$\Si^{\si}\!\times\!\R^{n+2}$ which differs from~$\fs_{\can}$  on precisely
one component $\Si^{\si}_i$ of~$\Si^{\si}$.\\

\noindent
By Lemma~\ref{lmm_repar}, we can assume that the map~$\Psi$ in~\eref{canonisom_e19}
is the identity outside of a tubular neighborhood $U\!\subset\!\Si$ of~$\Si^{\si}_i$
with $\ov{U}\!\subset\!\Si$ disjoint from \hbox{$\Si^{\si}\!-\!\Si^{\si}_i$}.
Pinching each of the two components of~$\prt\ov{U}$, we obtain
a nodal symmetric surface~$(\Si_0,\si_0)$ consisting of $(\P^1,\tau)$ 
and a smooth symmetric, possibly disconnected, surface $(\Si',\si')$
which share a pair of conjugate points.
We can choose a flat family 
$$\big(\pi\!:\cU\!\lra\!\De,\fc\!:\cU\!\lra\!\cU\big)$$
of deformations of~$(\Si_0,\si_0)$ as in Section~\ref{GrConv_subs} with $\De\!\subset\!\C^2$
and \hbox{$(\Si_{\t^*},\si_{\t^*})\!=\!(\Si,\si)$} for some \hbox{$\t^*\!\in\!\De_{\R}$} and 
a continuous collapsing map \hbox{$q\!:\cU\!\lra\!\Si_0$} so that
$q^{-1}(\Si_0^*)\!=\!\Si\!-\!\prt\ov{U}$ and the~map
$$\pi\!\times\!q\!: q^{-1}(\Si^*) \lra \De\!\times\!\Si^* $$
is a diffeomorphism.
The isomorphism~$\Psi$ in~\eref{canonisom_e19} then induces an isomorphism
$$ \wt\Psi\!:\big(\cU|_{\De_{\R}}\!\times\!\C^{n+2},\wt\fc\!\times\!\fc\big)
\lra \big(\cU|_{\De_{\R}}\!\times\!\C^{n+2},\wt\fc\!\times\!\fc\big)$$
which restricts to~$\Psi$ over $(\Si_{\t^*},\si_{\t^*})$ and to
the identity outside of $q^{-1}(\P^1)$.
Denote by
$$\Psi_0\!: \big(\Si_0\!\times\!\C^{n+2},\si_0\!\times\!\fc\big)
\lra \big(\Si_0\!\times\!\C^{n+2},\si_0\!\times\!\fc\big)$$
the restriction of~$\wt\Psi$ over $(\Si_0,\si_0)$ and by
\BE{canonisom_e20b}\det D_{\Psi_0}\!:\det\!\big(\!(n\!+\!2)\dbar_{\Si_0;\C}\!\big)
\lra \det\!\big(\!(n\!+\!2)\dbar_{\Si_0;\C}\!\big)\EE
the homotopy class of isomorphisms induced by $\Psi_0$.\\

\noindent
By~\ref{dNatI_it} on page~\pageref{dNatI_it}, $\wt\Psi$ 
determines a homotopy class of isomorphisms
$$\det D_{\wt\Psi}\!:
\det\!\big(\!(n\!+\!2)\dbar_{\cU|_{\De_{\R}};\C}\!\big)
\lra \det\!\big(\!(n\!+\!2)\dbar_{\cU|_{\De_{\R}};\C}\!\big)$$
of determinant line bundles over~$\De_{\R}$.
Since $\det D_{\wt\Psi}$ restricts to $\det D_{\Psi}$ over $(\Si_{\t^*},\si_{\t^*})$
and to $\det D_{\Psi_0}$ over $(\Si_0,\si_0)$,
the isomorphisms~\eref{canonisom_e20} are orientation-reversing if and only if 
the isomorphisms $\det D_{\Psi_0}$ are orientation-reversing.
The latter correspond to the tensor products of isomorphisms 
$\det D_{\Psi_0|_{\Si'}}$ for $(\Si',\si')$ and
$\det D_{\Psi_0|_{\P^1}}$ for $(\P^1,\tau)$.
The isomorphisms $\det D_{\Psi_0|_{\Si'}}$ are the identity.
Since the isomorphisms~$\Psi_0$ reverse the spin structure on the fixed locus $q(\Si^{\si}_i)$
of $(\P^1,\tau)$, the isomorphisms $\det D_{\Psi_0|_{\P^1}}$
are orientation-reversing;
see \cite[Proposition~8.1.7]{FOOO}.
We conclude that the isomorphisms~\eref{canonisom_e20b} and thus~\eref{canonisom_e20}
are orientation-reversing.
\end{proof}

\begin{crl}\label{canonisom_crl2}
Suppose $(X,J,\phi)$ is an almost complex manifold with an anti-complex involution and
$(V,\vph)$ is a rank~$n$ real bundle pair over~$(X,\phi)$.
Let $B\!\in\!H_2(X;\Z)$, $g,l\!\in\!\Z^{\ge0}$, and $(\Si,\si)$ be a genus~$g$ symmetric surface.
Then a real orientation on $(V,\vph)$ as in Definition~\ref{realorient_dfn4}
induces an  orientation on the line bundle 
\BE{canonisom_crl2_e}
\rdet\,D_{(V,\vph)}\equiv
\big(\!\det D_{(V,\vph)}\big)\otimes 
\big(\!\det\dbar_{\C}\big)^{\otimes n}\lra \cH_{g,l}(X,B)^{\phi,\si}\,.\EE
\end{crl}

\begin{proof}
By Corollary~\ref{canonisom_crl2a} applied with the real bundle pairs $u^*(V,\vph)$
and $u^*(L,\wt\phi)$ over~$(\Si,\si)$,
a real orientation on $(V,\vph)$ determines an orientation
on the fiber of the line bundle 
$$\big(\!\det D_{(V,\vph)}\big)\otimes 
\big(\!\det\dbar_{\C}\big)^{\otimes n}\lra \fB_g(X,B)^{\phi,\si}\!\times\!\cJ_{\Si}^{\si}$$
over each point $(u,\fJ)$ which varies continuously with~$(u,\fJ)$.
Since the resulting orientation on this line bundle is completely determined by 
the chosen real orientation on $(V,\vph)$ via the isomorphisms~\eref{realorient2_e2}, 
it descends to the quotient~\eref{canonisom_crl2_e}.
\end{proof}

\subsection{The orientability of uncompactified moduli spaces}
\label{applic_subs}

\noindent
We will now apply Proposition~\ref{canonisom_prp} to study 
the orientability of the uncompactified real moduli spaces in Theorem~\ref{orient_thm}.
We first consider the case $X\!=\!\pt$ and then use it
to establish  the restriction of Theorem~\ref{orient_thm}
to the main stratum $\fM_{g,l}(X,B;J)^{\phi}$ of $\ov\fM_{g,l}(X,B;J)^{\phi}$. 

\begin{prp}\label{DM_prp} 
Let $g,l\in \Z^{\ge0}$ be such that $g\!+\!l\!\ge\!2$.
For every genus~$g$ type~$\si$ of orientation-reversing involutions,
the line bundle 
\BE{CidentDM_e}\La_{\R}^{\top}\big(T\cM_{g,l}^{\si}\big)
\otimes\big(\!\det\dbar_{\C}\big)\lra \cM_{g,l}^{\si}\EE
is canonically oriented.
The interchanges of pairs of conjugate points and the forgetful morphisms preserve 
this orientation;
the interchange of the points within a conjugate pair reverses this orientation.
\end{prp}

\begin{proof}
The cardinality of the automorphism group is an upper semi-continuous function 
on the compact moduli space~$\ov\cM_{g,l}^{\si}$.
Thus, there exists $l(g)\!\in\!\Z^+$
so that for every $l\!\ge\!l(g)$ every element
$$[\cC]\equiv \big[\Si,(z_1^+,z_1^-),\ldots,(z_l^+,z_l^-),\fJ\big]\in \cM^{\si}_{g,l}$$
has no automorphisms.
We first establish the proposition under the assumption that $l\!\ge\!l(g)$.\\

\noindent
Let $\cT\!\lra\!\cU_{g,l}^{\si}$ denote the vertical tangent bundle over 
the universal curve for~$\cM_{g,l}^{\si}$.
For~$\cC$ as above, let
\BE{DMprp_e1a}
T\cC=T\Si\big(\!-\!z_1^+\!-\!z_1^-\!-\!\ldots\!-\!z_l^+\!-\!z_l^-\big)
\quad\hbox{and}\quad
T^*\cC=T^*\Si\big(z_1^+\!+\!z_1^-\!+\!\ldots\!+\!z_l^+\!+\!z_l^-\big),\EE
be the twisted holomorphic line bundles associated to 
the sheaves of holomorphic tangent vector fields vanishing at the marked points
and of meromorphic one-forms with at most simple poles at the marked points 
and holomorphic everywhere else.
We construct these line bundles using holomorphic identifications of 
small neighborhoods of~$z_i^+$ and~$z_i^-$ interchanged by~$\si$.
The involutions~$\tnd\si$ on~$T\Si$ and~$\tnd\si^*$ on~$T^*\Si$ then
induce involutions on~$T\cC$ and~$T^*\cC$;
we denote the induced involutions also by~$\tnd\si$ and~$\tnd\si^*$.\\

\noindent
Let $S\cC^+$ and $S\cC^-$ be the skyscraper sheaves over~$\Si$ given~by 
$$S\cC^+=T^*\Si|_{z_1^++\ldots+z_l^+}, \qquad S\cC^-=T^*\Si|_{z_1^-+\ldots+z_l^-}\,.$$
The projection
\BE{H0ScC_e}\pi_1\!:  \wch{H}^0(\Si;S\cC^+\!\oplus\!S\cC^-)^{\si}=
\big(\wch{H}^0(\Si;S\cC^+)\!\oplus\!\wch{H}^0(\Si;S\cC^-)\big)^{\si} \lra \wch{H}^0(\Si;S\cC^+)\EE
is an isomorphism of real vector spaces.
We orient $\wch{H}^0(\Si;S\cC^+\!\oplus\!S\cC^-)^{\si}$ and its dual via the isomorphism
$$\pi_1^*\!:
 \wch{H}^0(\Si;S\cC^+)^*=T_{z_1^+}\Si\!\oplus\!\ldots\!\oplus\!T_{z_l^+}\Si
\lra \big(\wch{H}^0(\Si;S\cC^+\!\oplus\!S\cC^-)^{\si}\big)^*$$
from the complex orientations of $T_{z_1^+}\Si,\ldots,T_{z_l^+}\Si$.\\

\noindent
The Kodaira-Spencer (or \sf{KS}) map and the Dolbeault isomorphism 
provide canonical isomorphisms
\BE{DM_prp_e3} T_{[\cC]}\cM^{\si}_{g,l}\approx \wch{H}^1(\Si;T\cC)^{\si}
\approx H^1(\Si;T\cC)^{\si};\EE
see \cite[Section~3.1.2]{Melissa} and \cite[p151]{GH}.
By Serre Duality (or \sf{SD}), there is a canonical isomorphism
$$H^1(\Si;T\cC)\approx 
\big(H^0( \Si;T^*\cC\!\otimes\!T^*\Si)\big)^*;$$
see \cite[p153]{GH}.
Since $\si$ is orientation-reversing,
the real part of the SD pairing identifies the space of invariant sections
on one side with the space of anti-invariant sections on the other; 
the latter is isomorphic
to the space of invariant sections by multiplication by~$\fI$.
Thus, there is a canonical isomorphism
\BE{DM_prp_e5} 
H^1(\Si;T\cC)^{\si}\approx 
\big(H^0( \Si;T^*\cC\!\otimes\!T^*\Si)^{\si}\big)^*.\EE
Since the degree of the holomorphic line bundle $T\cC$ is negative,  
$$\La_{\R}^{\top}\big(H^0(\Si;T^*\cC\!\otimes\!T^*\Si)^{\si}\big)
=  \det\dbar_{(T^*\cC,\tnd\si^*)\otimes(T^*\Si,\tnd\si^*)}.$$
The long exact sequence in cohomology for the sequence 
\BE{FMses_e}0\lra T^*\Si\!\otimes\!T^*\Si \lra T^*\cC\!\otimes\!T^*\Si \lra S\cC^+\!\oplus\!S\cC^-
\lra  0\EE
and the chosen orientation on $\wch{H}^0(\Si;S\cC^+\!\oplus\!S\cC^-)^{\si}$ induce an orientation 
on the line
\BE{DM_prp_e7}\det\dbar_{(T^*\cC,\tnd\si^*)\otimes(T^*\Si,\tnd\si^*)}\otimes
\det\dbar_{(T^*\Si,\tnd\si^*)^{\otimes2}}\,.\EE
Thus, the real line bundle
\BE{DM_prp_e8}
\La^{\top}_{\R} \big(T\cM_{g,l}^{\si}\big)\otimes
\big(\!\det\dbar_{(\cT^*,\tnd\si^*)^{\otimes2}}\big)\lra   \cM_{g,l}^{\si}\EE
is canonically oriented.\\

\noindent
By Corollary~\ref{canonisom_crl} applied with $(L,\wt\phi)\!=\!(T^*\Si,\tnd\si^*)$, 
there is a canonical homotopy class of isomorphisms
$$\big(T^*\Si^{\otimes2}\!\oplus\!2\,T\Si,(\tnd\si^*)^{\otimes2}\!\oplus\!2\tnd\si\big)
\approx\big(\Si\!\times\!\C^3,\si\!\times\!\fc\big)$$
of real bundle pairs over~$(\Si,\si)$.
Since the determinants of $\dbar$-operators on the real bundle pairs $2(T\Si,\tnd\si)$
and $2(\Si\!\times\!\C^2,\si\!\times\!\fc)$ are canonically oriented,
so is the line bundle 
\BE{domains_e15}
\big(\!\det\dbar_{(\cT^*,\tnd\si^*)^{\otimes2}}\big)\otimes 
\big(\!\det\dbar_{\C}\big)\lra   \cM_{g,l}^{\si}.\EE
Combining this orientation with the canonical orientation for the line bundle~\eref{DM_prp_e8},
we obtain an orientation on the line bundle~\eref{CidentDM_e}.\\

\noindent
Since the interchanges of pairs of conjugate points and the forgetful morphisms preserve
the orientation of~\eref{H0ScC_e}, 
they also preserve the orientation on~\eref{CidentDM_e} constructed above.
Since the interchange of the points within a conjugate pair reverses 
the orientation of~\eref{H0ScC_e},
it also reverses the orientation  on~\eref{CidentDM_e}.\\

\noindent
For $l\!<\!l(g)$, we orient the line bundle~\eref{CidentDM_e} 
by downward induction 
from the orientation of~\eref{CidentDM_e} with~$l$ replaced by~$l\!+\!1$
and the orientation of the fibers of the forgetful morphism
\BE{thm_maps_e11a}\cM_{g,l+1}^{\si}\lra \cM_{g,l}^{\si}\EE
obtained from the complex orientation of $T_{z_{l+1}^+}\Si$. 
If the fixed locus~$\Si^{\si}$ of~$(\Si,\si)$ is separating, 
the fibers of this morphism are disconnected and differ by 
the interchange of the points in the last conjugate pair of points.
However, the induced orientation on~\eref{CidentDM_e} is still well-defined 
for the following reason.
By Proposition~\ref{DM_prp} with $l$ replaced by $l\!+\!1$, 
the interchange of the points within a conjugate pair reverses the orientation 
on the line bundle~\eref{CidentDM_e} with~$l$ replaced by~$l\!+\!1$.
In the case of the last conjugate pair of points,
such an interchange also reverses the orientation of the fibers of~\eref{thm_maps_e11a}.
Thus, it has no effect on the induced orientation on~\eref{CidentDM_e}.
\end{proof}

\begin{crl}\label{orient0_crl}
Theorem \ref{orient_thm} holds with $\ov\fM_{g,l}(X,B;J)^{\phi}$
replaced by $\fM_{g,l}(X,B;J)^{\phi,\si}$ for every genus~$g$ 
orientation-reversing involution~$\si$.
\end{crl}

\begin{proof}
We first assume that $g\!+\!l\!\ge\!2$ as in Proposition~\ref{DM_prp}.
The forgetful morphism $\ff$ induces a canonical isomorphism
\BE{thm_maps_e3}\La_{\R}^{\top}\big(T\fM_{g,l}(X,B;J)^{\phi,\si}\big)\approx
\big(\!\det D_{(TX,\tnd\phi)}\big)\otimes \ff^*\big(\La_{\R}^{\top}(T\cM_{g,l}^{\si})\big)
\lra \fM_{g,l}(X,B;J)^{\phi,\si}\EE
of real line bundles.
By Corollary~\ref{canonisom_crl2} applied with $(V,\vph)\!=\!(TX,\tnd\phi)$, 
a real orientation on $(X,\om,\phi)$ determines an orientation~on
\BE{thm_maps_e9} 
\rdet\,D_{(TX,\tnd\phi)}\equiv
\big(\!\det D_{(TX,\tnd\phi)}\big)\otimes 
\big(\!\det\dbar_{\C}\big)^{\otimes n}\lra 
\cH_{g,l}(X,B)^{\phi,\si}\,.\EE
Combining the canonical isomorphism~\eref{thm_maps_e3},
the canonical orientation of~\eref{CidentDM_e},
and the orientation of~\eref{thm_maps_e9} determined by the chosen real 
orientation on $(X,\om,\phi)$,
we obtain an orientation on the line bundle~\eref{orient_thm_e}
over $\fM_{g,l}(X,B;J)^{\phi,\si}$.\\

\noindent
If $g\!+\!l\!<\!2$, we orient the line bundle~\eref{orient_thm_e} 
from the orientation of~\eref{orient_thm_e} with~$l$ replaced by~$l\!+\!2$
and the orientation of the fibers of the forgetful morphism
\BE{thm_maps_e11}\fM_{g,l+2}(X,B;J)^{\phi,\si}\lra \fM_{g,l}(X,B;J)^{\phi,\si}\EE
obtained from the complex orientations of $T_{z_{l+1}^+}\Si$ and~$T_{z_{l+2}^+}\Si$. 
The induced orientation on~\eref{orient_thm_e} is still well-defined 
for the following reason.
By Proposition~\ref{DM_prp}, the interchange of the points within 
a conjugate pair reverses the orientation on the line bundle~\eref{CidentDM_e}
with~$l$ replaced by~$l\!+\!2$ and thus on the line bundle~\eref{orient_thm_e} 
with~$l$ replaced by~$l\!+\!2$.
In the case of the last two pairs of conjugate points,
such an interchange also reverses the orientation of
the fibers of~\eref{thm_maps_e11}.
Thus, it has no effect on the induced orientation on~\eref{orient_thm_e}.
\end{proof}

\noindent
Proposition~\ref{DM_prp} is also obtained in~\cite{Remi}; 
see Corollaires~1.2 and~1.1, Proposition~1.4, and Lemmes~1.3 and~1.4 in~\cite{Remi}.
A version of Corollary~\ref{orient0_crl} for certain covers of 
the uncompactified moduli spaces $\fM_{g,l}(X,B;J)^{\phi,\si}$ appears in~\cite{Remi} as well.
The orientability of these covers is obtained in~\cite{Remi} in a subset of cases
for which Corollary~\ref{orient0_crl} implies the orientability of 
the spaces $\fM_{g,l}(X,B;J)^{\phi,\si}$ themselves
(while Theorem~\ref{orient_thm} also yields the orientability of their compactifications).
For example, let $X_{n;\de}\!\subset\!\P^{n-1}$ denote a hypersurface of degree $\de\!\in\!\Z^+$ 
preserved by~$\tau_n$.
Corollary~\ref{orient0_crl} implies that $\fM_{g,l}(X_{n;\de},B;J)^{\tau_{n;\de},\si}$ is orientable~if
$$\de=0,1\mod\,4\qquad\hbox{and}\qquad \de\equiv n\mod2.$$
With the second condition strengthened to $\de\!\equiv\!n\mod4$,
this conclusion is obtained in \cite[Corollaire 2.4]{Remi}
under the additional assumption that $\Si^{\si}$ is a single circle.
If $\Si^{\si}$ consists of more than one circle, 
\cite[Corollaire~2.4]{Remi} shows that this conclusion
holds after pulling back to a cover of $\fM_{g,l}(X_{n;\de},B;J)^{\phi,\si}$.
The orientability of the compactified moduli spaces of real maps necessary for 
defining real GW-invariants is not considered in~\cite{Remi}.\\

\noindent
A canonical orientation on the real line $\rdet\,D$ in Corollary~\ref{canonisom_crl2a}
under overlapping topological assumptions is obtained in~\cite{Remi}
using a completely different approach.
We obtain it as an immediate consequence of the existence of a canonical homotopy class of
isomorphisms for the corresponding real bundle pairs.
The argument of~\cite{Remi} is heavily analytic in nature and is based
on explicit sign computations for certain automorphisms of determinant line bundles in~\cite{Remi0}.
In contrast, our proof is completely topological;
the proofs of the two statements from \cite{Teh} and~\cite{BHH} cited 
in the proofs of Lemma~\ref{homotopextend_lmm} and Proposition~\ref{canonisom_prp}, respectively, 
are also topological and take up only a few pages in total.
This approach allows us to study the extendability of the canonical orientations
of Corollary~\ref{orient0_crl} across the codimension-one boundary strata of the moduli spaces
on the topological level of real bundle pairs; see Section~\ref{BdExt_sec}.

\section{Extensions over compactifications}
\label{BdExt_sec}

\noindent
In this section, we study the extendability of the canonical isomorphisms
and orientations of Section~\ref{signcomp_sec} across paths passing through 
one-nodal symmetric surfaces.
Proposition~\ref{DMext_prp} below implies that the line bundle~\eref{DMext_e}
is orientable.
This is a key technical result needed to 
extend the proof of Corollary~\ref{orient0_crl} to
the compactified setting of Theorem~\ref{orient_thm}.
We deduce this proposition from the proof of Proposition~\ref{DM_prp}
and the statements of Corollary~\ref{realorient_crl} and Lemma~\ref{KSext_lmm}.

\begin{prp}\label{DMext_prp}
Let $g,l\in \Z^{\ge0}$ be such that $g\!+\!l\!\ge\!2$.
The orientation on the restriction of the real line bundle~\eref{DMext_e}
to $\R\cM_{g,l}$ provided by Proposition~\ref{DM_prp} flips 
across the codimension-one boundary strata of types~(E) and~(H1) and 
extends across the codimension-one boundary strata of types~(H2) and~(H3).
\end{prp}

\subsection{One-nodal symmetric surfaces}
\label{OneNodal_subs}

\noindent
A \sf{one-nodal oriented surface} $\Si$ is a topological space obtained by 
identifying two distinct points of a closed oriented smooth
surface~$\wt\Si$, not necessarily connected.
The surface $\wt\Si$ is called \sf{the normalization of~$\Si$};
it is unique up to a diffeomorphism preserving the two distinct points as a set.
A \sf{one-nodal symmetric surface} $(\Si,\si)$ is a connected one-nodal surface~$\Si$ 
with an involution~$\si$ induced by an orientation-reversing involution~$\wt\si$ 
on the normalization~$\wt\Si$ of~$\Si$.
Throughout this section, we will denote the two distinguished points of~$\wt\Si$
by~$x_1$ and~$x_2$ and their image in~$\Si$, i.e.~the node, by~$x_{12}$. 
The four topological possibilities for the singular structure of~$(\Si,\si)$ are described by (E)-(H3)
in Section~\ref{bnd_subs}.
Note~that 
$$\wt\si(x_i)=\begin{cases} x_{3-i},
&\hbox{if}~(\Si,\si)~\hbox{is of type~(E)};\\
x_i, &\hbox{if}~(\Si,\si)~\hbox{is of type~(H)}.
\end{cases}$$
Let $\wt\si'(x_1)\!=\!x_2$ and $\wt\si'(x_2)\!=\!x_1$.\\

\noindent
We begin by extending the main statements of Sections~\ref{topolprelim_subs}
and~\ref{ROisom_subs} to one-nodal symmetric surfaces.
In particular, we observe that Proposition~\ref{canonisom_prp} extends to such surfaces.
In~\cite{RBP}, we show that Proposition~\ref{canonisom_prp} actually extends to all
nodal symmetric surfaces.

\begin{prp}\label{canonisomExt_prp}
The conclusion of Proposition~\ref{canonisom_prp} holds for one-nodal symmetric surfaces.
\end{prp}

\begin{lmm}\label{Exthomotopextend_lmm}
The conclusion of Lemma~\ref{homotopextend_lmm} holds for one-nodal symmetric surfaces. 
\end{lmm}

\begin{proof}
Let $\wt{f}\!\in\!\cC(\wt\Si,\wt\si;\SL_n\C)$ be the function corresponding 
to $f\!\in\!\cC(\Si,\si;\SL_n\C)$. In particular, $\wt{f}(x_1)\!=\!\wt{f}(x_2)$.\\

\noindent
Suppose $(\Si,\si)$ is of type~(E). 
Proceeding as in the proof of Lemma~\ref{homotopextend_lmm}, 
choose~$\Si^b$ and~$U$ so that $x_1\!\in\!\Si^b\!-\!U$,
the cutting paths~$C_i$ so that $x_1\!\not\in\!C_i$,
and the extensions of the homotopies of~$\wt{f}$ from~$C_i$ to~$\Si^b$ 
so that they do not change~$\wt{f}$ at~$x_1$.
Choose an embedded path~$\ga$ in the disk~$D^2$ in the last paragraph of
the proof of Lemma~\ref{homotopextend_lmm} from~$x_1$ to~$\prt D^2$.
Since $\wt{f}(x_1)\!\in\!\SL_n\R$ in this case, 
we can homotope~$\wt{f}$ to~$\Id$ over~$\ga$ while keeping
the values of~$\wt{f}$ at~$x_1$
 and at the other endpoint in~$\SL_n\R$
and at~$\Id$, respectively.
Similarly to the second paragraph in the proof of Lemma~\ref{homotopextend_lmm}, 
this homotopy extends over~$D^2$ without changing~$\wt{f}$ over~$\prt D^2$ 
and thus descends to~$\Si^b$.
We then cut~$D^2$ along~$\ga$ into another disk and proceed as in
the second half of the last paragraph in the proof of Lemma~\ref{homotopextend_lmm}.
The doubled homotopy then satisfies $\wt{f}_t(x_1)\!=\!\wt{f}_t(x_2)$ and so descends to~$\Si$.\\

\noindent
If $(\Si,\si)$ is of type~(H), then 
$$\wt{f}\!:\bigcup_{|c_i|=0}(\prt\Si^b)_i\lra\SL_n\R$$
is homotopic to $\Id$ through maps~$\wt{f}_t$ such that $\wt{f}_t(x_1)\!=\!\wt{f}_t(x_2)$.
The remainder of the proof of Lemma~\ref{homotopextend_lmm} preserves this condition 
on the homotopy.
\end{proof}

\begin{crl}\label{RBPhomotExt_crl}
The conclusion of Corollary~\ref{RBPhomot_crl} holds for one-nodal symmetric surfaces. 
\end{crl}

\begin{proof}
The first paragraph of the proof of Corollary~\ref{RBPhomot_crl} applies without any changes.
The second paragraph applies with Lemma~\ref{homotopextend_lmm} replaced by Lemma~\ref{Exthomotopextend_lmm}.
\end{proof}

\begin{lmm}\label{LinAlg_lmm}
Let $(V,\fI)$ be a finite-dimensional complex vector space and $A,B\!:V\!\lra\!V$
be \hbox{$\C$-antilinear} isomorphisms such that $A^2,B^2\!=\!\Id_V$.
Then there exists a $\C$-linear isomorphism $\psi\!:V\!\lra\!V$ such that 
$\psi\!=\!A\!\circ\!\psi\!\circ\!B$.
If 
\BE{LinAlg_e}\big\{\La_{\C}^{\top}A\big\}\!\circ\!\big\{\La_{\C}^{\top}B\big\}
=\big\{\La_{\C}^{\top}B\big\}\!\circ\!\big\{\La_{\C}^{\top}A\big\}\!:
\La_{\C}^{\top}V\lra \La_{\C}^{\top}V,\EE
then $\psi$ can be chosen so that $\La_{\C}^{\top}\psi\!=\!\Id$.
\end{lmm}

\begin{proof}
Since $A^2,B^2\!=\!\Id_V$, the isomorphisms $A,B$ are diagonalizable with all eigenvalues $\pm1$.
Since $A,B$ are $\C$-antilinear, we can choose $\C$-bases $\{v_i\}$ and $\{w_i\}$ for~$V$
such~that 
$$A(v_i)=v_i, \quad A(\fI v_i)=-\fI v_i, \qquad B(w_i)=w_i, \quad B(\fI w_i)=-\fI w_i.$$
The $\C$-linear isomorphism $\psi\!:V\!\lra\!V$ defined by $\psi(w_i)\!=\!v_i$
then has the first desired property.\\

\noindent
The automorphisms $\La_{\C}^{\top}A$ and $\La_{\C}^{\top}B$ are $\C$-antilinear and
have one eigenvalue of $+1$ and one of~$-1$.
If~\eref{LinAlg_e} holds, the eigenspaces of $\La_{\C}^{\top}A$ and $\La_{\C}^{\top}B$ 
are the same and~so
$$v_1\!\w_{\C}\!\ldots\!\w_{\C}\!v_n= r\cdot w_1\!\w_{\C}\!\ldots\!\w_{\C}\!w_n
\in \La_{\C}^{\top}V$$
for some $r\!\in\!\R^*$.
Replacing $w_1$ by $rw_1$ in the previous paragraph, we obtain an isomorphism~$\psi$
that also satisfies the second property.
\end{proof}

\begin{proof}[{\bf\emph{Proof of Proposition~\ref{canonisomExt_prp}}}]
Let $\wt{V},\wt{L}\!\lra\!\wt\Si$ be complex vector bundles and
$$\psi_1\!:\wt{V}\big|_{x_1}\lra \wt{V}\big|_{x_2}
\qquad\hbox{and}\qquad 
\psi_2\!:\wt{L}\big|_{x_1}\lra \wt{L}\big|_{x_2}$$
be isomorphisms of complex vector spaces such~that 
$$V=\wt{V}\big/\!\!\sim, ~~ v\!\sim\!\psi_1(v)~\forall\,v\!\in\!\wt{V}\big|_{x_1},
\qquad\hbox{and}\qquad 
L=\wt{L}\big/\!\!\sim, ~~ v\!\sim\!\psi_2(v)~\forall\,v\!\in\!\wt{L}\big|_{x_1}.$$
Denote by $\wt\vph_1$ and~$\wt\vph_2$ the lift of $\vph$ to~$\wt{V}$ and 
the lift of $\wt\phi$ to~$\wt{L}$, respectively.
Define
$$\big(\wt{W},\wt\vph_{12}\big)=\big(\wt{V}\!\oplus\!2\wt{L}^*,\wt\vph_1\!\oplus\!2\wt\vph_2^*\big),
\qquad 
\psi_{12}=\psi_1\oplus2(\psi_2^{-1})^*\!: \wt{W}\big|_{x_1}\lra \wt{W}\big|_{x_2}\,.$$
Thus,  $(\wt{V},\wt\vph_1)$ and $(\wt{L},\wt\vph_2)$ are real bundle pairs over $(\wt\Si,\wt\si)$
that descend to the real bundle pairs $(V,\vph)$ and $(L,\wt\phi)$ over~$(\Si,\si)$.
Furthermore,
\BE{canIsomExt_e3}
\psi_{12}(v)=\begin{cases} \wt\vph_{12}(\psi_{12}^{-1}(\wt\vph_{12}(v))),
&\hbox{if}~(\Si,\si)~\hbox{is of type~(E)};\\
\wt\vph_{12}(\psi_{12}(\wt\vph_{12}(v))), &\hbox{if}~(\Si,\si)~\hbox{is of type~(H)};
\end{cases}\EE
for all $v\!\in\!\wt{W}\big|_{x_1}$.\\

\noindent
For any $f\!\in\!\cC(\wt\Si,\wt\si;\GL_{n+2}\C)$, let
$$\wt\Psi_f\!:\big(\wt\Si\!\times\!\C^{n+2},\wt\si\!\times\!\fc\big)\lra
\big(\wt\Si\!\times\!\C^{n+2},\wt\si\!\times\!\fc\big), \qquad
\wt\Psi_f(z,v)=\big(z,f(z)v\big).$$
The choices~\ref{isom_it2} and~\ref{spin_it2} in Definition~\ref{realorient_dfn4}
for~$(\Si,\si)$  lift to~$(\wt\Si,\wt\si)$. 
By Proposition~\ref{canonisom_prp}, there thus exists an isomorphism 
$$\wt\Phi\!:(\wt{W},\wt\vph_{12})\lra 
\big(\wt\Si\!\times\!\C^{n+2},\wt\si\!\times\!\fc\big)$$
of real bundle pairs over $(\wt\Si,\wt\si)$ that lies in the homotopy class
determined by the lifted real orientation.
By the proof of Proposition~\ref{canonisom_prp}, $\wt\Phi$ can be chosen so that 
it induces the isomorphism in~\eref{realorient2_e2b} over~$(\wt\Si,\wt\si)$
determined by the lift of a given isomorphism in~\eref{realorient_e4} over~$(\Si,\si)$.
This implies~that 
\BE{canIsomExt_e2a}
\big\{\wt\si'\!\times\!\Id\big\}\!\circ\!\big\{\La_{\C}^{\top}\wt\Phi\big\}=
\big\{\La_{\C}^{\top}\wt\Phi\big\}\!\circ\!\big\{\La_{\C}^{\top}\psi_{12}\big\}\!:
\La_{\C}^{\top}\wt{W}|_{x_1}\lra 
\{x_2\}\!\times\!\La_{\C}^{\top}\C^{n+2}\!=\!\{x_2\}\!\times\!\C.\EE
We show below that there exists $f\!\in\!\cC(\Si,\si;\SL_{n+2}\C)$ so that 
\BE{canIsomExt_e2b} \big\{\wt\si'\!\times\!\Id\big\}\!\circ\!\wt\Psi_f\!\circ\!\wt\Phi=
\wt\Psi_f\!\circ\!\wt\Phi\!\circ\!\psi_{12}\!:
\wt{W}|_{x_1}\lra \{x_2\}\!\times\!\C^{n+2}\,.\EE
Thus, $\wt\Psi_f\!\circ\!\wt\Phi$ descends 
to an isomorphism~$\Psi$ in~\eref{realorient2_e2} of real bundle pairs over~$(\Si,\si)$
that induces the isomorphism in~\eref{realorient2_e2b} determined by 
a given isomorphism in~\eref{realorient_e4}.
Furthermore, $f$ can be chosen so that~$\Psi$ satisfies the spin structure
requirement of Proposition~\ref{canonisom_prp}.
By Corollary~\ref{RBPhomotExt_crl}, any two isomorphisms~\eref{realorient2_e2}
satisfying the conditions at the end of Proposition~\ref{canonisom_prp}
are homotopic.\\

\noindent
Suppose $(\Si,\si)$ is of type~(E). 
By~\eref{canIsomExt_e3}, the $\C$-antilinear isomorphisms
$$\id\!\times\!\fc,\big\{\wt\si\!\times\!\fc\big\}
\!\circ\!\wt\Phi\!\circ\!\psi_{12}\!\circ\!\wt\Phi^{-1}
\!=\!\wt\Phi\!\circ\!\wt\vph_{12}\!\circ\!\psi_{12}\!\circ\!\wt\Phi^{-1}\!: 
\{x_1\}\!\times\!\C^{n+2} \lra \{x_1\}\!\times\!\C^{n+2}$$
square to the identity.
By~\eref{canIsomExt_e2a}, the top exterior powers of these automorphisms commute 
(both compositions are the identity).
By Lemma~\ref{LinAlg_lmm}, there thus exists $\psi\!\in\!\SL_{n+2}\C$ such~that 
\BE{canIsomExt_e4a} 
\id\!\times\!\psi=\big\{\wt\si\!\times\!\fc\psi\fc\big\}
\!\circ\!\wt\Phi\!\circ\!\psi_{12}\!\circ\!\wt\Phi^{-1}\!:
\{x_1\}\!\times\!\C^{n+2}\lra \{x_1\}\!\times\!\C^{n+2}.\EE
Since $\SL_{n+2}\C$ is connected, there exist $f\!\in\!\cC(\wt\Si,\wt\si;\SL_{n+2}\C)$ 
and a neighborhood $U$ of $x_1$ in~$\wt\Si$ such~that 
\BE{canIsomExt_e4b} 
f(z)=
\begin{cases}\psi,&\hbox{if}~z\!=\!x_1;\\
\Id,&\hbox{if}~z\!\not\in\!U\!\cup\!\wt\si(U);
\end{cases}
\qquad U\!\cap\!\wt\si(U)=\eset.\EE
By~\eref{canIsomExt_e4a} and~\eref{canIsomExt_e4b}, $f$ satisfies~\eref{canIsomExt_e2b}.
Since $f$ restricts to the identity over~$\wt\Si^{\wt\si}$, $\wt\Psi_f\!\circ\!\wt\Phi$ 
induces the same orientation and spin structure over $\wt\Si^{\wt\si}\!-\!\{x_{12}\}$
as~$\wt\Phi$.
The orientation and spin conditions are automatically satisfied over~$x_{12}$,
since they are determined by the real part of the isomorphism~\eref{realorient2_e2b}.\\

\noindent
If $(\Si,\si)$ is of type~(H), define $\psi\!\in\!\GL_{n+2}\C$ by
\BE{canIsomExt_e7}
\id\!\times\!\psi=\big\{\wt\si'\!\times\!\Id\big\}
\!\circ\!\wt\Phi\!\circ\!\psi_{12}\!\circ\!\wt\Phi^{-1}\!:
\{x_1\}\!\times\!\C^{n+2}\lra \{x_1\}\!\times\!\C^{n+2}\,.\EE
By~\eref{canIsomExt_e3} and~\eref{canIsomExt_e2a}, 
$\psi\!\circ\!\fc=\!\fc\!\circ\!\psi$
and $\det_{\C}\!\psi\!=\!1$, i.e.~$\psi\!\in\!\SL_{n+2}\R$.
If $(\Si,\si)$ is of type~(H2) or~(H3), i.e.~$x_1$ and $x_2$ lie on different topological components
$\wt\Si^{\wt\si}_1,\wt\Si^{\wt\si}_2$ of~$\wt\Si^{\wt\si}$,
let 
\BE{canIsomExt_e8a} \wt\psi\!:\wt\Si^{\wt\si}_1 \lra\SL_{n+2}\R\EE 
be the constant function with value~$\psi$.
If $(\Si,\si)$ is of type~(H1), i.e.~$x_1$ and $x_2$ lie on the same topological component
$\wt\Si_1^{\wt\si}$ of~$\wt\Si^{\wt\si}$, first choose~\eref{canIsomExt_e8a}
so that $\wt\psi(x_1)\!=\!\psi$ and~$\wt\psi(x_2)\!=\!\Id$.
Since $f\!=\!\wt\psi$ satisfies~\eref{canIsomExt_e2b},  $\wt\Psi_f\!\circ\!\wt\Phi$ 
induces a trivialization of $V^{\vph}\!\oplus\!2(L^*)^{\wt\phi^*}$ over the image~$\Si^{\si}_1$
of $\wt\Si^{\wt\si}_1$ in~$\Si$.
This is also the case if~$\wt\psi$ is replaced by $\wt\psi'\wt\psi$ for any~$\wt\psi'$
as in~\eref{canIsomExt_e8a} such that $\wt\psi'(x_1),\wt\psi'(x_2)\!=\!\Id$.
Choose such $\wt\psi'$ so that the induced trivialization on each of the two loops 
in~$\Si^{\si}_1$ lies in the chosen spin structure; we then replace $\wt\psi$ 
with $\wt\psi'\wt\psi$.
Returning to the general~$(H)$ case, choose $f\!\in\!\cC(\wt\Si,\wt\si;\SL_{n+2}\C)$ 
and a neighborhood~$U$ of $\wt\Si^{\wt\si}_1$ in~$\wt\Si$ such~that 
\BE{canIsomExt_e8} 
f(z)=\begin{cases}\wt\psi,&\hbox{if}~z\!\in\!\wt\Si^{\wt\si}_1;\\
\Id,&\hbox{if}~z\!\not\in\!U;
\end{cases}
\qquad U\!\cap\!\big(\wt\Si^{\wt\si}\!-\!\wt\Si_1^{\wt\si}\big)=\eset;\EE
this is possible by Lemma~\ref{lmm_repar}.
By~\eref{canIsomExt_e7} and~\eref{canIsomExt_e8}, $\wt\Psi$ satisfies~\eref{canIsomExt_e2b}.
Since $f$ restricts to the identity over~$\wt\Si^{\wt\si}\!-\!\wt\Si^{\wt\si}_1$, 
$\wt\Psi_f\!\circ\!\wt\Phi$ 
induces the same orientation and spin structure over $\wt\Si^{\wt\si}\!-\!\wt\Si^{\wt\si}_1$
as~$\wt\Phi$.
If $(\Si,\si)$ is of type~(H2) or~(H3), the latter is also the case over~$\wt\Si^{\wt\si}_1$
because $f$ is constant over~$\wt\Si^{\wt\si}_1$.
If $(\Si,\si)$ is of type~(H1),  
the orientation and spin structure structure induced by $\wt\Psi_f\!\circ\!\wt\Phi$ 
over~$\Si^{\si}_1$ are those of the original real orientation by the choice of~$\wt\psi$ above.
\end{proof}

\begin{crl}\label{canonisomExt_crl2a}
The first conclusion of Corollary~\ref{canonisom_crl2a} holds for one-nodal symmetric surfaces.
\end{crl}

\begin{proof}
An orientation on the determinant line of a real CR-operator
on a real bundle pair~$(V,\vph)$ over a  one-nodal symmetric surface~$(\Si,\si)$
is determined~by 
\begin{enumerate}[label=(\arabic*),leftmargin=*]
\item an orientation on the determinant line of a real CR-operator 
on the corresponding real bundle pair~$(\wt{V},\wt\vph)$ over $(\wt\Si,\wt\si)$ 
as in the proof of Proposition~\ref{canonisomExt_prp}, and
\item\label{match_it} an orientation on the real vector space $V_{x_{12}}^{\vph}$.
\end{enumerate}
An isomorphism of real bundle pairs over~$(\Si,\si)$ as in~\eref{realorient2_e2}
lifts to a similar isomorphism over~$(\wt\Si,\wt\si)$ which respects 
all identifications on the lifted bundles.
A real orientation on~$(V,\vph)$ determines~\ref{match_it} and 
an isomorphism of real bundle pairs over~$(\Si,\si)$ as in~\eref{realorient2_e2};
see Proposition~\ref{canonisomExt_prp}.
Thus, the claim follows from
the proof of the first conclusion of Corollary~\ref{canonisom_crl2a}.
\end{proof}

\subsection{Smoothings of one-nodal symmetric surfaces}
\label{SymmSurfDegen_subs}

\noindent
Let $\cC\!\equiv\!(\Si,\si,(z_1^+,z_1^-),\ldots,(z_l^+,z_l^-))$ be 
a one-nodal marked symmetric Riemann surface and 
$$(\pi\!:\cU\!\lra\!\De,\wt\fc\!:\cU\!\lra\!\cU,s_1\!:\De\!\lra\!\cU,\ldots,s_l\!:\De\!\lra\!\cU)$$
be a flat family of deformations of~$\cC$ as in Section~\ref{GrConv_subs}
with $\De\!\subset\!\C$.
Define
$$\De^*=\De\!-\!\{0\}, \quad \De_{\R}=\De\!\cap\!\R, \quad 
\De_{\R}^*=\De^*\!\cap\!\R, \quad 
\De_{\R}^{\pm}=\De\!\cap\!\R^{\pm}\,.$$
Denote by $x_{12}\!\in\!\Si$ the node of~$\Si$,
by $\wt\Si\!\lra\!\Si$ its normalization, and by 
\hbox{$\Si^*\!\equiv\!\Si\!-\!\{x_{12}\}$} its smooth locus.\\

\noindent
A neighborhood of $x_{12}$ in~$\cU$ is isomorphic to
$$\cU_0\equiv\big\{(t,z_1,z_2)\!\in\!\De\!\times\!\C^2\!:\,|z_1|,|z_2|\!<\!1,~z_1z_2\!=\!t\big\}.$$
As fibrations over~$\De$,
\BE{cUdfn_e}\cU\approx\big(\cU_0 \sqcup \cU'\big)\big/\sim, \qquad 
(t,z_1,z_2)\sim\begin{cases}(t,z_1),&\hbox{if}~|z_1|\!>\!|z_2|;\\
(t,z_2),&\hbox{if}~|z_1|\!<\!|z_2|;\end{cases}\EE
for some family~$\cU'$ of deformations of~$\Si^*$ over~$\De$,
a choice of coordinates $z_i$ on~$\wt\Si$ centered at~$x_i$,
and their extensions to~$\cU$.
The local coordinates $z_1,z_2$ and the family~$\cU'$ in~\eref{cUdfn_e} 
can be chosen so that~$\cU'$ is preserved by~$\wt\fc$ and 
the identification in~\eref{cUdfn_e} intertwines 
the involution
\BE{cU0inv_e}\cU_0\lra\cU_0, \qquad (t,z_1,z_2)\lra\big(\ov{t},\ov{z_2},\ov{z_1}\big)
~~\hbox{or}~~ (t,z_1,z_2)\lra\big(\ov{t},\ov{z_1},\ov{z_2}\big),\EE
depending on whether $(\Si,x_{12},\si)$ is of type~(E) or~(H),
with the involution~$\wt\fc$ on~$\cU$.
In particular, $\cU$ retracts onto~$\Si_0$ respecting the involution~$\wt\fc$.\\

\noindent
Suppose $\pi\!:\cU\!\lra\!\De$ and $\wt\fc$ are as above, $(V,\vph)\!\lra\!(\cU,\wt\fc)$
is a real bundle pair, and $\na$ and~$A$ are a connection and a 0-th order deformation
term on~$(V,\vph)$ as in Section~\ref{DetLB_subs}.
The restriction of $\na$ and~$A$ to $(V,\vph)|_{(\Si_t,\si_t)}$ with 
$t\!\in\!\De_{\R}$ determines a real CR-operator~$D_t$.
The determinant lines of these operators form a line bundle
\BE{detExt_e}\det D_{(V,\vph)}\lra\De_{\R}\,;\EE
see Section~\ref{DetLB_subs} and Appendix~\ref{detLB_app}.
We denote by $\det\dbar_{\C}\!\lra\!\De_{\R}$ the determinant line bundle 
associated with the standard holomorphic structure on $(\cU\!\times\!\C,\wt\fc\!\times\!\fc)$.

\begin{crl}\label{canonisomExt2_crl2a}
Let $(\pi,\wt\fc)$, $(V,\vph)$, and $(\na,A)$ be as above.
Then a real orientation on $(V,\vph)$ as in Definition~\ref{realorient_dfn4}
induces an orientation on the line bundle
\BE{canonisomExt2_e}
\wh\det\, D_{(V,\vph)}\equiv
\big(\!\det D_{(V,\vph)}\big) \otimes (\det\dbar_{\C})^{\otimes n}\lra \De_{\R} ,\EE
where $n\!=\!\rk_{\C}V$.
The restriction of this orientation to the fiber over each $t\!\in\!\De_{\R}^*$
is the orientation on~$\wh\det\,D_t$
induced by the restriction of the real orientation to 
$(V,\vph)|_{(\Si_t,\si_t)}$ as in Corollary~\ref{canonisom_crl2a}.
\end{crl}

\begin{proof}
By Proposition~\ref{canonisomExt_prp}, 
the restriction of the real orientation to $(V,\vph)|_{(\Si_0,\si_0)}$
determines a homotopy class of isomorphisms~$\Psi$ of real bundle pairs as 
in~\eref{realorient2_e2}.
Since $\cU$ retracts onto~$\Si_0$ respecting the involution~$\wt\fc$,
every isomorphism~$\Psi_0$ over~$(\Si_0,\si_0)$ extends to 
an isomorphism
\BE{canonisomExt2_e1a}\Psi\!:\big(V\!\oplus\!2L^*,\vph\!\oplus\!2\wt\phi^*\big)
\approx\big(\cU\!\times\!\C^{n+2},\wt\fc\!\times\!\fc\big)\EE
of real bundle pairs over~$(\cU,\wt\fc)$.
Since an isomorphism~$\Psi_0$ in the homotopy class determined by 
the restriction of the real orientation to $(V,\vph)|_{(\Si_0,\si_0)}$ satisfies 
the spin structure and $\La_{\C}^{\top}$ conditions at the end of Proposition~\ref{canonisom_prp},
the restriction~$\Psi_t$ of~\eref{canonisomExt2_e1a} to
$(V,\vph)|_{(\Si_t,\si_t)}$   also satisfies these conditions.
The restriction of the orientation of the line bundle~\eref{canonisomExt2_e}
induced by~$\Psi$ to the fiber over each $t\!\in\!\De_{\R}^*$
is the orientation induced by~$\Psi_t$.
The latter is the orientation induced by the restriction of the real orientations to 
$(V,\vph)|_{(\Si_t,\si_t)}$.
\end{proof}

\noindent
Thus, the real line bundle~\eref{thm_maps_e9} extends across the (codimension-one) boundary strata
of the moduli spaces $\ov\fM_{g,l}(X,B;J)^{\phi}$ and so does its orientation induced by 
a real orientation on~$(X,\phi)$.
The other factor in orienting the line bundle~\eref{orient_thm_e} 
over the uncompactified  space $\fM_{g,l}(X,B;J)^{\phi}$
is the canonical orientation of the line bundle~\eref{CidentDM_e}.
The next lemma makes it possible to extend the orientations induced by the isomorphisms~\eref{DM_prp_e3} 
used in orienting~\eref{CidentDM_e} to (but not {\it across})
the boundary strata.\\

\noindent
Let $\wt\Si$ be a smooth Riemann surface and $x\!\in\!\wt\Si$.
A holomorphic vector field~$\xi$ on a neighborhood of~$x$ in~$\wt\Si$ 
with $\xi(x)\!=\!0$ determines an element 
$$\na\xi\big|_{x}\in T_x^*\wt\Si\otimes_{\C} T_x\wt\Si=\C\,.$$
Similarly, a meromorphic one-form~$\eta$ on a neighborhood of~$x$ in~$\wt\Si$ 
has a well-defined residue at~$x$, which we denote by~$\fR_x\eta$.
For a holomorphic line bundle $L\!\lra\!\wt\Si$, we denote by $\Om(L)$ the sheaf
of holomorphic sections of~$L$.

\begin{lmm}\label{TSiext_lmm}
Let $(\pi\!:\cU\!\lra\!\De,\wt\fc)$ be a flat family of deformations 
of a one-nodal symmetric Riemann surface $(\Si,\si)$ with $\De\!\subset\!\C$
and $x_1,x_2\!\in\!\wt\Si$ be the preimages of the node $x_{12}\!\in\!\Si$ in
its normalization.
There exist holomorphic line bundles $\cT,\wh\cT\!\lra\!\cU$ with involutions~$\vph,\wh\vph$
lifting~$\wt\fc$ such~that 
\begin{gather*}
(\cT,\vph)\big|_{\Si_t}=\big(T\Si_t,\tnd\wt\fc|_{T\Si_t}\big), \quad 
(\wh\cT,\wh\vph)\big|_{\Si_t}=\big(T^*\Si_t,(\tnd\wt\fc|_{T\Si_t})^*\big) 
\qquad\forall~t\!\in\!\De^*,\\
\begin{split}
\Om\big(\cT|_{\Si_0}\big)=\big\{\xi\!\in\!\Om\big(T\wt\Si(-x_1\!-\!x_2)\big)\!:\,~
\na\xi|_{x_1}\!+\!\na\xi|_{x_2}\!=\!0\big\},\\
\Om\big(\wh\cT|_{\Si_0}\big)=\big\{\eta\!\in\!\Om\big(T^*\wt\Si(x_1\!+\!x_2)\big)\!:\,~
\fR_{x_1}\eta\!+\!\fR_{x_2}\eta\!=\!0\big\}.
\end{split}\end{gather*}
Furthermore, $(\wh\cT,\wh\vph)\!\approx\!(\cT,\vph)^*$.
\end{lmm}

\begin{proof}
We continue with the notation as in~\eref{cUdfn_e} and~\eref{cU0inv_e}.
Denote by $T^{\vrt}\cU'\!\lra\!\cU'$ the vertical tangent bundle.
Let
\begin{alignat*}{2}
\cT&=\big(\cU_0\!\times\!\C \sqcup T^{\vrt}\cU'\big)\big/\!\sim, &\qquad
\wh\cT&=\big(\cU_0\!\times\!\C \sqcup (T^{\vrt}\cU')^*\big)\big/\!\sim,\\
(t,z_1,z_2,c)&\sim\begin{cases}
c\,z_1\frac{\prt}{\prt z_1}\big|_{(t,z_1)},&\hbox{if}~|z_1|\!>\!|z_2|;\\
-c\,z_2\frac{\prt}{\prt z_2}\big|_{(t,z_2)},&\hbox{if}~|z_1|\!<\!|z_2|;\\
\end{cases} &\qquad
(t,z_1,z_2,c)&\sim\begin{cases}
c\,\frac{\tnd_{(t,z_1)}z_1}{z_1},&\hbox{if}~|z_1|\!>\!|z_2|;\\
-c\,\frac{\tnd_{(t,z_2)}z_2}{z_2},&\hbox{if}~|z_1|\!<\!|z_2|.\\
\end{cases}
\end{alignat*}
Under the identifications~\eref{cUdfn_e}, the vector field and one-form on
a neighborhood of the node in~$\cU$ associated with $(t,z_1,z_2,c)\!\in\!\cU_0\!\times\!\C$
correspond to the vector field and one-form on~$\cU_0$
given~by
$$c\bigg(z_1\frac{\prt}{\prt z_1}-z_2\frac{\prt}{\prt z_2}\bigg)
\qquad\hbox{and}\qquad c\,\frac{\tnd z_1|_{\Si_t}}{z_1}=-c\,\frac{\tnd z_2|_{\Si_t}}{z_2},$$
respectively (the above equality of one-forms holds for $t\!\neq\!0$).
Thus, $\cT$ and~$\wh\cT$ have the desired restriction properties.
Since the~map
\begin{equation*}
{}\quad\cT\!\otimes_{\C}\!\wh\cT\lra\cU\!\times\!\C, 
\hspace{-.7in}
\begin{split}
&\big[t,z_1,z_2,c_1\big]\!\otimes\!\big[t,z_1,z_2,c_2\big]\lra  \big([t,z_1,z_2],c_1c_2\big),
\quad (t,z_1,z_2)\in\cU_0,~c_1,c_2\!\in\!\C,\\
&[v]\otimes[\al]\lra \al(v)\quad\forall\,v\!\in\!T_{z_i}\Si_t,~\al\!\in\!T_{z_i}^*\Si_t,~
(t,z_i)\in\cU',
\end{split}\end{equation*}
is a well-defined isomorphism of holomorphic line bundles, $\wh\cT\!\approx\!\cT^*$.\\

\noindent
The identifications in the construction of~$\cT$ and~$\wh\cT$ above intertwine
the trivial lift of~\eref{cU0inv_e} to a conjugation on $\cU_0\!\times\!\C$
with the conjugations on $T^{\vrt}\cU'$ and $(T^{\vrt}\cU')^*$ induced by~$\tnd\wt\fc$.
Thus, they induce conjugations~$\vph$ and~$\wh\vph$ on~$\cT$ and~$\wh\cT$.
The above trivialization of $\cT\!\otimes_{\C}\!\wh\cT$ intertwines the resulting
conjugation on the domain with the conjugation~$\wt\fc\!\times\!\fc$ on $\cU\!\times\!\C$.
Thus, $(\wh\cT,\wh\vph)$ and $(\cT,\vph)^*$ are isomorphic as real bundle pairs 
over~$(\cU,\wt\fc)$.
\end{proof}

\begin{lmm}[Dolbeault Isomorphism]\label{Cech_lmm}
Suppose $(\Si,\si)$ and $(\pi\!:\cU\!\lra\!\De,\wt\fc)$ are as in Lemma~\ref{TSiext_lmm}
and \hbox{$(L,\wt\phi)\!\lra\!(\cU,\wt\fc)$} is a holomorphic line bundle so that  
$\deg L|_{\Si}\!<\!0$ and $\deg L|_{\Si'}\!\le\!0$ for each irreducible component $\Si'\!\subset\!\Si$.
The families of vector spaces $H_{\dbar}^1(\Si_t;L)$ and $\wch{H}^1(\Si_t;L)$
then form vector bundles $R_{\dbar}^1\pi_*L$ and $\wch{R}^1\pi_*L$ over~$\De$ 
with conjugations lifting~$\fc$ which are canonically isomorphic as real bundle pairs
over~$(\De,\fc)$.
\end{lmm}

\begin{proof}
The assumptions on $L$ ensure that $H_{\dbar}^0(\Si_t;L)\!=\!0$ for all $t\!\in\!\De$.
By the Dolbeault Theorem \cite[p151]{GH}, this implies that $\wch{H}^0(\Si_t;L)\!=\!0$
for all $t\!\in\!\De$.
Since $H_{\dbar}^0(\Si_t;L)\!=\!0$ for all $t\!\in\!\De$,
the vector spaces $H_{\dbar}^1(\Si_t;L)$ naturally form a vector bundle 
$R_{\dbar}^1\pi_*L$ over~$\De$.
By the second statement, the sheaf $R^1\pi_*L$ is locally free over~$\De$
and thus corresponds to a vector bundle  $\wch{R}^1\pi_*L$ over~$\De$.
The involution~$\wt\fc$ and conjugation~$\wt\phi$ induce conjugations on the two bundles.
The Dolbeault Isomorphism provides an isomorphism between the two resulting real bundle pairs
over~$(\De,\fc)$.
\end{proof}

\begin{lmm}[Serre Duality]\label{Serre_lmm}
Suppose $(\Si,\si)$, $(\pi\!:\cU\!\lra\!\De,\wt\fc)$,  and $(\wh\cT,\wh\vph)$
are as in Lemma~\ref{TSiext_lmm}
and \hbox{$(L,\wt\phi)\!\lra\!(\cU,\wt\fc)$} is a holomorphic line bundle 
so that  $\deg L|_{\Si}\!>\!2g_a(\Si)\!-\!2$
and $\deg L|_{\Si'}\!\ge\!2g_a(\Si')\!-\!2$
for each irreducible component $\Si'\!\subset\!\Si$.
The family of vector spaces $H_{\dbar}^0(\Si_t;L)$ 
then forms a vector bundle $R_{\dbar}^0\pi_*L$ over~$\De$  with a conjugation lifting~$\fc$ 
and there is a canonical isomorphism
\BE{Serre_e}R_{\dbar}^1\pi_*\big(L^*\!\otimes\!\wh\cT\big)\approx\big(R_{\dbar}^0\pi_*L\big)^*\EE
of real bundle pairs over~$(\De,\fc)$.
\end{lmm}

\begin{proof}
The left-hand side of~\eref{Serre_e} is a vector bundle by Lemma~\ref{Cech_lmm}.
The assumptions on~$L$ ensure that $H_{\dbar}^1(\Si_t;L)\!=\!0$ for all $t\!\in\!\De$.
Thus, the vector spaces $H_{\dbar}^0(\Si_t;L)$ with $t\!\in\!\De$ 
naturally form a vector bundle $R_{\dbar}^0\pi_*L$ over~$\De$.
The involution~$\wt\fc$ and conjugation~$\wt\phi$ induce a conjugation on 
the right-hand side of~\eref{Serre_e}.
The Serre Duality provides an isomorphism between the two bundles in~\eref{Serre_e}. 
Its composition with the multiplication by~$\fI$ is an isomorphism between 
the two bundles in~\eref{Serre_e} as real bundle pairs over~$(\De,\fc)$.
\end{proof}

\begin{rmk}\label{DIandSD_rmk}
The justification of Dolbeault Isomorphism Theorem in the case of Lemma~\ref{Cech_lmm}
consists of applying the exact sequence of sheaves at the bottom of \cite[p150]{GH}
with $p,q\!=\!0$ and $E\!=\!L$.
As the standard $\dbar$-operator on a wedge of two disks is surjective,
this sequence is indeed exact over the central fiber $\Si_0\!=\!\Si$
(the exactness is established in~\cite{GH} over complex manifolds).
The Serre Duality for CR-operators over nodal Riemann surfaces 
appears in \cite[Lemma~2.3]{FanoGV} and endows the total spaces of 
the left-hand side in~\eref{Serre_e} and of the bundle $R_{\dbar}^1\pi_*L$ in Lemma~\ref{Cech_lmm}  
with a topology via the fiberwise SD~isomorphisms.
The Serre Duality appears on the level of \v{C}ech cohomology in
the standard algebro-geometric perspective; see \cite[p98]{ACH}.
This viewpoint would establish Corollary~\ref{Serre_crl} below by applying the Serre Duality first
and the Dolbeault Isomorphism second.
\end{rmk}

\noindent
Let $(\pi,\wt\fc,s_1,\ldots,s_l)$ be a smoothing of a one-nodal marked symmetric Riemann surface 
\BE{cCdfn_e}\cC \equiv \big(\Si,(z_1^+,z_1^-),\ldots,(z_l^+,z_l^-)\big),\EE
$\cT,\wh\cT\!\lra\!\cU$ be the holomorphic line bundles with involutions~$\vph,\wh\vph$
as in Lemma~\ref{TSiext_lmm},
and
$$\cT\cC=\cT\big(\!-\!s_1\!-\!\wt\fc\!\circ\!s_1\!-\!\ldots\!-\!s_l\!-\!\wt\fc\!\circ\!s_l\big), \qquad
\wh\cT\cC=\wh\cT\big(s_1\!+\!\wt\fc\!\circ\!s_1\!+\!\ldots\!+\!s_l\!+\!\wt\fc\!\circ\!s_l\big).$$
By the last statement of Lemma~\ref{TSiext_lmm}, $\cT\cC^*\!=\!\wh\cT\cC$. 

\begin{crl}\label{Serre_crl}
If the marked curve~\eref{cCdfn_e} is stable, the orientation on the restriction of the real line bundle
\BE{SerreCrl_e}
\La_{\R}^{\top}\big((\wch{R}^1\pi_*\cT\cC)^{\si}\big)\otimes
\La_{\R}^{\top}\big((R_{\dbar}^0\pi_*(\wh\cT\cC\!\otimes\!\wh\cT))^{\si}\big)
 \lra \De_{\R}\EE
to $\De_{\R}^*$ induced by the Dolbeault and SD isomorphisms 
as in  the proof of Proposition~\ref{DM_prp} 
extends across $t\!=\!0$.
\end{crl}

\begin{proof}
By Lemma~\ref{Cech_lmm} with $L\!=\!\cT\cC$ 
and Lemma~\ref{Serre_lmm} with $L\!=\!\wh\cT\cC\!\otimes\!\wh\cT$,
there are canonical isomorphisms of vector bundles
$$\wch{R}^1\pi_*\cT\cC \approx  R_{\dbar}^1\pi_*\cT\cC
= R_{\dbar}^1\pi_*\big((\wh\cT\cC\!\otimes\!\wh\cT)^*\!\otimes\!\wh\cT\big)
\approx \big(R_{\dbar}^0\pi_*(\wh\cT\cC\!\otimes\!\wh\cT)\big)^*$$
over $\De$ which restrict to 
the Dolbeault and SD isomorphisms over each point.
Since they commute with the involutions on the vector bundles, 
these isomorphisms induce an orientation on the real line bundle~\eref{SerreCrl_e} that restricts 
to the orientation on each fiber induced by the real parts of 
the Dolbeault and SD isomorphisms.
\end{proof}

\subsection{The orientability of the real Deligne-Mumford space}
\label{DMext_subs}

\noindent
We now study the extendability of the canonical orientations of the line bundles
appearing in the proof of Proposition~\ref{DM_prp} and establish Proposition~\ref{DMext_prp}.
The two main ingredients in this proof are Lemmas~\ref{realorient_lmm} 
and~\ref{KSext_lmm} below. The next lemma summarizes 
the fundamental difference between the two pairs of cases in Proposition~\ref{DMext_prp}.

\begin{lmm}\label{EH1vsH2H3_lmm}
Let $(\Si,x_{12},\si)$,  $(\pi,\wt\fc)$, and
$(\cT,\vph)$ be as in Lemma~\ref{TSiext_lmm}.
The restriction of the real line bundle $\cT^{\vph}\!\lra\!\Si^{\si}$ 
to the singular topological component $\Si^{\si}_1\!\subset\!\Si^{\si}$
is orientable if the one-nodal symmetric surface $(\Si,x_{12},\si)$ is
of type~(E) or~(H1) and is not orientable if  $(\Si,x_{12},\si)$ is
of type~(H2) or~(H3).
\end{lmm}

\begin{proof}
If  $(\Si,x_{12},\si)$ is of type~(E),  $\Si^{\si}_1$ consists of the node~$x_{12}$ and there is nothing 
to prove.
Otherwise, a local section of~$\cT^{\vph}$ near~$x_{12}$  is given by 
$x\frac{\prt}{\prt x}$ along the $x$-axis and 
$-y\frac{\prt}{\prt y}$ along the $y$-axis.
It points away from the origin along the $x$-axis and towards along the $y$-axis.
The claims in the  (H1) and (H2)/(H3)  are thus immediate from 
the middle diagrams in Figures~\ref{H1_fig} and~\ref{H2H3_fig}, respectively.
\end{proof}

\noindent
Suppose $(\Si,x_{12},\si)$ is of type~(E) or~(H1).
By the first part of the proof of Corollary~\ref{canonisom_crl}, the restriction 
of the real bundle~pair
\BE{whcTcC_e} 
\big(\wh\cT^{\otimes2}\!\oplus\!2\cT,\wh\vph^{\otimes2}\!\oplus\!2\vph\big)
\lra(\cU,\wt\fc)\EE
to the central fiber $(\Si,\si)$ thus has a canonical real orientation. 
It extends to a real orientation on~\eref{whcTcC_e} which restricts 
to the canonical real orientation over each fiber~$(\Si_t,\si_t)$ with $t\!\in\!\De_{\R}^*$.\\

\noindent
Suppose $(\Si,x_{12},\si)$ is of type~(H2) or~(H3).
The singular component~$\Si^{\si}_1$ of~$\Si^{\si}$ consists of two copies of~$S^1$
with a point~$x_1$ on the first copy identified with a point~$x_2$ on the second copy.
By Corollary~\ref{canonisom_crl},
there are then {\it four} natural  real orientations on  the restriction of~\eref{whcTcC_e}
to~$(\Si,\si)$.
They correspond to the two orientations of each of the two irreducible components
of~$\Si^{\si}_1$.
Each of the four real orientations
extends to a real orientation on the real bundle pair~\eref{whcTcC_e}
over~$(\cU,\wt\fc)$. 

\begin{lmm}\label{realorient_lmm}
Let $\cC$, $(\pi,\wt\fc)$, and $\cT,\wh\cT\!\lra\!\cU$ be as in Lemma~\ref{TSiext_lmm}
with $(\Si,x_{12},\si)$ of type~(H2) or~(H3).
For each of the four natural real orientations on the restriction of~\eref{whcTcC_e} to~$(\Si,\si)$,
there exists $\ve\!\in\!\{\pm1\}$ such that 
the restriction over~$(\Si_t,\si_t)$ of the extension of this real orientation 
over~$(\cU,\wt\fc)$ is the canonical real orientation if $\ve t\!\in\!\De_{\R}^+$ 
and differs from the canonical real orientation by the spin structure over precisely one 
component of~$\Si_t^{\si_t}$ if $\ve t\!\in\!\De_{\R}^-$. 
\end{lmm}

\begin{proof}
For $t\!\in\!\De_{\R}^*$, the topological component $\Si_{t;1}^{\si_t}$ of $\Si_t^{\si_t}$
corresponding to~$\Si^{\si}_1$ is obtained as follows.
Cut the first copy of~$S^1$ at~$x_1$ into a closed interval~$S^1_1$ with endpoints~$1^-$ and~$1^+$;
cut the second copy of~$S^1$ at~$x_2$ into a closed interval~$S^1_2$ with endpoints~$2^-$ and~$2^+$.
For $t\!\in\!\De_{\R}^+$, $\Si_{t;1}^{\si_t}\!\approx\!S^1$ is formed from~$S^1_1$ and~$S^1_2$
by identifying either~$1^-$ with~$2^+$ and~$1^+$ with~$2^-$
or~$1^-$ with~$2^-$ and~$1^+$ with~$2^+$.
For  $t\!\in\!\De_{\R}^-$, $\Si_{t;1}^{\si_t}$ is formed by the other identification.
Thus, the transition from $\Si_{t;1}^{\si_t}$ with $t\!\in\!\De_{\R}^-$ to 
$\Si_{t;1}^{\si_t}$ with $t\!\in\!\De_{\R}^+$ is equivalent to flipping
the second copy of~$S^1$ around~$x_2$ and another point.
This flips the orientation on~$S_2^1$.
By the second part of the proof of Corollary~\ref{canonisom_crl}, this is equivalent 
to flipping the spin structure on the restriction of the real part of~\eref{whcTcC_e}
to half of~$\Si_t^{\si_t}\!\approx\!S^1$ with $t\!\in\!\De_{\R}^*$.
Thus, precisely one of the two spin structures (either before or after the flip)
on the restriction of the real part of~\eref{whcTcC_e} to~$\Si_t^{\si_t}$ is the canonical~one.
\end{proof}

\begin{rmk}\label{realorient_rmk}
Suppose both copies of $S^1$ in the proof of Lemma~\ref{realorient_lmm} are oriented from 
the~$-$ to~$+$ end.
These orientations determine spin structures on the restrictions of the real part of~\eref{whcTcC_e} 
to the two irreducible components of~$\Si^{\si}_1$.
The spin structure over~$\Si_t^{\si_t}$ is then the canonical one if~$\Si_t^{\si_t}$
is obtained by gluing $1^-$ with~$2^-$ and~$1^+$ with~$2^+$.
This gluing untwists back a half-spin of $\R$ in~$\R^2$ over the first circle,
instead of completing it to a full twist.
\end{rmk}

\begin{crl}\label{realorient_crl}
Let $(\Si,\si)$, $(\wt\Si,\wt\si)$, $(\pi,\wt\fc)$, and 
$\cT,\wh\cT\!\lra\!\cU$ be as in Lemma~\ref{TSiext_lmm}.
The orientation on the restriction of the real line bundle
\BE{crealorient_crl_e}
\big(\!\det\dbar_{(\wh\cT,\wh\vph)^{\otimes2}}\big)\otimes 
\big(\!\det\dbar_{\C}\big)\lra   \De_{\R}\EE
to $\De_{\R}^*$ determined by the canonical isomorphisms of
Corollary~\ref{canonisom_crl} extends across $t\!=\!0$
if $(\Si,x_{12},\si)$ is of type~(E) or~(H1) and 
flips if $(\Si,x_{12},\si)$ is of type~(H2) or~(H3).
\end{crl}

\begin{proof}
Since $\cU$ retracts onto~$\Si$ respecting the involution~$\wt\fc$,
a real orientation on the restriction of the real bundle pair~\eref{whcTcC_e} to 
the central fiber~$(\Si,\si)$ extends to a real orientation on~\eref{whcTcC_e}.
By Corollary~\ref{canonisomExt2_crl2a}, the former induces an orientation on
the real line bundle~\eref{crealorient_crl_e} over~$\De_{\R}$.
The restriction of this orientation to the fiber over each $t\!\in\!\De_{\R}^*$
is the orientation induced by the restriction of the extended real orientation to
the fiber of~\eref{whcTcC_e} as in Corollary~\ref{canonisom_crl2a}.\\

\noindent
Suppose $(\Si,x_{12},\si)$ is of type~(E) or~(H1). 
The canonical real orientation on~\eref{whcTcC_e} over~$(\Si,\si)$ then
induces the canonical real orientation on the restriction of~\eref{whcTcC_e}
over $(\Si_t,\si_t)$ with $t\!\in\!\De_{\R}^*$.
Thus, the orientation on~\eref{crealorient_crl_e} induced by 
the canonical real orientation on~\eref{whcTcC_e}
over~$(\Si,\si)$ restricts to the canonical orientation over~$t\!\in\!\De_{\R}^*$.
This establishes the claim for types~(E) and~(H1).\\   

\noindent
Suppose $(\Si,x_{12},\si)$ is of type~(H2) or~(H3).
Fix one of the four natural real orientations on~\eref{whcTcC_e}
over~$(\Si,\si)$ and let $\ve\!\in\!\{\pm1\}$ be as in Lemma~\ref{realorient_lmm}.
Since this real orientation induces the canonical real orientation on~\eref{whcTcC_e}
over~$(\Si,\si)$ if $\ve t\!\in\!\De_{\R}^+$, 
the orientation on~\eref{crealorient_crl_e} induced by the former 
restricts to the canonical  orientation if~$\ve t\!\in\!\De_{\R}^+$.
Since the chosen real orientation on~\eref{whcTcC_e} over~$(\Si,\si)$
induces an  orientation on~\eref{whcTcC_e}
differing from the canonical one by the spin structure over precisely one 
component of~$\Si_t^{\si_t}$ if $\ve t\!\in\!\De_{\R}^-$, 
the orientation on~\eref{crealorient_crl_e} induced by the former 
restricts to the opposite of the canonical orientation if~$\ve t\!\in\!\De_{\R}^-$;
see Corollary~\ref{canonisom_crl2a}.
This establishes the claim for types~(H2) and~(H3).
\end{proof}

\begin{lmm}\label{KSext_lmm}
Suppose $g,l\!\in\!\Z^{\ge0}$ with $g\!+\!l\!\ge\!2$ and
$(\Si,x_{12},\si)$, $\cC$, and $(\pi,\wt\fc,s_1,\ldots,s_l)$ are as in~\eref{cCdfn_e}
with $\cU|_{\De_{\R}}\!\lra\!\De_{\R}$ embedded inside of the universal curve fibration
over~$\R\ov\cM_{g,l}$.
The orientation on the restriction of the real line bundle
\BE{KSlbr_e}\big(\La_{\R}^{\top}(T\R\ov\cM_{g,l})\big)^*\otimes
\La_{\R}^{\top}\big((\wch{R}^1\pi_*\cT\cC)^{\si}\big) \lra \De_{\R}\EE
to $\De_{\R}^*$ induced by the KS~isomorphism 
as in~\eref{DM_prp_e3} flips across $t\!=\!0$.
\end{lmm}

\begin{proof}
Let $x_1,x_2\!\in\wt\Si$ be the preimages of the node $x_{12}\!\in\!\Si$ as before and
$$T\wt\cC=T\wt\Si\big(\!-\!z_1^+\!-\!z_1^-\!-\!\ldots\!-\!z_l^+\!-\!z_l^-\!-\!x_1\!-\!x_2\big).$$
Denote by $\cN_{g,l}\!\subset\!\ov\cM_{g,l}$ and $\R\cN_{g,l}\!\subset\!\R\ov\cM_{g,l}$ 
the one-node strata,  by $L^{\R}\!\lra\!\R\cN_{g,l}$ the normal bundle
of $\R\cN_{g,l}$ in~$\R\ov\cM_{g,l}$, and~by $\cT\wt\cC\!\lra\!\wt\cU_{g-2,l+2}$
the twisted down vertical tangent bundle over the universal curve 
$\pi\!:\wt\cU_{g-2,l+2}\!\lra\!\cN_{g,l}$.
Let $\C_{x_{12}}\!\lra\!\Si$ be the skyscraper sheaf over~$x_{12}$.\\

\noindent
The short exact sequence of sheaves 
\BE{KSext_e5a0}0\lra \cO(\cT\cC|_{\Si}) \lra \cO\big(T\wt\cC\big)\lra  \C_{x_{12}}\lra 0
\EE
induces an exact sequence 
$$0\lra \C \lra \wch{H}^1\big(\Si;\cO(\cT\cC|_{\Si})\big) 
\lra \wch{H}^1\big(\wt\Si;\cO\big(T\wt\cC\big)\big) \lra 0$$
of complex vector spaces.
Its real part is a short exact sequence 
\BE{KSext_e5a} 0\lra \R \lra \wch{H}^1\big(\Si;\cO(\cT\cC|_{\Si})\big)^{\si} 
\lra \wch{H}^1\big(\wt\Si;\cO\big(T\wt\cC\big)\big)^{\si} \lra 0\EE
of real vector spaces.
By the definition of~$L^{\R}$, there is also a natural short exact sequence 
\BE{KSext_e5b} 
0\lra T_{\cC}\R\cN_{g,l} \lra T_{\cC}\R\ov\cM_{g,l}\lra L^{\R}|_{\cC}\lra 0\EE
of real vector spaces.\\

\noindent
By~\eref{KSext_e5a} and~\eref{KSext_e5b}, there is a canonical isomorphism
\BE{KSext_e7}\begin{split}
&\La_{\R}^{\top}(T_{\cC}\R\cN_{g,l}) \!\otimes\!
\La_{\R}^{\top}\big(\wch{H}^1\big(\wt\Si;\cO\big(T\wt\cC\big)\big)^{\si}\big)\\
&\hspace{.7in}\approx
\Big(\La_{\R}^{\top}(T_{\cC}\R\ov\cM_{g,l})\!\otimes\!
\La_{\R}^{\top}\big(\wch{H}^1\big(\Si;\cO(\cT\cC|_{\Si})\big)^{\si}\big)\Big)
\otimes L^{\R}\!\otimes\!\R\,.
\end{split}\EE
The complex vector bundles 
$$T\cN_{g,l},\wch{R}^1\pi_*\big(\cT\wt\cC\big)\lra \cN_{g,l}$$
extend over a neighborhood of $\cN_{g,l}$ in $\ov\cM_{g,l}$
as a subbundle of $T\ov\cM_{g,l}$ and a quotient bundle of $\wch{R}^1\pi_*\cT\cC$.
The KS~map induces an isomorphism between these two extensions.
Over a neighborhood of~$\cC$, these extensions can be chosen to be $\si$-invariant.
We then obtain a diagram 
$$\xymatrix{ T_{\cC}\R\cN_{g,l} \ar[r]\ar[d]_{\tn{KS}}^{\approx}&
T_{\cC_t}\R\ov\cM_{g,l} \ar[d]_{\tn{KS}}^{\approx}\ar[r]& 
L^{\R}|_{\cC}\ar[d]_{\tn{KS}}^{\approx}\\
\wch{H}^1\big(\wt\Si;\cO\big(T\wt\cC\big)\big)^{\si}
& \ar[l] \wch{H}^1\big(\Si_t;\cO(\cT\cC|_{\Si_t})\big)^{\si_t}& \ar[l]\R}$$
of vector space homomorphisms which
commutes up to homotopy of the isomorphisms given by the vertical arrows.
The KS~map for $(\wt\Si,x_1,x_2)$  induces a continuous orientation 
on the first tensor product on the right-side side in~\eref{KSext_e7}
and its extension over~$\De_{\R}$.
Thus, it is sufficient to show that for small values of $t\!\in\!\De_{\R}^*$
the KS map for $(\Si_t,\si_t)$ associates the radial vector 
\BE{KSext_e9}\frac\prt{\prt|t|}\in T_{\Si_t}\R\cM_{g,l}\EE
with the same direction of the factor~$\R$ in~\eref{KSext_e7}, 
regardless of whether $t\!\in\!\De_{\R}^+$ or $t\!\in\!\De_{\R}^-$ 
(in these two cases, the radial vector field determines opposite orientations on~$L^{\R}|_{\Si}$).\\

\noindent
We use the explicit description of the KS map at the bottom of page~11 in~\cite{Melissa}
and continue with the notation in the proof of Lemma~\ref{TSiext_lmm}.
We cover a neighborhood of~$\Si_t$ in~$\cU$ by the open~sets
$$\cU_1=\big\{(t,z_1,z_2)\!\in\!\cU_0\!:\,2|z_2|\!<\!1\big\} \qquad\hbox{and}\qquad
\cU_2=\big\{(t,z_1,z_2)\!\in\!\cU_0\!:\,2|z_1|\!<\!1\big\},$$
along with coordinate charts each of which intersects at most one of~$\cU_1$ and~$\cU_2$.
Since $z_1z_2\!=\!t$ on~$\cU_0$, the overlaps between the coordinates 
$z_1$ on~$\cU_1$ and $z_2$ on~$\cU_2$ are given~by
$$z_1\equiv f_{12}(t,z_2)= tz_2^{-1} \qquad\hbox{and}\qquad 
z_2\equiv f_{21}(t,z_1)= tz_1^{-1};$$
all other overlap maps do not depend on~$t$.
Thus, the KS map takes the tangent vector~\eref{KSext_e9} to the \v{C}ech 1-cocycle on~$\Si_t$ given~by
$$\th_{t;12}\equiv \frac{\prt f_{12}}{\prt|t|}\,\frac{\prt}{\prt z_1}
=|t|^{-1}z_1\frac{\prt}{\prt z_1}, \qquad
\th_{t;21}\equiv \frac{\prt f_{21}}{\prt|t|}\,\frac{\prt}{\prt z_2}
=|t|^{-1}z_2\frac{\prt}{\prt z_2},$$
and vanishing on all remaining overlaps.
The positive factor of $|t|^{-1}$ does not effect the orientation 
on the fiber of~\eref{KSlbr_e} over $t\!\in\!\De_{\R}^*$ 
induced  by the KS map and can be dropped above.
The resulting  \v{C}ech 1-cocycle~$\wh\th_t$ is then an extension of 
the  \v{C}ech 1-cocycle~$\wh\th_0$ on~$\Si$ given~by
\BE{KSext_e15}\wh\th_{0;12}\equiv z_1\frac{\prt}{\prt z_1} - z_2\frac{\prt}{\prt z_2}, 
\qquad
\wh\th_{0;21}\equiv -z_1\frac{\prt}{\prt z_1} + z_2\frac{\prt}{\prt z_2},\EE
and vanishing on all remaining overlaps.
For $t\!\in\!\De_{\R}*$, the positive direction of the last tensor product 
on the right-hand side of~\eref{KSext_e7} is thus given~by
$$\frac\prt{\prt|t|}\otimes \wh\th_t;$$
this orientation does not extend across $t\!=\!0$.
\end{proof}

\begin{proof}[{\bf \emph{Proof of Proposition~\ref{DMext_prp}}}] 
Suppose $(\Si,x_{12},\si)$, $\cC$, and $(\pi,\wt\fc,s_1,\ldots,s_l)$ are as in~\eref{cCdfn_e}
with $\cU|_{\De_{\R}}\!\lra\!\De_{\R}$ embedded inside of the universal curve fibration
over~$\R\ov\cM_{g,l}$.
The orientation on the restriction of the real line bundle~\eref{DMext_e}
to~$\De_{\R}^*$ provided by Proposition~\ref{DM_prp} is the tensor product~of 
\begin{enumerate}[label=(\arabic*),leftmargin=*]

\item\label{KSisom_it} the orientation~on the restriction of the real line bundle~\eref{KSlbr_e}  
to $\De_{\R}^*$ induced by the KS isomorphism,

\item\label{DIandSDisom_it}
the orientation~on the restriction of the real line bundle~\eref{SerreCrl_e}  
to $\De_{\R}^*$ induced by the Dolbeault and SD isomorphisms, 

\item\label{FMisom_it} the orientation~on the restriction of the real line bundle
$$ \big(\!\det\dbar_{(\wh\cT\cC,\wh\vph)\otimes(\wh\cT,\wh\vph)}\big)
\otimes  \big(\!\det\dbar_{(\wh\cT,\wh\vph)^{\otimes2}}\big) \lra   \De_{\R}$$
to $\De_{\R}^*$ induced by the short exact sequences~\eref{FMses_e}
and the specified orientations of~\eref{H0ScC_e}, 

\item\label{SQisom_it} the orientation~on the restriction of the real line 
bundle~\eref{crealorient_crl_e} to $\De_{\R}^*$
determined by the canonical isomorphisms of
Corollary~\ref{canonisom_crl}. 

\end{enumerate}
Since the family of the short exact sequences~\eref{FMses_e} and 
the specified orientations of~\eref{H0ScC_e}
extend across $t\!=\!0$, so does the orientation in~\ref{FMisom_it}.
By Corollary~\ref{Serre_crl}, the orientation in~\ref{DIandSDisom_it} also extends across $t\!=\!0$.
By Lemma~\ref{KSext_lmm}, the orientation in~\ref{KSisom_it} flips across $t\!=\!0$.
By Corollary~\ref{realorient_crl}, the orientation in~\ref{SQisom_it} 
extends across $t\!=\!0$ if $(\Si,x_{12},\si)$ is of type~(E) or~(H1) and 
flips if $(\Si,x_{12},\si)$ is of type~(H2) or~(H3).
Combining these four statements, we obtain the claim.
\end{proof}

\subsection{Proofs of the main statements}
\label{MainPfs_subs}

\noindent
We now establish the main statements of this paper,
Theorems~\ref{orient_thm} and~\ref{dim3_thm}.

\begin{proof}[{\bf \emph{Proof of Theorem \ref{orient_thm}}}] 
By Corollary~\ref{orient0_crl}, a real orientation on $(X,\om,\phi)$ determines
an orientation on the restriction of the real line bundle~\eref{orient_thm_e}
to the uncompactified moduli space
$$\fM_{g,l}(X,B;J)^{\phi,\si}\subset \ov\fM_{g,l}(X,B;J)^{\phi}$$
for every topological type~$\si$ of genus~$g$ orientation-reversing involutions.
We show that  these orientations multiplied by $(-1)^{g+|\si|_0+1}$
extend across the codimension-one strata of $\ov\fM_{g,l}(X,B;J)^{\phi}$.\\

\noindent
Suppose $[u,(z_1^+,z_1^-),\ldots,(z_l^+,z_l^-),\fJ]$ is a stable real morphism
from a one-nodal symmetric surface~$(\Si,\si)$.
Since the fibers of the forgetful morphism
$$\ov\fM_{g,l+1}(X,B;J)^{\phi}\lra \ov\fM_{g,l}(X,B;J)^{\phi}$$
are canonically oriented, we can assume that 
$$\cC\equiv \big(\Si,(z_1^+,z_1^-),\ldots,(z_l^+,z_l^-),\fJ\big)$$
is a stable symmetric surface and thus defines an element of $\R\ov\cM_{g,l}$.
The canonical isomorphism~\eref{thm_maps_e3} then extends across~$[u]$.
By Corollary~\ref{canonisomExt2_crl2a}, the canonical orientation on the restriction 
of the real line bundle~\eref{thm_maps_e9} to $\fM_{g,l}(X,B;J)^{\phi}$ also 
extends across~$[u]$.
Since $(-1)^{g+|\si|_0+1}$ flips 
across the codimension-one boundary strata of types~(E) and~(H1) and 
extends across the codimension-one boundary strata of types~(H2) and~(H3),
the claim now follows from Proposition~\ref{DMext_prp}.
\end{proof}

\begin{proof}[{\bf \emph{Proof of Theorem \ref{dim3_thm}}}] 
For $J\!\in\!\cJ_{\om}^{\phi}$ and a (real) perturbation~$\nu$ of 
the real $\dbar_J$-equation, we denote by $\ov\fM_{1,l;k}(X,B;J,\nu)^{\phi}$ 
the moduli space of real genus~1 degree~$B$ $(J,\nu)$-maps 
with $l$~conjugate pairs of marked points and $k$ real marked points.
For $k\!=\!0$, we omit the corresponding subscript.
If $(X,\om,\phi)$ is semi-positive in the sense of \cite[Definition~1.2]{RealRT},
then $\nu$ can be taken to be a real Ruan-Tian perturbation as defined 
in \cite[Section~3.1]{RealRT}.
In general, $\nu$ is a perturbation in the sense of Kuranishi structures.\\

\noindent
By Theorem~\ref{orient_thm}, the compactified moduli space 
$\ov\fM_{1,l}(X,B;J,\nu)^{\phi}$ is orientable. 
Thus, the orientability of $\ov\fM_{1,l;k}(X,B;J,\nu)^{\phi}$ is determined by 
the orientability of the vertical tangent bundle of the forgetful morphism
\BE{Rforg_e} \ov\fM_{1,l;k}(X,B;J,\nu)^{\phi}\lra \ov\fM_{1,l}(X,B;J,\nu)^{\phi}\EE
dropping the real marked points.
The fibers of~\eref{Rforg_e} over the main strata 
$$\fM_{1,l}(X,B;J,\nu)^{\phi,\si}\subset \ov\fM_{1,l}(X,B;J,\nu)^{\phi}$$
are open subsets of $(S^1)^k$.\\

\noindent 
Since there are diffeomorphisms $h\!\in\!\cD_{\si}$
which reverse an orientation on the fixed locus,
the vertical tangent bundle of~\eref{Rforg_e} is not orientable over 
$\fM_{1,l}(X,B;J,\nu)^{\phi,\si}$ if $k$ is odd.
If $k$ is even, the fibers of~\eref{Rforg_e} are canonically oriented as follows.
If $|\si|_0\!=\!1$, an orientation on the fixed locus determines an orientation on
each  fiber of~\eref{Rforg_e} which is independent of the choice of the first orientation.
If $|\si|_0\!=\!2$, the fixed locus~$\Si^{\si}$ splits~$\Si$ into two annuli;
let~$\Si^b$ be either of these annuli.
Endow one of the boundary circles of~$\Si^b$ with the induced boundary orientation 
and the other with the opposite of the induced boundary orientation.
These choices determine an orientation on each fiber of~\eref{Rforg_e}.
Since $k$ is even, this orientation is independent of which circle is oriented as a boundary
and thus of the choice of the half~$\Si^b$.
We determine the orientability of the vertical tangent bundle over $\ov\fM_{1,l}(X,B;J,\nu)^{\phi,\si}$
by studying how these canonical orientations change across the codimension-one boundary strata.\\

\noindent
If $g\!=\!1$, the codimension-one boundary strata can be of types~(E), (H1), and~(H3) only. 
If $k\!>\!0$, the domains of all morphisms of type~(E)
are one-nodal symmetric surfaces $(\Si,x_{12},\si)$ with the fixed locus consisting
of the node~$x_{12}$ and a fixed circle $\Si^{\si}_1$ containing all of the real marked points.
The canonical orientations on the fibers of~\eref{Rforg_e} extend across such strata.\\

\noindent
In the (H1)~case, the nodal symmetric surface~$(\Si,\si)$ is $(\P^1,\tau)$ 
with two real points identified.
In particular, the fixed locus~$\Si^{\si}$ splits~$\Si$ into two copies of 
a disk with two boundary points identified;
denote by~$\Si^b$ either of these copies and by $x_{12}\!\in\!\Si^{\si}$ the node.
Let $(\cT,\vph)$ be the real bundle pair over a one-parameter family of smoothings of~$(\Si,\si)$
as in Lemma~\ref{TSiext_lmm}.
An orientation on $\cT^{\vph}|_{\Si^{\si}}\!\lra\!\Si^{\si}$ induces an orientation 
on~$T\Si_t^{\si_t}$ for every smoothing of~$(\Si_t,\si_t)$.
By the matching condition on $\Om(\cT|_{\Si_0})$ in Lemma~\ref{TSiext_lmm},
the orientation~on
$$ \cT^{\vph}|_{\Si^{\si}-x_{12}} = T\big(\Si^{\si}\!-\!x_{12}\big)$$
as the boundary of~$\Si^b$ does not extend over~$x_{12}$.
This implies that the orientation on~$\Si_t^{\si_t}$ with $|\si_t|_0\!=\!2$
induced by an orientation on $\cT^{\vph}|_{\Si^{\si}}$ is not 
the boundary orientation from either of the annuli obtained by cutting~$\Si_t$
along~$\Si_t^{\si_t}$.
Thus, the canonical orientations on the fibers of~\eref{Rforg_e} extend across 
the (H1) boundary strata as~well.\\ 

\noindent
In the (H3)~case, the nodal symmetric surface~$(\Si,\si)$ consists of a genus~1 surface with 
a sphere bubble attached.
A choice of an orientation on~$\Si^{\si}$ is compatible with the orientation of 
the fixed locus on only one side of the boundary. 
If the number of the real marked points on either the torus or the sphere is even, 
then the orientation of the fibers of~\eref{Rforg_e} still extends across this stratum.
We will call the codimension-one boundary strata of type~(H3) with odd numbers of real marked
points on the torus and the sphere to be of type~(H3$^-$).
Following the approach of \cite{Cho,Sol}, 
we show that in a generic one-parameter family the cut-down moduli space 
does not cross such strata  and 
thus the counting invariant~\eref{dim3_thm} is well-defined.\\

\noindent
Let
\begin{gather*}
\ev\!:\fM_{1,l;k}^*(X,B;J,\nu)^{\phi}\lra X^l\!\times\!(X^{\phi})^k, \\
\big[u,(z_1^+,z_1^-),\ldots,(z_l^+,z_l^-),x_1,\ldots,x_k,\fJ\big]\lra
\big(u(z_1^+),\ldots,u(z_l^+),u(x_1),\ldots,u(x_k)\big),
\end{gather*}
be the total evaluation map from the moduli space of simple $(J,\nu)$-maps.
Choose pseudocycle representatives
$$h_1\!:Y_1\lra X, \quad\ldots, \quad h_l\!:Y_l\lra X$$
for the Poincare duals of $\mu_1,\ldots,\mu_l$;
this is possible to do by \cite[Theorem~1.1]{pseudo}.
We can assume~that 
\BE{eq_dimc}\sum_{i=1}^{l}(\deg\mu_i\!-\!2)+2k=\blr{c_1(X),B}\EE
and so $k$ is even under our assumptions.
Choose $k$ real points $p_1,\ldots,p_k\!\in\!X^{\phi}$.
If $(X,\om,\phi)$ is semi-positive, $(J,\nu)$ is generic,
and $h_1,\ldots,h_l,p_1,\ldots,p_k$ are chosen generically,
then $\ev$ is transverse to the pseudocycle
$$\prod_{i=1}^lY_i\lra X^l\!\times\!(X^{\phi})^k, \quad
\big(y_1,\ldots,y_l\big)\lra 
\big(h_1(y_1),\ldots,h_l(y_l),p_1,\ldots,p_k\big) \,.$$
The intersection of~$\ev$ with this pseudocycle, i.e.
\begin{equation*}\begin{split}
&\fM_{1,l;k}^*(X,B;J,\nu)_{h_1,\ldots,h_l;p_1,\ldots,p_k}^{\phi}\\
&\quad\equiv
\Big\{\big([u,(z_1^+,z_1^-),\ldots,(z_l^+,z_l^-),x_1,\ldots,x_k,\fJ],y_1,\ldots,y_l\big)
\in\!\fM_{1,l;k}^*(X,B;J,\nu)^{\phi}\!\times\!\prod_{i=1}^lY_i\!:\\
&\hspace{2.35in}u(z_i^+)\!=\!h_i(y_i)~\forall\,i\!=\!1,\ldots,\ell,~u(x_i)\!=\!p_i~\forall\,i\!=\!1,\ldots,k\Big\},
\end{split}\end{equation*}
is then a zero-dimensional manifold.
A real orientation on~$(X,\om,\phi)$ and the canonical orientation 
on the vertical tangent bundle of~\eref{Rforg_e} determine an orientation of this manifold.
We~set
$$\blr{\mu_1,\ldots,\mu_l;\pt^k; J,\nu}_{1,B}^{\phi}
=\, ^{\pm}\!\big|\fM_{1,l;k}^*(X,B;J,\nu)_{h_1,\ldots,h_l;p_1,\ldots,p_k}^{\phi}\big|$$
to be the signed cardinality of this set.\\

\noindent
Let $(J_1,\nu_1)$ and $(J_2,\nu_2)$ be two regular $\phi$-invariant pairs  
and $\{J_t,\nu_t\}$ be a generic path between them.
If $(X,\om,\phi)$ is semi-positive, 
the image of the $\phi$-multiply covered maps is of codimension at least~2;
a generic path of cut-down moduli spaces thus avoids~them.
Along the path $\{J_t,\nu_t\}$, the cut-down moduli space forms a one-dimensional bordism
and contains finitely many points in the codimension-one boundary strata of type~(H3$^-$).  
We orient this bordism outside of the (H3$^-$) elements as the preimage~of the~submanifold
$$\big\{(q_1,\ldots,q_l,p_1,\ldots,p_k,q_1,\ldots,q_l)\!:\,q_1,\ldots,q_l\!\in\!X\big\}
\subset X^l\!\times\!(X^{\phi})^k\!\times\!X^l$$
under the transverse morphism 
\begin{align*}
\ev\times h_1\!\times\!\ldots\!\times\!h_l\!:
\bigcup_{t\in[0,1]} \!\!\! \{t\}\!\times\!\fM_{1,l;k}^*(X,B;J_t,\nu_t)
\times\prod_{i=1}^l Y_i \lra X^{l}\!\times\!(X^\phi)^k\!\times\!X^l.
\end{align*}
The signed cardinalities of the boundaries of this bordism over $t\!=\!0$
and $t\!=\!1$ are  
\BE{condend_e}- \blr{\mu_1,\ldots,\mu_l;\pt^k; J_0,\nu_0}_{1,B}^{\phi}
\qquad\hbox{and}\qquad  
\blr{\mu_1,\ldots,\mu_l;\pt^k; J_1,\nu_1}_{1,B}^{\phi},\EE  
respectively.\\

\noindent
Suppose that in a one-parameter family the cut-down moduli space crosses
a codimension-one boundary stratum of type~(H3) with the map degree splitting 
into classes $B_1,B_2\!\in\!H_2(X;\Z)$ between the genus~1 surface and 
the sphere bubble, respectively.
Let $l_1,l_2\!\in\!\Z^{\ge0}$ be the numbers of conjugate pairs of marked points 
carried by the two components and $k_1,k_2\!\in\!\Z^{\ge0}$
be the numbers of real marked points carried by~them.
Thus,
$$B_1\!+\!B_2=B, \qquad l_1\!+\!l_2=l, \qquad k_1\!+\!k_2=k.$$
By a dimension count, this can happen  only~if 
$$ \sum_{i=1}^{l_1}(\deg\mu_{j_i}\!-\!2)+2k_1\le \blr{c_1(X),B_1}+1 \quad\text{and}\quad
 \sum_{i=1}^{l_2 }(\deg\mu_{j_i}\!-\!2)+2k_2 \le \blr{c_1(X),B_2}+1.$$
Using (\ref{eq_dimc}), we obtain  
$$\sum_{i=1}^{l_2 }(\deg\mu_{j_i}\!-\!2)+2k_2-1 \le  \blr{c_1(X),B_2}\le
 \sum_{i=1}^{l_2 }(\deg\mu_{j_i}\!-\!2)+2k_2+1\,. $$
Since $\deg\mu_{j_i}\!-\!2$ and $\lr{c_1(X),B_2}$ are divisible by~4, this implies that $k_2$ is even
and that the codimension-one boundary strata of type~(H3$^-$) are never crossed.
Thus, the canonical orientations extend over the whole cobordism
and the two counts in~\eref{condend_e} are equal.\\ 

\noindent
A similar cobordism argument holds for a semi-positive deformation of $\om$ 
and for a change of the pseudocycle representatives.
The general case is treated using Kuranishi structures similarly to \cite[Section~7]{Sol}.
\end{proof}

\appendix

\section{Topologizing determinant line bundles}
\label{detLB_app}

\noindent
The existence of topologies on the total space of~\eref{detLB_e} with good properties
is readily implied by the main algebraic conclusion of~\cite{KM}
in combination with some of the analytic results obtained in \cite[Section~3]{LT}.
The latter ensure that the kernels of surjective CR-operators 
$D_{(V,\vph);\t}$ in~\eref{DVvphdfn_e} and their extensions~$D_{\Th;\t}$
in~\eref{DThtdfn_e}
form vector bundles over~$\De$ and thus so do the determinants of these operators
(the determinants are then the top exterior powers of the kernels);
see Proposition~\ref{LTtopol_prp}.
In the general case, the bundle isomorphisms~\eref{cIThD_e2} would  topologize
the total space of~\eref{detLB_e} from the determinant line bundles of surjective operators 
if the resulting overlap maps between the latter 
are continuous.
For the property~\ref{dses_it} on page~\pageref{dses_it} to hold,
{\it all} isomorphisms~\eref{cIThD_e2} and thus the aforementioned overlap maps
{\it need to} be~continuous.\\

\noindent
The isomorphisms~\eref{cIThD_e2} are special cases of the isomorphisms~$\wh\cI_{\Th;D}$
in~\eref{whcIcIdfn_e}; the latter are induced by the isomorphisms~\eref{sum}
associated with the exact triples~\eref{cIhatdiag_e} of Fredholm operators.
The isomorphisms~\eref{sum} associated with the exact triples~\eref{compEX_e}
induce the isomorphisms~$\cI_{\Th;D}$ in~\eref{whcIcIdfn_e}
going in the opposite direction.
The property~\ref{dses_it} and thus the continuity of the isomorphisms in~\eref{whcIcIdfn_e}
would be implied by the two purely algebraic Compositions properties of \cite[Section~2]{detLB}
for homomorphisms between {\it finite}-dimensional vector spaces.
By the main algebraic conclusion of~\cite{KM}, it is possible to {\it choose} 
the isomorphisms~\eref{sum} for exact triples of operators
between {\it finite}-dimensional vector spaces so that
they satisfy these two properties.
Furthermore, the resulting topologies on the determinant line bundles~\eref{detLB_e} satisfy
all properties in \cite[Section~2]{detLB}
with Fredholm spaces replaced by real bundle pairs.
In fact, the choice of a good collection of the isomorphisms~\eref{sum} for
 non-surjective operators between finite-dimensional vector spaces is not unique.
However, any two choices induce topologies that differ by homeomorphisms 
intertwining all isomorphisms between determinant line bundles listed in
\cite[Section~2]{detLB}; see \cite[Theorem~2]{detLB}.\\

\noindent
As shown in \cite[Appendix~D.2]{Huang} and \cite[Section~3.2]{detLB}, 
the topologies on the determinant line bundles over families of Fredholm operators 
between {\it fixed} Banach spaces arise from {\it exactly} the same algebraic considerations.   
The only difference is that 
the analogue of Proposition~\ref{LTtopol_prp} for continuous families of 
surjective Fredholm operators between {\it fixed} Banach spaces is
straightforward.

\subsection{Linear algebra}
\label{KMdetLB_subs}

\noindent
We begin by recalling the relevant algebraic facts from~\cite{detLB}.
We denote by
$$\OmN \equiv e_1\!\w\!\ldots\!\w\!e_N$$ 
the standard volume tensor on~$\R^N$ and by $\Om_N^*\!\in\!(\La_{\R}^N\R^N)^*$
its dual.
For a Banach space~$X$ and \hbox{$N,N_1,N_2\!\in\!\Z^{\ge0}$}, define
\begin{alignat*}{2}
\io_{X;N}\!:X&\lra X\!\oplus\!\R^N, &\qquad \io_{X;N}(x)&=(x,0), \\
R_{X;N_1,N_2}\!:X\!\oplus\!\R^{N_1}\!\oplus\!\R^{N_2}&\lra 
X\!\oplus\!\R^{N_2}\!\oplus\!\R^{N_1}, &\qquad 
R_{X;N_1,N_2}(x,v_1,v_2)&=(x,v_2,v_1).
\end{alignat*}
For vector space homomorphisms $\Th\!:\R^N\!\lra\!Y$ and \hbox{$R:\R^{N'}\!\lra\!\R^N$}
and a Fredholm operator \hbox{$D\!:X\!\lra\!Y$},
define
\begin{alignat*}{2}
D_{\Th}\!:X\!\oplus\!\R^N&\lra Y, &\qquad D_{\Th}(x,v)&=Dx+\Th(v), \\
R_{\Th;D}\!:\ker D_{\Th\circ R}&\lra\ker D_{\Th}, &\qquad R_{\Th;D}(x,v')&=(x,Rv').
\end{alignat*}
In particular, the triple $\ft_{\Th;D}$ 
\BE{cIhatdiag_e}\begin{split}
\xymatrix{ 0\ar[r]& X\ar[d]^D\ar[r]& X\!\oplus\!\R^N 
\ar[d]^{D_{\Th}}\ar[r]^<<<<{\pi_2}&  \R^N \ar[d]\ar[r] & 0\\
0\ar[r]& Y \ar[r]& Y\ar[r]& 0\ar[r]& 0}
\end{split}\EE
of Fredholm operators is exact.\\

\noindent
A \sf{homomorphism between Fredholm operators $D\!:X\!\lra\!Y$ and
$D'\!:X'\!\lra\!Y'$} is a pair of homomorphisms $\phi\!:X\!\lra\!X'$ 
and $\psi\!:Y\!\lra\!Y'$ so that $D'\!\circ\!\phi\!=\!\psi\!\circ\!D$;
an \textsf{isomorphism between Fredholm operators $D$ and~$D'$} 
is a homomorphism $(\phi,\psi)\!:D\!\lra\!D'$
so that $\phi$ and $\psi$ are isomorphisms.
Such an isomorphism induces isomorphisms
\begin{gather*}
\det\phi\!:\La^{\top}_{\R}(\ker D)\stackrel{\approx}{\lra}\La^{\top}_{\R}(\ker D'), \qquad
\det\psi^{-1}\!:\La^{\top}_{\R}(\cok\,D')\stackrel{\approx}{\lra}\La^{\top}_{\R}(\cok\,D),\notag\\
\label{cIphipsi_e}
\cI_{\phi,\psi;D}\!:\det D\stackrel{\approx}{\lra} \det D', \quad
x\!\w\!\al\lra  \big(\{\det\phi\}(x)\big)\!\w\!\big(\al\!\circ\!\{\det\psi^{-1}\}\big).
\end{gather*}
For homomorphisms $\Th_1\!:\R^{N_1}\!\lra\!Y$ and $\Th_2\!:\R^{N_2}\!\lra\!Y$, let 
$$\cR_{\Th_1,\Th_2;D}\!=\!\cI_{R_{X;N_1,N_2},\id_Y;D_{\Th_1\oplus\Th_2}}\!:
\det D_{\Th_1\oplus\Th_2}\stackrel{\approx}{\lra} \det D_{\Th_2\oplus\Th_1}\,.$$ 

\vspace{.2in}

\noindent
A pair of Fredholm operators $D_1\!:X_1\!\lra\!X_2$ and $D_2\!:X_2\!\lra\!X_3$
determines an exact triple
\BE{compEX_e}\begin{split}
\xymatrix{0\ar[r]& X_1\ar[d]^{D_1}
\ar[r]^>>>>>{\fI_X}& X_1\!\oplus\!X_2\ar[d]|{D_2\circ D_1\oplus\id_{X_2}}\ar[r]^<<<<<{\fJ_X}& 
X_2\ar[d]^{D_2}\ar[r]&0
&*\txt{$~~~~\fI_X(x_1)=(x_1,D_1x_1)$\\ $\fJ_X(x_1,x_2)=D_1x_1\!-\!x_2$}\\
0\ar[r]& X_2\ar[r]^>>>>>{\fI_Y}& X_3\!\oplus\!X_2\ar[r]^<<<<<{\fJ_Y}& X_3\ar[r]&0
&*\txt{$~~~~~\fI_Y(x_2)=(D_2x_2,x_2)$\\ $\fJ_Y(x_3,x_2)=x_3\!-\!D_2x_2$\,,}}
\end{split}\EE
of Fredholm operators.
The isomorphism~\eref{sum} for this triple~becomes
\BE{cCisom_e}\wt\cC_{D_1,D_2}\!:\big(\!\det D_1\!\big)\!\otimes\!\big(\!\det D_2\!\big)
\stackrel{\approx}{\lra} \det(D_2\!\circ\!D_1)\,.\EE
If $D$, $\Th$, and $\pi_2$ are as above, \hbox{$D\!=\!D_{\Th}\!\circ\!\io_{X;N}$} and 
the projection
$$\pi_2\!: \cok\,\io_{X;N}\lra\R^N$$
is an isomorphism.
We thus obtain two isomorphisms induced by~\eref{sum},
\BE{whcIcIdfn_e}\begin{aligned}
\wh\cI_{\Th;D}\!: \det D&\stackrel{\approx}{\lra} \det D_{\Th}, &\qquad
\wh\cI_{\Th;D}(\vp)&=\Psi_{\ft_{\Th;D}}\big(\vp\!\otimes\!\OmN\!\otimes\!1^*\big),\\
\cI_{\Th;D}\!: \det D_{\Th}&\stackrel{\approx}{\lra} \det D, &\qquad
\cI_{\Th;D}(\vp)&=\cC_{\io_{X;N},D_{\Th}}\big(1\!\otimes\!
(\OmN^*\!\circ\!\{\det\pi_2\})\otimes\vp\big),
\end{aligned}\EE
via~\eref{cIhatdiag_e} and~\eref{cCisom_e}, respectively.\\

\noindent
Every short exact sequence
\BE{sesVdfn_e}0\lra V'\lra V \lra V''\lra 0\EE
of finite-dimensional vector spaces determines an isomorphism
\BE{FDsum_e}\big(\La_{\R}^{\top}V'\big)\!\otimes\!\big(\La_{\R}^{\top}V''\big)
\stackrel{\approx}{\lra} \La_{\R}^{\top}V\EE
between the top exterior powers of the vector spaces involved;
see \cite[Lemma~4.1]{detLB}.
By the Snake Lemma, an exact triple \eref{cTdiag_e} induces an exact sequence 
\BE{ETles_e} 0\lra \ker D'\stackrel{\fI_X}{\lra} \ker D \stackrel{\fJ_X}{\lra} \ker D''
\stackrel{\de}{\lra} \cok\,D'  \stackrel{\fI_Y}{\lra} \cok\,D
\stackrel{\fJ_Y}{\lra} \cok\,D''\lra0\EE
of finite-dimensional vector spaces.
It is equivalent to four {\it short} exact sequences, such as 
$$0\lra \ker D'\stackrel{\fI_X}{\lra} \ker D \stackrel{\fJ_X}{\lra}\Im\,\fJ_X\lra0.$$
The isomorphisms~\eref{sum} should clearly be induced by the isomorphisms~\eref{FDsum_e}
corresponding to these four short exact  sequences.
However, there are at least choices of signs involved in putting the four resulting isomorphisms
together, depending on the dimensions of the vector spaces appearing 
in the four sequences.
Choosing these signs in some compatible fashion is necessary to ensure
that the isomorphisms~\eref{cIThD_e2} used to topologize determinant line bundles
overlap continuously.\\

\noindent
If the operators in~\eref{cTdiag_e} are surjective, 
the exact sequence~\eref{ETles_e} reduces~to the exact sequence
$$0\lra \ker D'\stackrel{\fI_X}{\lra} \ker D \stackrel{\fJ_X}{\lra} \ker D''\lra0.$$
It is then standard to require that the corresponding isomorphism~\eref{sum}
be given by the isomorphism~\eref{FDsum_e} associated with this exact sequence
of kernels; this property is Normalization~II in \cite[Section~2]{detLB}.
An explicit formula for the isomorphism~\eref{sum} in the general case
with this property is given by \cite[(4.10)]{detLB}.
The induced isomorphisms~\eref{cCisom_e} satisfy the two 
algebraic Compositions properties in \cite[Section~2]{detLB}
and thus the remaining algebraic properties listed there 
(Naturality~II and~III and Exact Squares); see the paragraph
after Theorem~1 in \cite[Section~2]{detLB}.
The associated isomorphisms~\eref{whcIcIdfn_e} satisfy
\begin{gather}\label{whcIcIprp_e1}
\cI_{\Th;D}\!\circ\!\wh\cI_{\Th;D}=(-1)^{(\ind\,D)N}\id\!:
\det D\stackrel{\approx}{\lra}\det D, \\
\label{whcIcIprp_e2}
\cI_{\Th_2;D}^{-1}\!\circ\!\cI_{\Th_1;D}=(-1)^{N_1N_2}
\cI_{\Th_1;D_{\Th_2}}\!\circ\!\cR_{\Th_1,\Th_2;D}\!\circ\!\cI_{\Th_2;D_{\Th_1}}^{-1}
\!:\det D_{\Th_1}\stackrel{\approx}{\lra} \det D_{\Th_2}
\end{gather}
for all Fredholm operators $D\!:X\!\lra\!Y$ and homomorphisms
$$\Th\!:\R^N\lra Y, \qquad  \Th_1\!:\R^{N_1}\lra Y, \qquad \Th_2\!:\R^{N_2}\lra Y\,;$$
see \cite[Lemma~4.11]{detLB} and the end of the proof of \cite[Proposition~5.3]{detLB}.

\subsection{Analysis and topology}
\label{LTdetLB_subs}

\noindent
Let $(\Si_0,\si_0,\fJ_0)$, $(\pi,\wt\fc)$, and $(V,\vph)$
be as above~\eref{detLB_e}.
Fix $p\!>\!2$, $\wt\fc$-invariant Riemannian metric on~$\cU$, and 
a $\vph$-invariant metric on~$V$.
For each $\t\!\in\!\De_{\R}$, we denote by 
$$\cE_{\t}(V)^{\vph}\supset \Ga\big(\Si_{\t};V|_{\Si_{\t}}\big)^{\vph}
\qquad\hbox{and}\qquad 
\cE_{\t}^{0,1}(V)^{\vph}\supset \Ga_{\fJ_{\t}}^{0,1}\big(\Si_{\t};V|_{\Si_{\t}}\big)^{\vph}$$
the completions of the spaces of smooth $(\vph,\si_{\t})$-invariant bundle sections
in the modified $L^p_1$- and $L^p$-norms $\|\cdot\|_{p,1}$ and $\|\cdot\|_p$,
respectively, introduced in \cite[Section~3]{LT}. 
The norms $\|\cdot\|_{p,1}$ and $\|\cdot\|_p$ dominate the usual $L^p_1$- and $L^p$-norms,
but are equivalent to them away from the nodes of~$\Si_{\t}$.
Some of the key properties of these norms are summarized by the next statement.
Let
\BE{Dtdfn_e}D_{\t}\!:\cE_{\t}(V)^{\vph}\lra \cE_{\t}^{0,1}(V)^{\vph}\EE
be the operator induced by $D_{(V,\vph);\t}$.

\begin{lmm}\label{LTnorms_lmm}
For every $\t^*\!\in\!\De_{\R}$, there exist a neighborhood~$\De_{\t^*}$ of~$\t^*$ in~$\De_{\R}$
and $C_{\t^*}\!\in\!\R^+$ such~that 
$$\|\xi\|_{C^0}\le C_{\t^*}\|\xi\|_{p,1},\quad 
\big\|D_{\t}\xi\big\|_p\le  C_{\t^*}\|\xi\|_{p,1}, \quad
\|\xi\|_{p,1}\le C_{\t^*}\big(\|D_{\t}\xi\big\|_p\!+\!\|\xi\|_p \big)$$
for all $\xi\!\in\!\cE_{\t}(V)^{\vph}$ and $\t\!\in\!\De_{\t^*}$.
\end{lmm}

\noindent
The second inequality above is immediate from the definition of the norms
$\|\cdot\|_{p,1}$ and $\|\cdot\|_p$.
The first inequality holds even with the standard $L^p_1$-norm on the right-hand side;
see \cite[Lemma~3.2]{LT} and \cite[Proposition~4.10]{anal}.
The last inequality is the crucial uniform elliptic estimate of \cite[Lemma~3.9]{LT};
see the proof of \cite[Proposition~5.11]{gluing} and \cite[Section~4.3]{anal}
for more details.
By Lemma~\ref{LTnorms_lmm}, \eref{Dtdfn_e} is a Fredholm operator;
its index, which we denote by $\ind\,D_{(V,\vph)}$, does not depend on~$\t$.\\

\noindent
The normed topologies on the fibers of the projections
$$\cE(V)^{\vph}\!\equiv\!\bigsqcup_{\t\in\De_{\R}}\!\!\!
\big(\{\t\}\!\times\!\cE_{\t}(V)^{\vph}\big) \lra\De_{\R} \quad\hbox{and}\quad
\cE^{0,1}(V)^{\vph}\!\equiv\!\bigsqcup_{\t\in\De_{\R}}\!\!\!
\big(\{\t\}\!\times\!\cE_{\t}^{0,1}(V)^{\vph}\big) \lra\De_{\R}$$
are extended to topologies on $\cE(V)^{\vph}$ and $\cE^{0,1}(V)^{\vph}$
in \cite[Section~3]{LT}.
These topologies are described as follows.
Let $\t\!\in\!\De_{\R}$ and 
$$\psi_{\t'}\!:\Si_{\t}^*\lra q^{-1}(\Si_{\t}^*)\!\cap\!\Si_{\t'}$$
be analogues of the diffeomorphisms~\eref{psitdfn_e} defined for $\t'\!\in\!\De$
in a neighborhood of~$\t$.
For each $\de\!\in\!\R^+$, denote by $B_{\t;\de}'\!\subset\!\cU$ the $\de$-neighborhood
of the nodes of~$\Si_{\t}$.
Suppose $\t_r\!\in\!\De_{\R}$ is a sequence converging to~$\t$.
A sequence $\xi_r\!\in\!\cE_{\t_r}(V)^{\vph}$  \sf{converges} to 
$\xi\!\in\!\cE_{\t}(V)^{\vph}$ if 
\begin{enumerate}[label=(\alph*),leftmargin=*]

\item the sequence $\xi_r\!\circ\!\psi_{\t_r}$ converges to $\xi$ in
the $L^p_1$-norm on compact subsets of $\Si_{\t}$ and 

\item $\displaystyle\lim_{\de\lra0}\lim_{r\lra\i}
\big\|\xi_r|_{B_{\t;\de}'\cap\Si_{\t_r}}\big\|_{p,1}\!=\!0$.

\end{enumerate}
The topology on $\cE^{0,1}(V)^{\vph}$ introduced in~\cite{LT} is described
analogously, with the $L^p_1$-norms replaced by $L^p$-norms.\\

\noindent
For any bundle homomorphism
\BE{Thdfn_e}\Th\!: \De_{\R}\!\times\!\R^N\lra \cE^{0,1}(V)^{\vph},\qquad
\Th(\t,v)=\big(\t,\ze_{\t,v}\big),\EE
and $\t\!\in\!\De_{\R}$, let
$$\Th_{\t}\!:\R^N\lra \cE_{\t}^{0,1}(V)^{\vph}, \qquad
\Th_{\t}(v)=\ze_{\t,v}\,,$$
be the restriction of $\Th$ to the fiber over~$\t$.
Define
\begin{gather} 
\label{DThtdfn_e}
D_{\Th;\t}\!=\!(D_{\t})_{\Th_{\t}}\!: 
\cE_{\t}(V)^{\vph}\!\oplus\!\R^N\lra\cE_{\t}^{0,1}(V)^{\vph},\quad
U_{\Th}=\big\{\t\!\in\!\De_{\R}\!:\cok\,D_{\Th;\t}\!=\!\{0\}\!\big\},\\
\label{kerDThdfn_e}
\ker D_{\Th}=\big\{(\t,\xi,v)\!\in\!\cE(V)^{\vph}\!\times\!\R^N\!\!:\t\!\in\!U_{\Th},\,
D_{\Th;\t}(\xi,v)\!=\!0\big\}\lra U_{\Th}\,,\\
\notag
\wh\cI_{\Th;\t}\!=\!\wh\cI_{\Th_{\t};D_{\t}}\!: 
\det D_{\t}\stackrel{\approx}{\lra} \det D_{\Th;\t}, \quad
\cI_{\Th;\t}\!=\!\cI_{\Th_{\t};D_{\t}}\!: 
\det D_{\Th;\t}\stackrel{\approx}{\lra}\det D_{\t}\,.
\end{gather}
If in addition 
\begin{gather}
\label{Rbnddfb_e}
R:\De_{\R}\!\times\!\R^{N'}\lra\De_{\R}\!\times\!\R^N, \\
\label{Th12bnddfn_e}
\Th_1\!:\De_{\R}\!\times\!\R^{N_1}\lra \cE^{0,1}(V)^{\vph},\quad
\Th_2\!:\De_{\R}\!\times\!\R^{N_2}\lra \cE^{0,1}(V)^{\vph}
\end{gather}
are bundle homomorphisms, let
\begin{gather*}
R_{\Th;\t}\!=\!R_{\Th_{\t};D_{\t}}\!:\ker D_{\Th\circ R;\t}\lra\ker D_{\Th;\t},\\
\cR_{\Th_1,\Th_2;\t}\!=\!\cR_{(\Th_1)_{\t},(\Th_2)_{\t};D_{\t}}\!:
\det D_{\Th_1\oplus\Th_2;\t}\stackrel{\approx}{\lra} \det D_{\Th_2\oplus\Th_1;\t}\,.
\end{gather*}
By~\eref{whcIcIprp_e1} and~\eref{whcIcIprp_e2},
\begin{gather}\label{whcIcIprp_e3}
\cI_{(\Th_2)_{\t};D_{\Th_1;\t}}\!\circ\!\wh\cI_{(\Th_2)_{\t};D_{\Th_1;\t}}
=(-1)^{(\ind\,D_{(V,\vph})+N_1)N_2}\id\!:
\det D_{\Th_1;\t}\stackrel{\approx}{\lra}\det D_{\Th_1;\t}, \\
\label{whcIcIprp_e4}
\cI_{\Th_2;\t}^{-1}\!\circ\!\cI_{\Th_1;\t}=(-1)^{N_1N_2}
\cI_{(\Th_1)_{\t};D_{\Th_2;\t}}\!\circ\!\cR_{\Th_1,\Th_2;\t}\!\circ\!
\cI_{(\Th_2)_{\t};D_{\Th_1;\t}}^{-1}
\!:\det D_{\Th_1;\t}\stackrel{\approx}{\lra} \det D_{\Th_2;\t}
\end{gather}
for all $\t\!\in\!\De_{\R}$.\\

\noindent
We call a bundle homomorphism as in~\eref{Thdfn_e} \sf{smoothly supported} 
if $\ze_{\t,v}\!\in\!\cE_{\t}^{0,1}(V)^{\vph}$ is smooth and 
$\supp(\ze_{\t,v})\!\subset\!\Si_{\t}^*$
for all $\t\!\in\!\De_{\R}$ and $v\!\in\!\R^N$.

\begin{prp}\label{LTtopol_prp}
For every continuous smoothly supported bundle homomorphism~$\Th$ as in~\eref{Thdfn_e},
\hbox{$U_{\Th}\!\subset\!\De_{\R}$} is an open subset and \eref{kerDThdfn_e}
is a vector bundle.
If in addition~$R$ is a continuous bundle homomorphism as in~\eref{Rbnddfb_e},
then $U_{\Th\circ R}\!\subset\!U_{\Th}$ and 
$$R_{\Th}\!: \ker D_{\Th\circ R}\lra\ker D_{\Th}\big|_{U_{\Th\circ R}}, \qquad
R_{\Th}(\t,\xi,v)=\big(\t,R_{\Th;\t}(\xi,v)\!\big),$$
is a continuous bundle map.
\end{prp}

\noindent
This proposition follows from Lemma~\ref{LTnorms_lmm},
as demonstrated by the gluing construction of \cite[Section~3]{LT} for 
$(J,\nu)$-holomorphic maps instead of bundle sections.
The greatly simplified, linear version of this construction
(without the quadratic term of the first equation in the proof of 
\cite[Proposition~3.4]{LT}) provides
local trivializations~for the projection
$$\cE(V)^{\vph}\!\times\!\R^N\supset\ker D_{\Th}\lra U_{\Th}$$
around every point $\t\!\in\!\De_{\R}$ and thus that $U_{\Th}\!\subset\!\De_{\R}$ is open.
This construction in the $N\!=\!0$ case and without restricting to the invariant sections
in carried out in \cite[Section~3.2]{g1cone}.
By the smooth support assumption on~$\Th$ (which is in line with the setup in~\cite{LT}), 
the reasoning in~\cite{g1cone} applies in the general case,
including for invariant sections, and implies the first statement of
the proposition.
The claim $U_{\Th\circ R}\!\subset\!U_{\Th}$ is immediate from the definitions.
Factoring $R$ through its graph reduces the remaining claim of the proposition
to the case that $R$ has constant rank.
This case in turn reduces to showing  that $\ker D_{\Th\circ R}$ is a subbundle of 
$\ker D_{\Th}|_{U_{\Th\circ R}}$ if $R$ is induced by the inclusion 
of a coordinate subspace of~$\R^N$.
This follows readily from the setup in \cite[Section~3.2]{g1cone}.\\

\noindent
By the first statement of Proposition~\ref{LTtopol_prp}, the total space of the projection 
$$\det D_{\Th}\big|_{U_{\Th}}\!=\!\La_{\R}^{\ind\,D_{(V,\vph)}+N}\!(\ker D_{\Th})
\lra U_{\Th}$$
is a real line bundle with a natural topology.
The isomorphisms~\eref{FDsum_e} associated with short exact sequences of vector spaces
as in~\eref{sesVdfn_e}
induce continuous isomorphisms of the same kind for short exact sequences of vector bundles.
Along with the second statement of Proposition~\ref{LTtopol_prp}, 
this implies that the bundle isomorphisms
\begin{alignat*}{2}
\wh\cI_{\Th_2;D_{\Th_1}}\!:
\det D_{\Th_1}&\lra \det D_{\Th_1\oplus\Th_2}\big|_{U_{\Th_1}}, &\quad
\wh\cI_{\Th_2;D_{\Th_1}}(\t,\vp)&=\big(\t,\wh\cI_{(\Th_2)_{\t};D_{\Th_1};\t}(\vp)\big),\\
\cR_{\Th_1,\Th_2}\!:
\det D_{\Th_1\oplus\Th_2}\big|_{U_{\Th_1\oplus\Th_2}}&\lra
\det D_{\Th_2\oplus\Th_1}\big|_{U_{\Th_1\oplus\Th_2}}, &\quad
\cR_{\Th_1,\Th_2}(\t,\vp)&=\big(\t,\cR_{\Th_1,\Th_2;\t}(\vp)\big),
\end{alignat*}
are continuous with respect to the natural topologies on the domains and targets
for all continuous compactly supported bundle homomorphisms~$\Th_1$ and~$\Th_2$ 
as in~\eref{Th12bnddfn_e}.
Combining this with~\eref{whcIcIprp_e3} and~\eref{whcIcIprp_e4}, 
we obtain the following.

\begin{crl}\label{detLBoverlap_crl}
For all continuous compactly supported bundle homomorphisms~$\Th_1$ and~$\Th_2$
as in~\eref{Th12bnddfn_e}, the bundle map 
$$\wh\cI_{\Th_1\Th_2}\!:\det D_{\Th_2}\big|_{U_{\Th_1}\cap U_{\Th_2}}\lra
\det D_{\Th_1}\big|_{U_{\Th_1}\cap U_{\Th_2}}, \quad
\wh\cI_{\Th_1\Th_2}(\t,\vp)=\big(\t,
\wh\cI_{\Th_1;\t}\big(\wh\cI_{\Th_2;\t}^{-1}(\vp)\big)\!\big),$$
is continuous with respect to the natural topologies on 
its domain and~target.
\end{crl}

\noindent
We now topologize the total space of the projection~\eref{detLB_e}.
Let $\t^*\!\in\!\De_{\R}$.
By the elliptic regularity of~$D_{\t^*}$, there exists a homomorphism
\BE{Thtdfn_e}\Th_{\t^*}\!: \R^N\lra \cE_{\t^*}^{0,1}(V)^{\vph},\qquad
\Th_{\t^*}(v)=\ze_{\t^*,v},\EE
such that every $(0,1)$-form $\ze_{\t^*,v}$ is smooth and supported in~$\Si_{\t^*}^*$ and 
the operator 
\BE{Dtstardfn_e}\big(D_{\t^*}\big)_{\!\Th_{\t^*}}\!: 
\cE_{\t^*}(V)^{\vph}\!\oplus\!\R^N\lra\cE_{\t^*}^{0,1}(V)^{\vph},
\qquad \big(D_{\t^*}\big)_{\!\Th_{\t^*}}(\xi,v)=D_{\t^*}\xi\!+\!\ze_{\t^*,v},\EE
is surjective.
Choose  a continuous compactly supported homomorphism~$\Th$ as 
in~\eref{Thdfn_e} which restricts to~\eref{Thtdfn_e} over~$\t^*$.
By~\eref{Dtstardfn_e}, $\t^*\!\in\!U_{\Th}$.
We topologize $\det D_{(V,\vph)}|_{U_{\Th}}$ by requiring that the bundle map 
\BE{cIThD_e2}\wh\cI_{\Th}\!: \det D_{(V,\vph)}\big|_{U_{\Th}}\lra \det D_{\Th}\big|_{U_{\Th}}, 
\qquad \wh\cI_{\Th}(\t,\vp)=\big(\t,\wh\cI_{\Th;\t}(\vp)\big),\EE
be a homeomorphism with respect to the natural topology on its target.
By Corollary~\ref{detLBoverlap_crl}, the overlaps between these maps are continuous.
Thus, these maps define a topology on the total space of the projection~\eref{detLB_e}.\\

\noindent
It is immediate from the construction that the resulting topologies on the determinant
line bundle corresponding to different real bundle pairs~$(V,\vph)$ satisfy~\ref{dNatI_it}
on page~\pageref{dNatI_it}.
By the proof of \cite[Corollary~5.4]{detLB}, these topologies also satisfy~\ref{dses_it}.

\begin{rmk}\label{detLB_rmk}
Other topologies on the total space of the projection~\eref{detLB_e}
with good properties can be obtained by modifying the isomorphisms~\eref{sum}
associated with exact sequences~\eref{ETles_e} as described above Theorem~2 
in \cite[Section~3.2]{detLB}.
This would modify the topologizing maps~\eref{cIThD_e2} and would thus generally change
the topology on the total space of~\eref{detLB_e}.
The two topologies would differ by a homeomorphism which restricts to the identity
over the points $\t\!\in\!\De_{\R}$ such that $D_{\t}$ is surjective. 
\end{rmk}

\begin{rmk}\label{kerDTh_rmk}
A connection $\na$ as above~\eref{detLB_e} induces a splitting 
$$TV\approx \pi^*V\oplus \pi^*T\cU\,.$$
The complex structure $\fI$ in the fibers of $\pi\!:V\!\lra\!\cU$, 
the complex structure~$\fJ_{\cU}$ on~$\cU$, and 
the zeroth-order deformation term~\eref{AcUdfn_e} induce a complex structure~$J_V$
on the total space of~$V$ by
$$ J_V\big|_v\big(\dot{v},\dot{x}\big)=
\big(\fI\dot{v}\!+\!\{Av\}(\dot{x}),\fJ_{\cU}\dot{x}\big).$$
For each $\t\!\in\!\De_{\R}$, $D_{\t}$ is then the $\dbar_{J_V}$-operator on
the space of real maps from $(\Si_{\t},\fJ_{\t})$ to the total space of~$V$.
In particular, $\ker D_{\t}$ consists of real $(J_V,\fJ_{\t})$-holomorphic maps.
By the smooth support assumption on~$\Th$, 
the subspace topology on $\ker D_{\Th}$ can thus be described in terms of convergence
of sequences similarly to Definition~\ref{GromConv_dfn2}.
\end{rmk}

\vspace{.2in}

\noindent
{\it Institut de Math\'ematiques de Jussieu - Paris Rive Gauche,
Universit\'e Pierre et Marie Curie, 
4~Place Jussieu,
75252 Paris Cedex 5,
France\\
penka.georgieva@imj-prg.fr}\\

\noindent
{\it Department of Mathematics, Stony Brook University, Stony Brook, NY 11794\\
azinger@math.stonybrook.edu}



\begin{thebibliography}{99}

\bibitem{ACH} E.~Arbarello, M.~Cornalba, and P.~Griffiths,
\emph{Geometry of Algebraic Curves,~II},
Comprehensive Studies in Mathematics~268, Springer~2011

\bibitem{AB} M.~Atiyah and R.~Bott, 
\emph{The moment map and equivariant cohomology}, Topology~23 (1984), 1--28

\bibitem{BHH} I.~Biswas, J.~Huisman, and J.~Hurtubise, 
\emph{The moduli space of stable vector bundles over a real algebraic curve}, 
Math.~Ann.~347 (2010), no.~1, 201--233

\bibitem{Cho} C.-H.~Cho, 
\emph{Counting real J-holomorphic discs and spheres in dimension four and six}, 
J.~Korean Math.~Soc.~45 (2008), no.~5, 1427--1442

\bibitem{Remi0} R.~Cr\'etois,
\emph{Automorphismes r\'eels d'un fibr\'e et op\'erateurs de Cauchy-Riemann},
Math.~Z.~275 (2013), no.~1-2, 453--497 

\bibitem{Remi} R.~Cr\'etois,
\emph{D\'eterminant des op\'erateurs de Cauchy-Riemann r\'eels et application \'a 
l'orientabilit\'e d'espaces de modules de courbes r\'eelles}, math/1207.4771


\bibitem{Teh} M.~Farajzadeh Tehrani,
\emph{Counting genus zero real curves in symplectic manifolds},
Geom.~Top.~20 (2016), no.~2, 629–-695

\bibitem{Teh2} M.~Farajzadeh Tehrani,
\emph{Notes on genus one real Gromov-Witten invariants},
math/1406.3786

\bibitem{FO} K.~Fukaya and K.~Ono,
\emph{Arnold Conjecture and Gromov-Witten invariant},
Topology 38 (1999), no.~5, 933--1048

\bibitem{FOOO}
K.~Fukaya, Y.-G.~Oh, H.~Ohta, and K.~Ono,
\emph{Lagrangian Intersection Theory: Anomaly and Obstruction},
AMS Studies in Advanced Mathematics~46, 2009

\bibitem{FOOO9}
K.~Fukaya, Y.-G.~Oh, H.~Ohta, and K.~Ono,
\emph{Anti-symplectic involution and Floer cohomology},
Geom.~Topol.~21 (2017), no.~1, 1–-106

\bibitem{growi} A.~Gathmann, \emph{GROWI},
available on the author's website 

\bibitem{Ge2} P.~Georgieva,
\emph{Open Gromov-Witten invariants in the presence of an anti-symplectic involution},
Adv.~Math.~301 (2016), 116–-160

\bibitem{XCapsSetup} P.~Georgieva and A.~Zinger,
\emph{The moduli space of maps with crosscaps: Fredholm theory 
and orientability},  Comm.~Anal.~Geom.~23 (2015), no.~3, 81--140

\bibitem{XCapsSigns} P.~Georgieva and A.~Zinger,
\emph{The moduli space of maps with crosscaps: 
the relative signs of the natural automorphisms},
J.~Symplectic Geom.~14 (2016), no.~2, 359--430

\bibitem{RealEnum} P.~Georgieva and A.~Zinger,
\emph{Enumeration of real curves in $\C P^{2n-1}$ 
and a WDVV relation for real Gromov-Witten invariants}, 
Duke Math.~J.~166 (2017), no.~17, 3291–-3347

\bibitem{RealGWsII} P.~Georgieva and A.~Zinger,
\emph{Real Gromov-Witten theory in all genera and real enumerative geometry: properties},
math/1507.06633v4

\bibitem{RealGWsIII} P.~Georgieva and A.~Zinger,
\emph{Real Gromov-Witten theory in all genera and real enumerative geometry: computation},
math/1510.07568

\bibitem{RealGWsApp} P.~Georgieva and A.~Zinger,
\emph{Real Gromov-Witten theory in all genera and real enumerative geometry: appendix},
available from the authors' websites

\bibitem{RBP} P.~Georgieva and A.~Zinger,
\emph{On the topology of real bundle pairs over nodal symmetric surfaces}, 
Topology Appl.~214 (2016), 109–-126

\bibitem{GP} T.~Graber and R.~Pandharipande, \emph{Localization of virtual classes},
Invent.~Math.~135 (1999), no.~2, 487--518

\bibitem{GH}  P.~Griffiths and J.~Harris,
\emph{Principles of Algebraic Geometry}, Wiley, 1994

\bibitem{Gr} M.~Gromov, \emph{Pseudoholomorphic curves in symplectic manifolds},  
Invent.~Math.~82 (1985), no~ 2, 307--347 

\bibitem{Huang} Y.-Z.~Huang, 
\emph{Two-Dimensional Conformal Geometry and Vertex 
Operator Algebras}, Progress in Math.~148, Birkh\"auser 1997

\bibitem{KM} F.~Knudsen and D.~Mumford, 
\emph{The projectivity of the moduli space of stable curves, I: 
Preliminaries on ``det" and ``Div"}, Math.~Scand.~39 (1976), no.~1, 19--55

\bibitem{Kollar} J.~Koll\'ar,
\emph{Examples of vanishing Gromov--Witten--Welschinger invariants},
 J.~Math.~Sci. Univ.~Tokyo 22 (2015), no.~1, 261–-278 

\bibitem{LT}  J.~Li and G.~Tian, 
\emph{Virtual moduli cycles and Gromov-Witten invariants of general symplectic manifolds}, 
Topics in Symplectic \hbox{$4$-Manifolds},
47-83, First Int.~Press Lect.~Ser., I, Internat.~Press, 1998

\bibitem{Melissa} C.-C.~Liu,
\emph{Moduli of $J$-holomorphic curves with Lagrangian boundary condition and open
Gromov-Witten invariants for an $S^1$-pair},
math/0210257v2


\bibitem{McSa94} D.~McDuff and D.~Salamon, 
\emph{$J$-Holomorphic Curves and Quantum Cohomology},
University Lecture Series~6, AMS,~1994

\bibitem{MS} D. McDuff and D.~Salamon, 
\emph{J-holomorphic Curves and Symplectic Topology},
Colloquium Publications~52, AMS, 2012

\bibitem{Nat} S.~Natanzon, 
\emph{Moduli of real algebraic curves and their superanalogues: 
spinors and Jacobians of real curves}, Russian Math.~Surveys 54 (1999), no.~6, 1091--1147

\bibitem{NZ} J.~Niu and A.~Zinger,
\emph{Lower bounds for the enumerative geometry of positive-genus real curves},
math/1511.02206

\bibitem{NZapp} J.~Niu and A.~Zinger,
\emph{Lower bounds for the enumerative geometry of positive-genus real curves, appendix},
available from the authors' websites

\bibitem{PSW} R.~Pandharipande, J.~Solomon, and J.~Walcher,
\emph{Disk enumeration on the quintic 3-fold}, 
J.~Amer.~Math.~Soc. 21 (2008), no.~4, 1169--1209

\bibitem{RT} Y.~Ruan and G.~Tian, \emph{A mathematical theory of quantum cohomology}, 
J.~Differential Geom.~42 (1995), no.~2, 259--367

\bibitem{RT2} Y.~Ruan and G.~Tian,
\emph{Higher genus symplectic invariants and sigma models coupled with gravity},
Invent.~Math.~130 (1997), no.~3, 455--516


\bibitem{Sol} J.~Solomon,  
\emph{Intersection theory on the moduli space of holomorphic curves with 
Lagrangian boundary conditions}, math/0606429

\bibitem{Wal} J.~Walcher, \emph{Evidence for tadpole cancellation in the
topological string}, Comm.~Number Theory Phys.~3 (2009), no.~1, 111--172

\bibitem{Wel4}  J.-Y.~Welschinger,  
\emph{Invariants of real symplectic 4-manifolds and lower bounds in real enumerative geometry},
Invent.~Math.~162 (2005), no.~1, 195--234

\bibitem{Wel6} J.-Y.~Welschinger,  
\emph{Spinor states of real rational curves in real algebraic convex 3-manifolds 
and enumerative invariants},
Duke Math.~J.~127 (2005), no.~1, 89–-121

\bibitem{gluing} A.~Zinger, 
\emph{Enumerative vs.~symplectic invariants and obstruction bundles},
J.~Symplectic Geom.~2 (2004), no.~4, 445--543

\bibitem{g1cone} A.~Zinger, 
\emph{On the structure of certain natural cones over moduli spaces of genus-one holomorphic maps}
Adv.~Math.~214 (2007), no.~2, 878–-933

\bibitem{pseudo} A.~Zinger, 
\emph{Pseudocycles and integral homology}, 
Trans.~AMS 360 (2008), no.~5, 2741–-2765

\bibitem{g1diff} A.~Zinger, 
\emph{Standard vs.~reduced genus-one Gromov-Witten invariants},
Geom.\,Topol.\,12 (2008), no.~2, 1203--1241

\bibitem{g1comp} A.~Zinger,
\emph{A sharp compactness theorem for genus-one pseudo-holomorphic maps},
Geom.\,Topol.\,13 (2009), no.~5, 2427--2522

\bibitem{g1comp2} A.~Zinger,  \emph{Reduced genus-one Gromov-Witten invariants},
J.~Differential Geom.~83 (2009), no.~2, 407--460

\bibitem{anal} A.~Zinger, \emph{Basic Riemannian geometry and Sobolev estimates
used in symplectic topology}, math/1012.3980

\bibitem{FanoGV} A.~Zinger, 
\emph{A comparison theorem for Gromov-Witten invariants in the symplectic category},
Adv.~Math.~228 (2011), no.~1, 535--574

\bibitem{detLB} A.~Zinger, 
\emph{The determinant line bundle for Fredholm operators: construction, 
properties, and classification},  Math.~Scand.~118 (2016), no.~2, 203--268

\bibitem{RealRT} A.~Zinger,
\emph{Real Ruan-Tian perturbations}, math/1701.01420

\end{thebibliography}
\end{document}